\documentclass[a4paper,11pt]{amsart}

\usepackage{amsmath}
\usepackage{amsthm}
\usepackage{amssymb}
\usepackage{amsbsy}
\usepackage{amsfonts}
\usepackage{amstext}
\usepackage{amscd}
\usepackage{tikz}
\usepackage{graphicx}
\usepackage{verbatim}
\usepackage{bm}
\usepackage{commath,mathtools,setspace}
\usepackage{paralist}
\usepackage{extarrows}
\usepackage{color}
\usepackage{stmaryrd}

\usepackage[T1]{fontenc}
\usepackage{graphicx}
\usepackage{blindtext} 
\usepackage{geometry}
\usepackage[colorlinks=true,linkcolor=blue,citecolor=blue]{hyperref}
\usepackage[colorinlistoftodos,prependcaption,textsize=small]{todonotes}

\definecolor{red}{RGB}{255,0,0}
\definecolor{green}{RGB}{0,100,0}
\definecolor{blue}{RGB}{0,0,255}

\numberwithin{equation}{section}

\newtheorem{theorem}{Theorem}[section]  
\newtheorem{lemma}[theorem]{Lemma}

\newtheorem{corollary}[theorem]{Corollary}
\newtheorem{proposition}[theorem]{Proposition}
\newtheorem{notation}[theorem]{Notation}

\newtheorem{remark}[theorem]{Remark}
\newtheorem{definition}[theorem]{Definition}
\newtheorem{example}[theorem]{Example}

\newcommand{\weak}{\xrightarrow{w}}

\newcommand{\C}{\mathbb{C}}

\newcommand{\N}{\mathbb{N}}
\newcommand{\R}{\mathbb{R}}

\newcommand{\dx}{{\rm d} x}
\newcommand{\dt}{{\rm d} t}

\newcommand{\polplus}{\mathcal P_d(\mathbb R_{\geq 0})}
\newcommand{\pols}{\mathcal P}
\DeclareMathOperator{\Var}{Var}

\newcommand{\cc}{\mathbb{C}}

\newcommand{\nn}{\mathbb{N}}

\newcommand{\pp}{\mathbb{P}}

\newcommand{\rr}{\mathbb{R}}

\renewcommand{\tt}{\mathbb{T}}

\newcommand{\MM}{\mathcal{M}}

\newcommand{\PP}{\mathcal{P}}

\newcommand{\falling}[2]{\left(#1\right)_{#2}}
\newmuskip\pFqmuskip

\newcommand*\pFqN[6][8]{%
  \begingroup 
  \pFqmuskip=#1mu\relax
  \mathcode`\,=\string"8000
  \begingroup\lccode`\~=`\,
  \lowercase{\endgroup\let~}\pFqcomma
  {}_{#2}F_{#3}{\left(\genfrac..{0pt}{}{#4}{#5};#6\right)}%
  \endgroup
}
\newcommand{\pFqcomma}{\mskip\pFqmuskip}

\newcommand*\HGF[5]{%
 {\ }_{#1} F_{#2}{\left(\genfrac..{0pt}{}{#3}{#4};#5\right)}%
}

\newcommand*\HGP[3]{%
 \mathcal{H}_{#1}{\left[\genfrac..{0pt}{1}{#2}{#3}\right]}%
}

\newcommand{\dil}[1]{{\rm Dil}_{#1}}  
\newcommand{\meas}[1]{\mu \left\llbracket #1 \right\rrbracket}

\newcommand{\coef}[2]{\widetilde{\mathsf{e}}_{#1}^{(#2)}} 

\newcommand{\ffc}[2]{\kappa_{#1}^{(#2)}} 

\newcommand{\freec}[2]{{\bm r}_{#1}\left(#2\right)} 

\newcommand{\diff}[2]{\partial_{#1|#2}\, }  

\newcommand{\strans}[2]{S_{#1}^{(#2)}}   

\newcommand{\stranstilde}[2]{\widetilde{S}_{#1}^{(#2)}}   

\newcommand{\ttrans}[2]{T_{#1}^{(#2)}}   

\newcommand{\cdf}[1]{F_{#1}}   

\newcommand{\blocks}[1]{|#1|}

\newcommand{\partlat}{P}

\newcommand{\shift}[1]{{\rm Shi}_{#1}}

\newcommand{\reversed}{{\langle -1\rangle}}

\newcommand{\mfp}{\mathfrak{p}}

\usepackage{stackengine,scalerel}
\newcommand{\boxlor}{\mathop{\ThisStyle{%
  \ensurestackMath{\stackinset{c}{-.05\LMpt}{c}{-.2\LMpt}{\SavedStyle\lor}{\SavedStyle\square}}}}}

\usepackage[backend=biber,style=alphabetic]{biblatex} 
\title{$S$-transform in finite free probability}
\author{Octavio Arizmendi, Katsunori Fujie, Daniel Perales and Yuki Ueda }

\address{
Octavio Arizmendi:
Centro de Investigaci\'{o}n en Matem\'{a}ticas, Guanajuato, Mexico}
\email{octavius@cimat.mx}

\address{
Katsunori Fujie:
Department of Mathematics, Kyoto University, Kyoto, Japan}
\email{fujie.katsunori.42m@st.kyoto-u.ac.jp}

\address{
Daniel Perales:
Department of Mathematics, Texas A\&M University, TX, USA}
\email{daniel.perales@tamu.edu}

\address{
Yuki Ueda:
Department of Mathematics, Hokkaido University of Education, Hokkaido, Japan}
\email{ueda.yuki@a.hokkyodai.ac.jp}

\date{\today}

\begin{document}

\begin{abstract}
We characterize the limiting root distribution $\mu$ of a sequence of polynomials $\{p_d\}_{d=1}^\infty$ with nonnegative roots and degree $d$, in terms of their coefficients. Specifically, we relate the asymptotic behaviour of the ratio of consecutive coefficients of $p_d$ to Voiculescu's $S$-transform $S_\mu$ of $\mu$.

In the framework of finite free probability, we interpret these ratios of coefficients as a new notion of finite $S$-transform, which converges to $S_\mu$ in the large $d$ limit. It also satisfies several analogous properties to those of the $S$-transform in free probability, including multiplicativity and monotonicity.

The proof of the main theorem is based on various ideas and new results relating finite free probability and free probability. In particular, we provide a simplified explanation of why free fractional convolution corresponds to the differentiation of polynomials, by finding how the finite free cumulants of a polynomial behave under differentiation.

This new insight has several applications that strengthen the connection between free and finite free probability. Most notably, we generalize the approximation of $\boxtimes_d$ to $\boxtimes$ and prove a finite approximation of the Tucci--Haagerup--Möller limit theorem in free probability, conjectured by two of the authors. We also provide finite analogues of the free multiplicative Poisson law, the free max-convolution powers and some free stable laws.
\end{abstract}

\maketitle

\tableofcontents

\section{Introduction}
\label{sec:intro}

The main object of interest in this paper is the finite free multiplicative convolution of polynomials. This is a binary operation on the set $\pols_d$ of polynomials of degree $d$ which can be defined explicitly in terms of the coefficients of the polynomials involved. For polynomials 
\begin{equation} 
\label{eq:polynomial.coefficients}
p(x)= \sum_{k=0}^d  (-1)^k \binom{d}{k} \coef{k}{d}(p) x^{d-k} \qquad \text{and} \qquad q(x)=\sum_{k=0}^d  (-1)^k \binom{d}{k} \coef{k}{d}(q) x^{d-k},
\end{equation}
 the finite free multiplicative convolution of $p$ and $q$ is defined as
\begin{equation} 
\label{eq:finitefreemult1} [p\boxtimes_d q](x)= \sum_{k=0}^d  (-1)^k \binom{d}{k} \coef{k}{d}(p)\coef{k}{d}(q) x^{d-k}.\end{equation}

This operation is bilinear, commutative, and associative and more importantly it is closed in the set $\polplus$ of polynomials with all non-negative real roots.
Moreover, it is known that the maximum root of $p\boxtimes_d q$ is bounded above by the product of the maximum roots of $p$ and $q$.
All these basic properties have been well understood since more than a century ago, see for instance \cite{walsh1922location}. 

Recently, this convolution was rediscovered in the study of characteristic polynomials of certain random matrices \cite{MSS}.
To elaborate, for $d$-dimensional positive semidefinite matrices, $A$ and $B$, with characteristic polynomials $p$ and $q$, respectively, their finite free multiplicative convolution, $p\boxtimes_d q$, is the expected characteristic polynomial of $A U B U^{*}$, where $U$ is a random unitary matrix sampled according to the Haar measure on the unitary group of degree $d$.
One motivation behind this new interpretation in terms of random matrices was to derive new polynomial inequalities for the roots of the convolution. Marcus, Spielman, and Srivastava achieved a stronger bound on the maximum root using tools from free probability, in particular from Voiculescu's $S$-transform.
In the light of this, Marcus pursued this connection further  \cite{Mar21} by suggesting and providing evidence that, as the degree $d$ tends to infinity, the finite free convolution of polynomials, $\boxtimes_d$, tends to free multiplicative convolution of measures, $\boxtimes$, thus initiating the theory of finite free probability.
This was rigorously proved in \cite{AGP}. 

The study of polynomial convolutions from the perspective of finite free probability has strengthened the connections between geometry of polynomials, random matrices, random polynomials and free probability theory, and has given new insights into these topics.
In particular, after the original paper of Marcus, Spielman and Srivastava \cite{MSS} and the work of Marcus \cite{Mar21}, various papers have investigated limit theorems for important sequences of polynomials or their finite free convolutions, see \cite{fujie2023law, AP, AGP,mfmp2024hypergeometric,kabluchko2022lee,hoskins2020dynamics, kabluchko2021repeated, arizmendifujieueda}.  

In a certain sense, the present paper continues this line of research, understanding the relation of finite free probability with free probability in the large $d$ limit.

One of the motivations of this work is to give a concrete understanding of how the limiting behaviour of the coefficients of a polynomial is connected with convergence in distribution to a probability measure. More specifically, 
consider a sequence of polynomials $p_d$ of degree $d$
such that the empirical root distribution of the polynomials, denoted by $\meas{p_d}$, converges weakly to a measure $\mu$ as the degree $d$ tends to infinity and assume that $p_d$ has coefficients given as in \eqref{eq:polynomial.coefficients}.
Then a natural question is what happens with the coefficients $\widetilde{\mathsf{e}}_k(p_d)$ in the limit, or vice versa, to give conditions on the behaviour of the coefficients $\widetilde{\mathsf{e}}_{k}(p_d)$ so that $\meas{p_d}$ converges weakly to $\mu$.
The naive approach of fixing $k$ yields a trivial limit, see Corollary \ref{cor:lim_coef}.

From their studies of the law of large numbers, Fujie and Ueda \cite{fujie2023law} observed that it may be more fruitful to look at the ratio of two consecutive coefficients, see Section \ref{ssec:intro.LLN} below. Our main result says that this is the case. We find that one can get a meaningful limit when considering the ratio $ \widetilde{\mathsf{e}}_k(p_d)/ \widetilde{\mathsf{e}}_{k+1}(p_d)$ and doing a diagonal limit, i.e. letting $k$ and $d$ tend to infinity with $k/d$ approaching some constant $t$. Furthermore, such a limit can be tracked down and has precise relation to Voiculescu's $S$-transform.

\subsection{\texorpdfstring{$S$}\ -transform}

Recall that the main analytic tool to study the free multiplicative convolution is the \textit{$S$-transform}, introduced by Voiculescu \cite{voiculescu1987multiplication} and studied in detail by Bercovici and Voiculescu \cite{BerVoi1992, BerVoi1993}. Its importance in free probability stems from the fact that $S_\mu$ characterizes the measure $\mu$ and is multiplicative with respect to free multiplicative convolution:
\[
S_{\mu\boxtimes \nu} = S_\mu S_\nu \qquad \text{for }\mu,\nu \in \MM(\R_{\geq 0})\setminus\{\delta_0\}.
\] 
We refer the reader to Section \ref{ssec:prelim.S.transform} for more details.

Our results can be nicely presented if we introduce a new \textit{finite free $S$-transform}\footnote{Let us mention that Marcus already defined a modified $S$-transform in \cite{Mar21} that tends to the inverse of Voiculescu's $S$-transform. Our definition is different, and the relationd with Marcus' transform is not clear. We will address these in more detail in Remarks \ref{rem:finite.Strans} and \ref{rem:Marcus}.}. Consider again a polynomial $p$ as in \eqref{eq:polynomial.coefficients} and, for simplicity, let us assume that all the roots of $p$ are strictly positive so that the coefficients $\widetilde{\mathsf{e}}_k(p)$ are non-zero.
In this case the {\it finite $S$-transform} of $p$ is the map $S_p^{(d)}:  \left\{-k/d \mid k =1,2,\dots, d \right\} \to \R_{> 0}$ given by \[\strans{p}{d}\left(-\frac{k}{d}\right) := \frac{\coef{k-1}{d}(p)}{\coef{k}{d}(p)} \qquad \text{ for } \quad k=1,2, \dots, d.
\]

It is straightforward from \eqref{eq:finitefreemult1} that $S^{(d)}_{p\boxtimes_dq} = S^{(d)}_pS^{(d)}_q$.
Besides the multiplicative property, this map satisfies many other properties analogous to Voiculescu's $S$-transform, such as monotonicity, same image, same range and a formula for the reversed polynomial, see Section \ref{sec:properties_of_Stransform}.

More importantly, in the limit as the dimension $d$ goes to infinity, the convergence of the empirical root distribution of a sequence of polynomials $p_d$ to some measure $\mu$ is equivalent to the convergence of the finite $S$-transform of $p_d$ to the $S$-transform of $\mu$.
This is the content of our first main result. 

\begin{theorem}
\label{thm:main}
Let $(p_d)_{d\in \N}$ be a sequence of polynomials with $p_d\in \pols_d(\R_{\geq 0})$ and $\mu\in \MM(\R_{\geq 0})\backslash \{\delta_0\}$. The following are equivalent:
\begin{enumerate}
\item The weak convergence: $\meas{p_d} \weak \mu$. 
\item For every $t\in (0,1-\mu(\{0\}))$ and every sequence $(k_d)_{d\in\nn}$ with $1\leq k_d \leq d$ and $\lim_{d\to\infty} \frac{k_d}{d}=t$, one has 
\begin{equation}
\label{eq:intro.limit.S}
\lim_{\substack{d\to \infty}}\strans{p_d}{d}\left(-\frac{k_d}{d}\right)= S_\mu(-t).
\end{equation}
\end{enumerate}
\end{theorem}

The details of the proof are given in Section \ref{sec:finite.S.to.S}.
Similarly, we can also define a finite symmetric $S$-transform $\stranstilde{p}{2d}$ of a symmetric polynomial $p$ whose all the roots are non-zero as follows:
\[\stranstilde{p}{2d}\left(-\frac{k}{d}\right):=
\sqrt{\frac{\coef{2(k-1)}{2d}(p)}{\coef{2k}{2d}(p)}} \qquad \text{for }k=1,\dots, d,\]
and prove an analogous result, see Section \ref{sec:symmetric} for more details.

To show the effectiveness of the above result, let us give a few examples of well-known polynomials and their limiting distributions. Further related examples will be give in Section \ref{sec:examples.applications}

\begin{itemize}

\item (Laguerre polynomials) Let $L_d^{(b)}(x)=\sum_{k=0}^d (-1)^k \binom{d}{k} \frac{(bd)_k}{d^k}x^{d-k}$ be the normalized Laguerre polynomials of degree $d$ and with parameter $b\ge1$, where $(a)_0:=1$ and $(a)_k:=a(a-1)\cdots(a-k+1)$ for $a\in \R$ and $k\ge1$. Then its finite $S$-transform satisfies:
\[\strans{L_d^{(b)}}{d}\left(-\frac{k}{d}\right)= \frac{1}{b-\frac{k-1}{d}} \to \frac{1}{b-t} = S_{{\bf MP}_b}(-t) \qquad  \text{as}\quad \frac{k}{d}\to t,
\]
where ${\bf MP}_b$ is the Marchenko-Pastur distribution with parameter $b$. 
This gives an alternative proof of well-known fact: $\meas{L_{d}^{(b)}} \weak {\bf MP}_b$. For the notation and detail, see Example \ref{exm:Laguerre}.
\item (Hermite polynomials) Let $H_{2d}(x)= \sum_{k=0}^d (-1)^k\binom{2d}{2k} \frac{(2k)!}{k!(4d)^k} x^{2d-2k}$ be the normalized Hermite polynomials of degree $2d$. We have the convergence of finite symmetric $S$-transform:
\[\stranstilde{H_{2d}}{2d}\left(-\frac{k}{d}\right)= \frac{1}{\sqrt{-\tfrac{k}{d} +\tfrac{1}{2d}}}\to \frac{1}{\sqrt{-t}} = \widetilde{S}_{\mu_{\mathrm{sc}}}(-t)\qquad  \text{as}\quad \frac{k}{d}\to t,  \] 
where $\mu_{\mathrm{sc}}$ is the standard semicircular distribution.
This gives an alternative proof of well-known fact: $\meas{H_{2d}} \weak \mu_{\mathrm{sc}}$, see Example \ref{Example:Hermite}.

\item The (normalized) Chebyshev polynomials of the first kind may be written as 
$T_{2d}(x)=\sum_{k=0}^{d}  \binom{2d}{2k} (-1)^k\frac{(2k)_k}{(2d-1)_{k}2^{2k}}x^{2d-2k}.$
Then we have \[\stranstilde{T_{2d}}{2d}\left(-\frac{k}{d}\right)= \sqrt{\frac{2(k-2d)}{2k-1}}\to \sqrt{\frac{t-2}{t}}=S_{\mu_{\mathrm{arc}}}(-t)\qquad  \text{as}\quad \frac{k}{d}\to t,  \] 
where $\mu_{\mathrm{arc}}$ is the arcsine law on $(-1,1)$. See \cite{arizmendi2013class} for the $S$-transform of $\mu_{\mathrm{arc}}$.




\item The (normalized) Chebyshev polynomials of the second kind can be expressed as $U_{2d}(x)=\sum_{k=0}^d \binom{2d}{2k}(-1)^k\frac{(2k)_k}{(2d)_k2^{2k}}x^{2d-2k}$. Then we obtain \[\stranstilde{U_{2d}}{2d}\left(-\frac{k}{d}\right)=\sqrt{\frac{2(k-1-2d)}{2k-1}} \to \sqrt{\frac{t-2}{t}}=S_{\mu_{\mathrm{arc}}}(-t)\qquad  \text{as}\quad \frac{k}{d}\to t.\]

\end{itemize}

\subsection{Quick overview of the proof} 
\label{ssec:intro.overview.proof}

While the proof of our main theorem is technical for the general case, it relies on simple but useful results that provide a deeper understanding of the relation between finite free convolution and differentiation.

As we prove in Section \ref{Sec:Derivatives}, the finite free cumulants of the derivatives of a polynomial $p$ can be directly related to the finite free cumulants of $p$. That is, by using the notation
\[
\diff{j}{d} p := \frac{p^{(d-j)}}{\falling{d}{d-j}},
\]
we have the relation
\[
\ffc{n}{j}(\diff{j}{d}p) = \left( \frac{j}{d}\right)^{n-1} \ffc{n}{d}(p)\quad \text{ for } 1 \leq n \leq j \leq d. 
\]

Using the fact that convergence of finite free cumulants implies weak convergence, see \cite{AP}, the equation above allows us to give a new proof of the following result, relating derivatives to free convolution powers.

\begin{theorem}[\cite{hoskins2020dynamics}, \cite{AGP}] \label{Thm:AGVP_HK_S}
  Fix $K$ a compact subset on $\R$.
  Let $\mu \in \MM(\rr)$  and $(p_d)_{d\in\N}$ be a sequence of polynomials of degree $d$ such that every $p_d\in \PP_{d}(K)$ has only roots in $K$ and $\meas{p_d} \weak \mu$.
  Then, for a parameter $t \in (0, 1)$ and a sequence of integers $(j_d)_{d\in \N}$ such that $1\leq j_d \leq d$ and $\lim_{d\to\infty}\frac{j_d}{d}=t$, we have $\meas{\diff{j_d}{d} p_d} \weak \dil{t}(\mu^{\boxplus 1/t})$ as $d\to\infty$.
\end{theorem}

It is worth mentioning that if in the previous result we allow $t=0$ or $t=1$, we can also draw some conclusions about the limiting distribution, see Remark \ref{rem:added}.  Also, in Theorem \ref{thm:general.AGVP} of Section \ref{ssec:generalcase} we will show that the assertion of Theorem \ref{Thm:AGVP_HK_S} still holds if we drop the assumption on the uniform compactness.

Once Theorem \ref{Thm:AGVP_HK_S} is settled, the last ingredients that connect the finite $S$-transform with the $S$-transform are two observations.
First, the $S$-transform of a measure contained on $[a,b]$ for $a>0$ is related with the value of the Cauchy transform at $0$ of the free convolution powers:
\begin{equation*}
    G_{\dil{t}{(\mu^{\boxplus 1/t})}}(0) = -S_\mu(-t) \qquad \text{(see Lemma \ref{lem:approx_box_times})}.
\end{equation*}
Second, for the finite $S$-transform, a similar relation holds but in this case for derivatives of a polynomial with strictly positive roots, namely,
\begin{equation*}
    G_{\meas{\diff{k}{d}p_d}}(0)=-\strans{p_d}{d}\left(-\frac{k}{d}\right) \qquad \text{(see Lemma \ref{lem:Strans_conditions} (2))}.
\end{equation*}

Then, since weak convergence implies the convergence of the Cauchy transform on compact sets outside the support. From the convergence of $\meas{\diff{k}{d}p_d} \to \dil{t}(\mu^{\boxplus 1/t})$ we may conclude that 
\[
\lim_{\substack{d\to \infty \\ k/d\to t}} \strans{p_d}{d}\left(-\frac{k}{d}\right) = S_\mu(-t).
\]

However, when attempting to upgrade these considerations to the case when the support of $\mu$ includes $0$ or it is unbounded, one faces technical difficulties such as $G_{\meas{\diff{k}{d}p_d}}(0)$ or $G_{\dil{t}{(\mu^{\boxplus 1/t})}}(0)$ possibly being undefined. To get around this, we use uniform bounds on the roots of polynomials after repeated differentiation.  One important tool that we use is a partial order on polynomials, see Section \ref{sec:partial.order}.

\subsection{Relation to Multiplicative Law of Large Numbers}
\label{ssec:intro.LLN}

Notice that the existence of the limit \eqref{eq:intro.limit.S} amounts to the fact that the ratio of two consecutive coefficients approaches some limit. 

The intuition of why this limit should exist comes from the law of large numbers for the free multiplicative convolution of Tucci, and Haagerup and M\"{o}ller, as well as its finite counterpart of Fujie and Ueda \cite{fujie2023law}.  

Recall that Tucci \cite{tucci2010}, Haagerup and M\"{o}ller \cite[Theorem 2]{HM13} proved that, for every $\mu\in \MM(\R_{\geq 0})$, there exists a  measure $\Phi(\mu) \in \MM(\R_{\geq 0})$ such that
\[
(\mu^{\boxtimes n})^{\langle1/n\rangle} \weak \Phi(\mu) \qquad \text{as }n\to \infty,
\]
where for $\nu\in \MM(\R_{\geq 0})$ and $c>0$ the measure $\nu^{\langle c\rangle}$ denotes the pushforward of $\nu$ by $t\to t^{c}$.

The measure $\Phi(\mu)$ is characterized by
\[
\Phi(\mu)\left(\left[0,\frac{1}{S_\mu(t-1)}\right]\right) =t \qquad \text{for all } t\in (0,1 - \mu(\{0\}))
\]
unless $\mu$ is not a Dirac measure. If $\mu = \delta_a$ for some $a\ge 0$, then $\Phi(\mu)=\delta_a$.

Fujie and Ueda \cite[Theorem 3.2]{fujie2023law} proved a finite free analogue of this result. Namely, for $p\in \pols_d(\R_{\geq 0})$, there exists a limiting polynomial
\[
\lim_{n\to\infty}(p^{\boxtimes_d n})^{\langle1/n\rangle}\to  \Phi_d(p).
\] 
Here, if $q\in \pols_d(\R_{\geq 0})$ has roots $\lambda_1,\dots, \lambda_d$, then $q^{\langle c\rangle}$ stands for the polynomial with roots $\lambda_1^c,\dots, \lambda_d^c$. Moreover, for $k=1, \dots, d$, the $k$-th largest root of $\Phi_d(p)$ is given by  $\lambda_k(\Phi_d(p))=\frac{\coef{k}{d}(p)}{\coef{k-1}{d}(p)}$, which is the multiplicative inverse of the finite $S$-transform $1/{\strans{p}{d}\left(-\frac{k}{d}\right)}$ at $-k/d$.

In \cite[Conjecture 4.4]{fujie2023law}, it was conjectured that the map $\Phi_d$ tends to the map $\Phi$ as $d$ tends to $\infty$. As a consequence of our main theorem, we will prove this conjecture in Section \ref{sec:approx.THM}.
\begin{theorem}
\label{thm:main3}
Let $(p_d)_{d\in \N}$ be a sequence of polynomials with $p_d\in \pols_d(\R_{\geq 0})$ and $\mu\in \MM(\R_{\geq 0})$. The convergence   $\meas{p_d} \weak \mu$ is equivalent to the convergence $\meas{\Phi_d(p_d)}\weak \Phi(\mu)$.
\end{theorem}

\subsection{Organization of the paper}

Besides this introductory section and the upcoming section with basic preliminaries, the paper is roughly divided into three parts.

First, Sections \ref{Sec:Derivatives}, \ref{sec:partial.order} and \ref{sec:bounds.Jacobi} contain key results used in the proof of the main theorem. Each topic might be of independent interest.

\begin{itemize} \item In Section \ref{Sec:Derivatives}, we investigate the coefficients, finite free cumulants and limiting distribution of derivatives of polynomials, as a consequence we provide a new proof of Theorem \ref{Thm:AGVP_HK_S}.

 \item In Section \ref{sec:partial.order}, we equip the set of real-rooted polynomials with a partial order that allows us to reduce our study to simpler polynomials with all its roots equal to 0 or 1. 
 
 \item We bound the roots obtained after differentiating these simple polynomials several times in Section \ref{sec:bounds.Jacobi} by using classic bounds on the roots of Jacobi polynomials.

 \end{itemize}

Next, Sections  \ref{sec:finite.Strans}, \ref{sec:finite.S.to.S}, and \ref{sec:symmetric.and.unitary} concern our central object, the finite $S$-transform and its relation to Voiculescu's $S$-transform.

\begin{itemize}
\item In Section \ref{sec:finite.Strans}, we introduce the finite $S$-transform and study its basic properties.
\item Section \ref{sec:finite.S.to.S} is devoted to showing that the finite $S$-transform tends to Voiculescu's $S$-transform.
\item In Section \ref{sec:symmetric.and.unitary}, we extend the definition of the $S$-transform to symmetric polynomials and explore the case where the polynomials have roots only in the unit circle. 
\end{itemize}

Finally, in Sections  \ref{sec:approx.THM} and \ref{sec:examples.applications} we provide examples and applications.
\begin{itemize}\item Theorem \ref{thm:main3}, which relates the laws of large numbers for $\boxtimes_d$ and for $\boxtimes$, is proved in Section \ref{sec:approx.THM}.
\item Section \ref{sec:examples.applications} contains examples and applications in various directions: an approximation of free convolution, a limit for the coefficients of polynomial, examples with hypergeometric polynomials, finite analogue of some free stable laws, a finite free  multiplicative Poisson's law, and finite free max-convolution powers.
\end{itemize}

\section{Preliminaries}
\label{sec:prelim}

\subsection{Measures}
\label{ssec:measures}

We use $\MM$ to denote the family of all Borel probability measures on $\C$.  When we want to specify that the support of the measure is contained in a subset $K \subset \C$ we use the notation
\begin{equation*}
  \MM(K) := \{\mu \in \MM \mid \text{the support of $\mu$ is in $K$} \}.
\end{equation*}
In most of the cases we will let $K$ be a subset of the real line $\R$, like the positive real line $\mathbb{R}_{>0}:=(0,\infty)$, the set of non-negative real line $\rr_{\geq 0}:=[0,\infty)$, or a compact interval $C=[\alpha,\beta]$. 

\begin{notation}[Basic transformations of measures]
\label{not:measure.transformations} 
We fix the following notation:
\begin{itemize}
\item \textbf{Dilation.} For $\mu \in \MM$ and $c\neq 0$, we let $\dil{c}\mu \in \MM$ be the measure satisfying that
\[\dil{c}\mu (B):= \mu(\{x/c: x\in B\}) \qquad \text{for every Borel set }B\subset \C.\]
\item \textbf{Shift.} For $\mu \in \MM$, we let $\shift{c}\mu \in \MM$ be the measure satisfying that
\[
\shift{c}\mu (B):= \mu(\{x-c: x\in B\})\qquad  \text{for every Borel set }B\subset \C.
\]
\item \textbf{Power of a measure.} For $\mu\in \MM(\R_{\geq 0})$ and $c>0$, we denote by $\mu^{\langle c \rangle}$ the measure satisfying that
\[\mu^{\langle c \rangle}(B):= \mu(\{x^{\frac{1}{c}}:x\in B\}) \qquad  \text{for every Borel set }B\subset \rr_{\geq 0}.\] 

\item \textbf{Reversed measure.} For $\mu\in \MM(\rr_{\geq 0})$ such that $\mu(\{0\})=0$, we denote by $\mu^{\langle -1 \rangle}$ the measure satisfying that
\[\mu^{\reversed}(B):=\mu(\{x^{-1}: x\in B\})\qquad  \text{for every Borel set }B\subset \rr_{> 0}.\]
\end{itemize}
\end{notation}

\subsection{Polynomials}
\label{ssec:polynomials}

We denote by $\pols_d$ the family of monic polynomials of degree $d$, similar as with measures, for a subset $K\subset\cc$ we use the notation
\[
\pols_d(K):=\{p \in \pols_d\mid \text{all roots of } p_d \text{ are in } K\}.
\]

Given $p\in \pols_d$ we denote by $\lambda_1(p),\lambda_2(p),\dots,\lambda_d(p)$ the $d$ roots of $p$ (accounting for multiplicity). If $p\in \pols_d(\rr)$ we further assume that $\lambda_1(p)\geq \lambda_2(p)\geq \dots\geq \lambda_d(p)$.

Given $p\in \pols_d$, the symbol $\coef{k}{d}(p)$ denotes the normalized $k$-th elementary symmetric polynomial on the roots of $p$, namely
\[
    \coef{k}{d}(p):= {\binom{d}{k}}^{-1} \sum_{1\leq i_1<\dots < i_k\leq d} \lambda_{i_1}(p) \cdots \lambda_{i_k}(p), \qquad k=1,\dots, d,
\]
with the convention that $\coef{0}{d}(p):=1$. Hence, we can express any $p\in \pols_d$ as
\[
p(x)\,=\, \prod_{j=1
}^d(x-\lambda_j(p)) \, =\, \sum_{k=0}^d  (-1)^k \binom{d}{k} \coef{k}{d}(p)x^{d-k}.
\] 

The {empirical root distribution} of $p\in \pols_d$ is defined as
\[
\meas{p} := \frac{1}{d} \sum_{i=1}^d \delta_{\lambda_i(p)},
\]
and we let the $n$-th moment of $p$ be the $n$-th moment of $\meas{p}$, that is,
\[
m_n(p):=\frac{1}{d}\sum_{i=1}^d \lambda_i^n, \qquad n=1,2,\dots.
\]
Notice that the map $\meas{\cdot}$ provides a bijection from $\pols_d$ onto the set
\[\MM_d:=\left\{ \left.\frac{1}{d} \sum_{i=1}^d \delta_{\lambda_i} \in \MM \right| \lambda_1,\dots,\lambda_d\in\cc\right\}.\]
Moreover, we have that 
$\meas{\pols_d(K)}=\MM_d(K)$ for every $K\subset \cc$, where the latter is defined as $\MM_d(K):=\MM(K)\cap \MM_d$. Notice also that for every $d$, the subset $\MM_d$ is invariant under all the transformations in Notation \ref{not:measure.transformations}, thus we can use the bijection $\meas{\cdot}$ to define the analogous transformations on the set of polynomials:
\begin{notation}[Basic transformations of polynomials]
\label{not:measure.transfations} 
Let $p \in \pols_d$ and $c\in \cc$.
\begin{itemize}
\item \textbf{Dilation.} For $c\neq 0$, we define $[\dil{c}p](x):=c^d\ p\left(\frac{x}{c}\right)$.
\item \textbf{Shift.} For $c \in \C$, we define $[\shift{c}p](x):=p(x-c)$.
\item \textbf{Power of a polynomial.} Given $p\in \polplus$ and $c> 0$,  $p^{\langle c \rangle}$ is the polynomial with roots $\lambda_i(p^{\langle c \rangle})=(\lambda_i(p))^c$  for $i=1,\dots,d$.
\item \textbf{Reversed polynomial.}
For $p\in \pols_d(\rr_{>0})$, the reversed polynomial of $p$ is the polynomial $p^\reversed\in \pols_d(\rr_{>0})$ with coefficients
\begin{equation}
\label{eq:coef_pinverse}
\coef{k}{d}(p^\reversed)= \frac{\coef{d-k}{d}(p)}{\coef{d}{d}(p)} \qquad \text{ for }k=0,\dots d.
\end{equation}
\end{itemize}
\end{notation}

\subsection{Weak convergence}

In this section, we present the well-known facts and results on the weakly convergent sequence of probability measures on $\R$. 

Given a compact interval $C \subset \R$ and a probability measure $\mu \in \MM(C)$, the \textit{Cauchy transform} of $\mu$ is
\[
G_\mu(z): = \int_C \frac{1}{z-x} \mu(\dx), \qquad z\in \C\setminus C.
\]

\begin{lemma}
\label{lem:weak_convergence_Cauchy}
Let us consider $L\in\rr$ and measures $\mu_n, \mu \in \MM(\R_{\geq L})$. The following assertions are equivalent.
\begin{enumerate}[\rm (1)]
\item The weak convergence: $\mu_n\weak \mu$ as $n\to \infty$.
\item For all $z\in \C\setminus \R_{\geq L}$, we have $\displaystyle \lim_{n\to\infty}G_{\mu_n}(z)=G_\mu(z)$.
\item Let $r<s<L$. For all $a\in[r,s]$ we have that $\displaystyle \lim_{n\to\infty}G_{\mu_n}(a)=G_\mu(a)$.
\end{enumerate}
\end{lemma}

\begin{proof}
The proof is similar to \cite[Corollary 3.1]{CB2019}, but we include the proof for the reader's convenience. For $z\in \C\setminus \R_{\geq L}$, the real and imaginary parts of $x\mapsto (z-x)^{-1}$ are bounded and continuous on $\R_{\geq L}$. Thus the definition of weak convergence shows the implication $(1)\Rightarrow (2)$. The implication $(2)\Rightarrow (3)$ is trivial. We are only left to show that $(3)\Rightarrow (1)$. By Helly's selection theorem (see e.g. \cite[Theorem 3.2.12]{durrett2019probability}), for any subsequence of $(\mu_n)_{n\in \N}$, there exists a further subsequence $(\mu_{n_k})_{k\in \N}$ and a finite measure $\nu$ (on $\R_{\geq L}$) such that $\mu_{n_k}$ vaguely converges to $\nu$, and therefore $\lim_{k \to \infty} G_{\mu_{n_k}}(z) = G_\nu(z)$ for all $z\in [r,s]$, where the Cauchy transform $G_\nu$ can be defined even if $\nu$ is a finite measure.
By assumption (3), we have $G_\mu(a)=G_\nu(a)$ for all $a\in [r,s]$.
Since $G_\mu$ and $G_\nu$ are analytic on $\C\setminus \R_{\geq L}$, the principle of analytic extension shows that $G_\nu(z)=G_\mu(z)$ for all $z\in \C\setminus \R_{\geq L}$.
Finally, we get $\nu=\mu \in \MM(\R_{\geq L})$ by the Stieltjes inversion formula, and hence $\mu_{n_k}\weak \mu$. Then \cite[Theorem 3.2.15]{durrett2019probability} yields $\mu_n\weak \mu$.
\end{proof}

Given $\mu \in \MM(\rr)$, its \textit{cumulative distribution function} is the function $\cdf{\mu}: \R \to [0,1]$ such that
\[
  \cdf{\mu}(x) := \mu((-\infty, x]) \qquad \text{for all }x\in \R.
\]
Let $\{\mu_n\}_{n\in\N}$ and $\mu$ be probability measures on $\R$.
It is well known that the weak convergence of $\mu_n$ to $\mu$ is equivalent to the convergence of their cumulative distribution functions $\cdf{\mu_n}$ to $\cdf{\mu}$ on the continuous points of $\cdf{\mu}$.
Actually, such convergence is locally uniform by Polya's theorem if $\cdf{\mu}$ is continuous.

\begin{lemma}\label{lem:luz}
  Point-wise convergence of $\cdf{\mu_n}$ to the continuous function $\cdf{\mu}$ means the locally uniform convergence.
 \end{lemma}

The \textit{right-continuous inverse} of $F_\mu$ is defined to be 
\[
  \cdf{\mu}^{-1}(t) = \sup\{x \in \R \mid \cdf{\mu}(x) < t \} \qquad \text{for }t\in (0,1),
\]
see \cite[The proof of Theorem 1.2.2.]{durrett2019probability} for example.
The following lemma is usually used to show Skorohod's representation theorem in dimension one.

\begin{lemma}[Ref. {\cite[Theorem 3.2.8.]{durrett2019probability}}]\label{lem:oscuridad}
  The convergence of $\cdf{\mu_n}$ to $\cdf{\mu}$ on the continuous points of $\cdf{\mu}$ is equivalent to the convergence of their right-continuous inverse functions $\cdf{\mu_n}^{-1}$ to $\cdf{\mu}^{-1}$ on the continuous points of $\cdf{\mu}^{-1}$.
\end{lemma}

\subsection{Free Probability}

In this section we review some of the basics of free probability. 
For a comprehensive introduction to free probability, we recommend the monographs \cite{voiculescu1992free, nica2006lectures, mingo2017free}. 

Free additive and multiplicative convolutions, denoted by $\boxplus$ and $\boxtimes$ respectively, correspond to the sum and product of free random variables, that is, $\mu_a\boxplus\mu_b=\mu_{a+b}$ and $\mu_a\boxtimes\mu_b=\mu_{ab}$ for free random variables $a$ and $b$. In this paper, rather than using the notion of free independence we will work solely with the additive and multiplicative convolutions, which can be defined in terms of the $R$ and $S$ transforms, respectively.

\subsubsection{Free additive convolution and $R$-transform}
\label{ssec:prelim.R.transform}

Given measures $\mu,\nu \in \MM(\R)$, their \textit{free additive  convolution} $\mu\boxplus \nu$ is the  distribution of $X+Y$, where $X$ and $Y$ are freely independent non-commutative random variables distributed as $\mu$ and $\nu$, respectively.
The convolution $\boxplus$ was introduced by Voiculescu as a binary operation on compactly supported measures. The definition was extended to measures with unbounded support in \cite{BerVoi1993}. 

The free convolution can be described analytically by the use of the $R$-transform.
For every $\mu \in \MM(\R)$, it is known that the Cauchy transform $G_\mu$ has a unique compositional inverse, $K_\mu$, in a neighborhood of infinity, $\Gamma_\mu$.
Thus, one has
\[G_\mu(K_\mu(z))=z,\qquad z \in \Gamma_\mu.\]
The $R$-transform of $\mu$ is  defined as $R_\mu(z)=K_\mu(z)-\frac{1}{z}$.

\begin{definition}[Free additive convolution]  Let $\mu$ and $\nu$  be probability measures on the real line. 
We define the free convolution of $\mu$ and $\nu$, denoted by $\mu\boxplus \nu$, as the unique measure satisfying 
\[R_{\mu\boxplus\nu} (z)=R_{\mu}(z)+R_{\nu}(z), \qquad z \in \Gamma_\mu\cap  \Gamma_\nu.\]
\end{definition}

\subsubsection{Free cumulants}

For any  probability measure $\mu$, we denote by $m_n(\mu):=\int_\R t^n \mu(dt)$ its $n$-th moment whenever $\int_\R |t|^n \mu(dt)$ is finite. The \emph{free cumulants} \cite{speicher94multiplicative} of $\mu$, denoted by $(\freec{n}{\mu})_{n=1}^\infty$, are recursively defined via the moment-cumulant formula
\begin{equation*}\label{MCF}
m_n(\mu) =\sum_{\pi\in \mathrm{NC}(n)}\freec{\pi}{\mu},
\end{equation*} 
where $\mathrm{NC}(n)$ is the noncrossing partitions of $\{1, \dots, n\}$ and $\freec{\pi}{\mu}$ are the multiplicative extension of $(\freec{n}{\mu})_{n=1}^\infty$.
It is easy to see that the sequence $(m_n(\mu))_{n=1}^\infty$ fully determines $(\freec{n}{\mu})_{n=1}^\infty$ and vice-versa. In this case we can recover the cumulants from the $R$-transform as follows:

\[R_\mu(z)=\sum_{n=0}^\infty\freec{n+1}{\mu}z^{n}.\]

Hence, we can define free convolutions of compactly supported measures on the real line via their free cumulants.  Indeed, given two compactly supported probability measures $\mu$ and $\nu$ on the real line, we define $\mu\boxplus \nu$ to be the unique measure with cumulant sequence given by
\[\freec{n}{\mu \boxplus \nu} = \freec{n}{\mu}+\freec{n}{\nu}.\]

If $\mu$ and $\nu$ are compaclty supported on $\R$ then $\mu\boxplus \nu$ is also  compactly supported probability measures on $\R$, as can be seen from \cite{speicher94multiplicative}. 

Let  $\mu^{\boxplus k} =\mu \boxplus \cdots \boxplus \mu$  be the free convolution of $k$ copies of $\mu$. From the above definition, it is clear that $\freec{n}{\mu^{\boxplus k}} = k\, \freec{n}{\mu}$.
In  \cite{nica1996multiplication} Nica and Speicher discovered that one can extend this definition to non-integer powers, we refer the reader to \cite[Section 1]{shlyakhtenko2022fractional} for a discussion on fractional powers. 

\begin{definition}[Fractional free convolution powers]
Let $\mu$ be a compactly supported probability measure on the real line. For $t\geq 1$, the \emph{fractional convolution power} $\mu^{\boxplus t}$ is defined to be the unique measure with cumulants
\[\freec{n}{\mu^{\boxplus t}} = t\, \freec{n}{\mu}. \]
\end{definition}

\subsubsection{Free multiplicative convolution and $S$-transform}
\label{ssec:prelim.S.transform}

In this section, we introduce the free multiplicative convolution, the $S$-transform and related results from \cite{voiculescu1987multiplication,BerVoi1993}.
Given measures $\mu\in \MM(\R_{\geq 0})$ and $\nu\in \MM(\R)$, their \textit{free multiplicative convolution}, denoted by $\mu\boxtimes \nu$, is the  distribution of $\sqrt{X}Y\sqrt{X}$, where $X\ge 0$ and $Y$ are freely independent non-commutative random variables distributed as $\mu$ and $\nu$, respectively. The convolution $\boxtimes$ was introduced by Voiculescu as a binary operation on compactly supported measures. The definition was extended to measures with unbounded support in \cite{BerVoi1993}. 

We now introduce Voicuelscu's $S$-transform \cite{voiculescu1987multiplication}, the main analytic tool used to study the multiplicative convolution $\boxtimes$. Given a probability measure $\mu \in \MM(\R_{\geq 0})$, the \textit{moment generating function} of $\mu$ is
\[
\Psi_\mu (z):= \int_0^\infty \frac{tz}{1-tz}\mu(\dt), \quad z\in \C\setminus\R_{\geq 0}.
\]
For $\mu \in \MM(\R_{\geq 0})\setminus\{ \delta_0\}$ it is known that $\Psi_\mu^{-1}$, the compositional inverse of $\Psi_\mu$, exists in a neighborhood of $(-1 + \mu(\{0\}),0)$. The \textit{$S$-transform} of $\mu$ is defined as
\[
S_\mu(z):= \frac{1+z}{z} \Psi_\mu^{-1}(z), \qquad z\in (-1 + \mu(\{0\}),0).
\]

\begin{remark}
\label{rem:S.transform.intuition}
In this paper, it is helpful for the intuition to think of $S_\mu$ as a function on the whole interval $(-1,0)$ by allowing $S_\mu(z):=\infty$ for $z\in(-1,-1 + \mu(\{0\})]$.
We can formalize this heuristic if we instead consider the multiplicative inverse of the $S$-transform.
See the definition of $T$-transform in Equation \eqref{eq:T.transform.def}. 
\end{remark}
According to \cite[Corollary 6.6]{BerVoi1993}, for every $\mu,\nu \in \MM(\R_{\geq 0})\setminus\{\delta_0\}$ the $S$-transform satisfies an elegant product formula in some open neighborhood
\[
S_{\mu\boxtimes \nu}= S_\mu S_\nu.
\] 
For example, the formula holds on $(-1+[\mu\boxtimes \nu](\{0\}),0)$, and it is known that $[\mu\boxtimes \nu](\{0\})=\max\{\mu(\{0\}),\nu(\{0\})\}$, see \cite{belinschi2003}.

\begin{lemma}
\label{lem:original_Strans}
Let us consider $\mu \in \MM(\R_{\geq 0})\setminus\{ \delta_0\}$.
\begin{enumerate}[\rm (1)]
\item {\rm (\cite[Eq.~(3.7) and (3.11)]{BN08}) \footnote{ In principle, this is  proved in \cite{BN08} only for compactly supported measures but their argument is also valid for all measures in  $\MM(\R_{\geq 0})\setminus\{ \delta_0\}$.} }  For $t\in (0,1-\mu(\{0\}))$, $z\in (0,1)$, we have
\[
S_\mu(-tz) = S_{\dil{t}(\mu^{\boxplus 1/t})}(-z).
\]

\item {\rm (\cite[Proposition 3.13]{HS07})} If $\mu(\{0\})=0$, then 
\[
S_\mu(-t) S_{\mu^{\langle -1 \rangle}}(t-1)=1, \qquad t\in (0,1).
\]

\item {\rm (\cite[Lemma 2]{HM13})} If additionally, $\mu$ is not a Dirac measure, then $S_\mu$ is strictly decreasing on $(-1 + \mu(\{0\}),0)$.
\end{enumerate}
\end{lemma}

The multiplicative inverse of the $S$-transform, which goes by the name of $T$-transform \cite[Eq.~(15)]{dykema2007multilinear} and plays the same role as the $S$-transform.
For $\mu\in\MM(\rr_{\geq 0})$, we define its \emph{(shifted) $T$-transform} as the function $T_\mu:
(0,1)\to\rr_{\geq 0}$ such that 
\begin{equation}
\label{eq:T.transform.def}
 T_{\mu}(t)=\begin{cases}
0& \text{if } t\in(0,\mu(\{0\})], \vspace{2mm}\\ 
\dfrac{1}{S_{\mu}(t-1)} & \text{if } t\in(\mu(\{0\}),1).
\end{cases}
\end{equation}

The $T$-transform is continuous on $(0,1)$ because
\begin{equation*}
\lim_{t^+\to \mu(\{0\})} T_\mu(t)=0,
\end{equation*}
which follows from the fact that $S_\mu(t)\to\infty$ as $t$ approaches $-1 + \mu(\{0\})$ from the right when $\mu(\{0\})>0$.

\begin{remark}
\label{rem.T.tranform}
Notice that the $T$-transform contains the same information of the $S$-transform and thus it determines the measure $\mu$. Also, all the properties of the $S$-transform mentioned above can be readily adapted to the $T$-transform. We highlight two advantages of using the new transform. First is that it can be defined over the whole interval $(0,1)$, formalizing the intuition mentioned in Remark \ref{rem:S.transform.intuition}. In particular, we can define the $T$-transform of $\delta_0$ as $T_{\delta_0}(t)=0$ for all $t\in(0,1)$. Moreover the $T$-tranform can be understood as the inverse of the cumulative distribution function of the so-called law of large numbers for the multiplicative convolution that we introduce below. 
\end{remark}

Tucci \cite{tucci2010}, Haagerup and M\"{o}ller \cite[Theorem 2]{HM13} proved that, for any $\mu\in \MM(\rr_{\geq 0})$, there exists $\Phi(\mu) \in \MM(\rr_{\geq 0})$ such that
\[
(\mu^{\boxtimes n})^{\langle 1/n \rangle} \weak \Phi(\mu) \qquad \text{as } n\to \infty.
\]
If $\mu$ is a Dirac measure, then $\Phi(\mu) =\mu$.
Otherwise, the distribution $\Phi(\mu)$ is determined by
\begin{align}\label{eq:THMcharacterization}
\Phi(\mu)(\{0\})=\mu(\{0\}) \quad \text{and} \quad F_{\Phi(\mu)}\left(T_\mu(t)\right)= t\qquad \text{for all } t\in (\mu(\{0\}),1).
\end{align}
Moreover, the support of the measure $\Phi(\mu)\in\MM(\rr_{\geq 0})$ is the closure of the interval
\begin{equation} \label{eq:cierra}
    (\alpha, \beta) = \left( \left(\int_0^\infty x^{-1} \mu(dx)\right)^{-1}, \int_0^\infty x\, \mu(dx)  \right).
\end{equation}
In other words, $T_{\mu}$ and $\cdf{\Phi(\mu)}$ are inverse functions of each other.
Besides, as long as $\mu$ is not a Dirac measure, $T_\mu$ is strictly increasing and then these functions provide a one-to-one correspondence between $(\mu(\{0\}),1)$ and $(\alpha,\beta)$.

\subsection{Finite free probability}

In this section, we summarize some definitions and basic results on finite free probability.

\subsubsection{Finite free convolutions}
The finite free additive and multiplicative convolutions that correspond to two classical polynomial convolutions were studied a century ago by Szeg\"{o} \cite{szego1922bemerkungen} and Walsh \cite{walsh1922location} and they were recently rediscovered in \cite{MSS} as expected characteristic polynomials of the sum and product of randomly rotated matrices.

\begin{definition}\label{def:finite.free}
Let $p,q\in\pols_d$ be polynomials of degree $d$.
\begin{itemize}

\item The finite free additive convolution of $p$ and $q$ is the polynomial $p\boxplus_d q\in \pols_d$ uniquely determined by
\begin{equation}
    \label{eq:coeffAddConv}
  \coef{k}{d}(p\boxplus_d q) = \sum_{j=0}^k \binom{k}{j} \coef{j}{d}(p) \coef{k-j}{d}(q) \qquad \text{for }k=0,1,\dots,d.
\end{equation}
\item The finite free multiplicative convolution of $p$ and $q$ is the polynomial $p\boxtimes_d q\in\pols_d$ uniquely determined by
\begin{equation}
    \label{eq:coeffMultiConv}
  \coef{k}{d}(p\boxtimes_d q) = \coef{k}{d}(p) \coef{k}{d}(q) \qquad \text{for }k=0,1,\dots,d.
\end{equation}
\end{itemize}
\end{definition}

In many circumstances the finite free convolution of two real-rooted polynomials is also real-rooted. For $p, q \in \pols_d$, then
\begin{enumerate}[(i)]
\item   $p,q\in \pols_d(\rr) \ \Longrightarrow \ p\boxplus_d q\in \pols_d(\rr)$.
\item  $p\in \pols_d(\rr),\ q\in \pols_d(\rr_{\geq 0}) \ \Longrightarrow \ p\boxtimes_d q\in \pols_d(\rr)$.
\item  $p,q\in \pols_d(\rr_{\geq 0}) \ \Longrightarrow \ p\boxtimes_n q\in \pols_d(\rr_{\geq 0})$.
    \end{enumerate}

If we replace above the sets $\R_{\ge 0}$ by strict inclusion $\R_{>0}$ the statements remain valid.
Moreover we can use a \emph{rule of signs} to determine the location of the roots when doing a multiplicative convolution of polynomials in $\pols_d(\rr_{\geq 0})$ or $\pols_d(\rr_{\leq 0})$, see \cite[Section 2.5]{mfmp2024hypergeometric}.

\begin{definition}[Interlacing] 
Let $p,q\in\pols_d(\rr)$. We say that $q$ {interlaces} $p$, denoted by $p \preccurlyeq q$, if 
\begin{equation} \label{interlacing1}
    \quad \lambda_d(p) \leq \lambda_d(q) \leq \lambda_{d-1}(p) \leq \lambda_{d-1}(q) \leq \cdots \leq  \lambda_1(p) \leq \lambda_1(q).
\end{equation}
We use the notation $p \prec q$ when all inequalities in \eqref{interlacing1} are strict.
\end{definition}

From the real-root preservation and the linearity of the free finite convolution, one can derive an interlacing-preservation property for the free convolutions, see \cite[Proposition 2.11]{mfmp2024hypergeometric}.
\begin{proposition}[Preservation of interlacing]
\label{lem:preservinginterlacingMult}
If $p, q \in \pols_d(\R)$ and $p \preccurlyeq q$, then
\[
r \in \pols_d(\R) \quad  \Rightarrow \quad p\boxplus_d r \preccurlyeq q \boxplus_d r
\]
and 
\[
r \in \pols_d(\R_{\ge 0}) \quad  \Rightarrow \quad p\boxtimes_d r \preccurlyeq q \boxtimes_d r.
\]
The same statements hold if we replace all $\preccurlyeq$ by $\prec$.
\end{proposition}

The finite free cumulants were defined in \cite{AP} as an analogue of the free cumulants \cite{speicher94multiplicative}. Below we define the finite free cumulants using the coefficient-cumulant formula from \cite[Remark 3.5]{AP} and briefly mention the basic facts that will be used in this paper. For a detailed explanation of these objects we refer the reader to \cite{AP}.

\begin{definition}[Finite free cumulants] 
\label{defi.finite.free.cumulants}
The \textit{finite free cumulants} of a polynomial $p\in \pols_d$ are the sequence $(\ffc{n}{d}(p))_{n=1}^d$ defined in terms of the coefficients as
\begin{equation}
\label{eq:cumulants}
\ffc{n}{d}(p) := \frac{(-d)^{n-1}}{(n-1)!} \sum_{\pi \in \partlat(n)} \ (-1)^{\blocks{\pi}-1}
(\blocks{\pi}-1)! \prod_{V\in \pi}  \coef{|V|}{d}(p) \qquad \text{ for } n=1,2,\dots,d,
\end{equation}
where $\partlat(n)$ is the set of all partitions of $\{1,\dots, n\}$, $\blocks{\pi}$ is the number of blocks of $\pi$ and $|V|$ is the size of the block $V\in\pi$.
\end{definition}

The main property of finite free cumulants is that they linearize the finite free additive convolution $\boxplus_d$.
\begin{proposition}[Proposition 3.6 of \cite{AP}]
For $p,q\in \pols_d$ it holds that \[\ffc{n}{d}(p\boxplus_d q)=\ffc{n}{d}(p) + \ffc{n}{d}(q) \qquad \text{for } n=1,\dots,d.\] 
\end{proposition}

\subsubsection{Limit theorems}

The connection between free and finite free probability is revealed in the asymptotic regime when we take the degree $d\to \infty$. To simplify the presentation throughout this paper, we introduce some notation for a sequence of polynomials that converge to a measure.

\begin{notation}
\label{not:converging}
We say that a sequence of polynomials $\left(p_j\right)_{j\in \nn}$ {converges to a measure $\mu\in\MM$} if
\begin{equation}
p_j\in \pols_{d_j} \,\, \text{for all}\,\,
j\in \nn, \qquad
\lim_{j\to\infty}d_j=\infty \qquad \text{and}\qquad 
\meas{p_j}\weak\mu\,\, \text{as}\,\, j\to\infty.
\end{equation}

Furthermore, given a degree sequence $(d_j)_{j\in\N}$ we say that $(k_j)_{j\in\nn}$ is a \emph{diagonal sequence with ratio limit $t$} if 
\begin{equation}
k_j\in \{1,\dots,d_j\} \,\, \text{for all}\,\,
j\in \nn \qquad\text{and}\qquad 
\lim_{j\to\infty} \frac{k_j}{d_j}=t.
\end{equation}
\end{notation}
Notice that if $C\subset \cc$ is a closed subset and $\mfp\subset \pols_d(C)$ converges to $\mu$ then necessarily $\mu\in\MM(C)$.

In case when the limit distribution is compactly supported, we can characterize the convergence in distribution via the moments or cumulants. While not stated explicitly, the proof of this proposition is implicit in \cite{AP}.

\begin{proposition}
\label{prop:convergence_cumulants}
Fix $K$ a compact subset on $\R$.
Let $\mu \in \MM(\rr)$  and $(p_d)_{d\in\N}$ be a sequence of polynomials of degree $d$ such that every $p_d\in \PP_{d}(K)$ has only roots in $K$.  
The following assertions are equivalent.
  \begin{enumerate}[\rm (1)]
  \item The sequence $\left(\meas{p_d}\right)_{d\in \N}$ converges to $\mu$ in distribution.
  \item Moment convergence: $\lim_{d\rightarrow \infty} m_n(p_d)=m_n(\mu)$ for all $n\in \N$.
  \item Cumulant convergence: $\lim_{d\rightarrow \infty} \ffc{n}{d}(p_d)=\freec{n}{\mu}$ for all $n\in \N$.
  \end{enumerate}
\end{proposition}

\begin{proposition}[Corollary 5.5 in \cite{AP} and Theorem 1.4 in \cite{AGP}]
    \label{prop:finiteAsymptotics}
Let $p_d, q_d \in \pols_d(\R)$ be two sequences of polynomials converging to compactly supported measures $\mu,\nu\in \MM(\R)$, respectively.
Then
\begin{enumerate}
    \item[\rm (1)]  The sequence $\left( \meas{p_d \boxplus_{d} q_d}\right)_{d\in \N}$ converges to $\mu \boxplus \nu$ in distribution.
    \item[\rm (2)] If additionally $q_d\in \pols_{d}(\rr_{\geq 0})$ and $\nu\in \MM(\rr_{\geq 0})$ then $\left(\meas{p_d \boxtimes_{d} q_d}\right)_{d\in \N}$ converges to $\mu \boxtimes \nu$ in distribution.
\end{enumerate}
\end{proposition}

A finite free version of Tucci, Haagerup and M\"{o}ller limit from \eqref{eq:THMcharacterization} was studied by Fujie and Ueda \cite{fujie2023law}.
Given $p\in \pols_d(\rr_{\geq 0})$ a polynomial with multiplicity $r$ at root $0$, then \cite[Theorem 3.2]{fujie2023law} asserts that there exists a limiting polynomial
\begin{equation}
\label{eq:phi.polynomial}
 \Phi_d(p):=
\lim_{n\to\infty} (p^{\boxtimes_d n})^{\langle {1}/{n}\rangle}.   
\end{equation}
Moreover, the roots of $\Phi_d(p)$ can be explicitly written in terms of the coefficients of $p$:
\begin{equation}
\label{eq:phi.roots}
\lambda_k\left( \Phi_d(p) \right)=\begin{cases}
\dfrac{\coef{k}{d}(p)}{\coef{k-1}{d}(p)} & \text{if }1\leq k\leq d-r, \vspace{2mm} \\
0 & \text{if } d-r+1\leq k\leq d.
\end{cases}
\end{equation}

\section{Finite free cumulants of derivatives}
\label{Sec:Derivatives}

Since the work of Marcus \cite{Mar21} and Marcus, Spielman and Srivastava \cite{MSS}, it has been clear that the finite free convolutions behave well when applying differential operators, in particular, with respect to differentiation. One instance of such behaviour is the content of Theorem \ref{Thm:AGVP_HK_S} from the introduction, stating that the asymptotic root distribution of polynomials after repeated differentiation tends to fractional free convolution.  

In this section, we will collect some simple but powerful lemmas regarding finite free cumulants and differentiation that will a source of useful insight for the rest of the paper. In particular, some of the ideas are key steps in the proof of our main theorem.
Interestingly enough, the results of this section allow us to provide a more direct proof of Theorem \ref{Thm:AGVP_HK_S}, which is given at the end of this section. 
Note that later in Section \ref{ssec:case.unbounded.support} we will generalize this result to Theorem \ref{thm:general.AGVP}.

\begin{notation}
We will denote by $\diff{k}{d}: \pols_d \to \pols_k$ the operation 
\[
\diff{k}{d} p := \frac{p^{(d-k)}}{\falling{d}{d-k}} 
\]
that differentiates $d-k$ times a polynomial of degree $d$ and then normalizes by $\frac{1}{\falling{d}{d-k}}$ to obtain a monic polynomial of degree $k$.
\end{notation}

Notice that directly from the definition, we have that
\[\diff{j}{k} \circ \diff{k}{d} = \diff{j}{d}\qquad \text{for }j \le k \le d.\] 
This can be understood as the finite free analogue of the following well-known property of the fractional convolution powers:
\begin{equation*}
\left(\mu^{\boxplus t}\right)^{\boxplus u} = \mu^{\boxplus tu}\qquad \text{for $t,u \ge 1$}.
\end{equation*}
This will be clear from the next series of lemmas in particular from Corollary \ref{cor:cumulant_derivatives}.

Our first main claim is that the normalized coefficients of a polynomial are invariant under the operations that we just introduced. 

\begin{lemma}
\label{lem:coefficients_derivatives}
If $p\in \pols_d$ then 
\[
\coef{j}{k}(\diff{k}{d} p ) = \coef{j}{d}(p) \qquad \text{ for } 1\leq j\leq k \leq d.
\]
\end{lemma}
\begin{proof}
Recall that we write the polynomial $p$ of degree $d$ as
\[
p(x)=\sum_{j=0}^d (-1)^j \binom{d}{j}\coef{j}{d}(p) x^{d-j}.
\]
Then
\begin{align*}
\frac{p'(x)}{d} = \sum_{j=0}^{d-1}(-1)^j \binom{d-1}{j} \coef{j}{d}(p)  x^{d-j-1}.
\end{align*}
Thus $\coef{j}{d}(p)=\coef{j}{d-1}(\diff{d-1}{d} p)$ for $j=1,\dots,d-1$.
If we now fix $j$ and iterate this procedure then we conclude that
\[
\coef{j}{d}(p) =\coef{j}{d-1}(\diff{d-1}{d} p)=\coef{j}{d-2}(\diff{d-2}{d} p)=\dots =\coef{j}{k}(\diff{k}{d} p)
\]
as desired.
\end{proof}

\begin{remark}
\label{rem:diff_addtive_multiplicative}
Since additive and multiplicative convolutions only depend on these normalized coefficients, a direct implication is that convolutions are operations that commute with differentiation. Specifically, for $k\leq d$ and $p,q\in \pols_d$, one has
\[
\diff{k}{d}(p\boxplus_d q) = \left(\diff{k}{d} p\right)\boxplus_{k}  \left(\diff{k}{d} q\right) \qquad \text{ and}\qquad \diff{k}{d}(p\boxtimes_d q) = \left(\diff{k}{d} p\right)\boxtimes_{k}  \left(\diff{k}{d} q\right).
\]
To the best of our knowledge, these identities have not appeared before in the literature. However, the formula for the additive convolution follows easily from two facts mentioned in \cite{MSS}: additive convolution commutes with differentiation (Section 1.1) and a relation between polynomials with different degrees (Lemma 1.16).
\end{remark}

Lemma \ref{lem:coefficients_derivatives} can be converted to a similar expression, but here we use finite free cumulants instead of coefficients.

\begin{proposition}
\label{prop:cumulant_derivatives}
Given a polynomial $p\in \pols_d$, one has 
\[
\ffc{n}{j}(\diff{j}{d}p) = \left( \frac{j}{d}\right)^{n-1} \ffc{n}{d}(p) \quad \text{ for } 1 \leq n \leq j \leq d. 
\]
\end{proposition}
\begin{proof}
For $1\le n \le j \le d$, we compute
\begin{align*}
\ffc{n}{j}(\diff{j}{d}p) 
&= \frac{(-j)^{n-1}}{(n-1)!} \sum_{\pi \in \partlat(n)} \ 
(-1)^{\blocks{\pi}-1}(\blocks{\pi}-1)! \prod_{V\in \pi}  \coef{|V|}{j}(\diff{j}{d}p)& \text{(by \eqref{eq:cumulants})} \\
&= \left(\frac{j}{d}\right)^{n-1} \frac{(-d)^{n-1}}{(n-1)!} \sum_{\pi \in \PP(n)} (-1)^{\blocks{\pi}-1}(\blocks{\pi}-1)! \prod_{V\in \pi}  \coef{|V|}{d}(p) & \text{(by Lemma \ref{lem:coefficients_derivatives})} \\
&=\left(\frac{j}{d}\right)^{n-1} \ffc{n}{d}(p). & \text{(by \eqref{eq:cumulants})}
\end{align*}
\end{proof}

By basic properties of cumulants, the factor $(j/d)^{n-1}$ can be interpreted as doing a dilation and a fractional convolution:

\begin{corollary}
\label{cor:cumulant_derivatives}
Given a polynomial $p\in \pols_d$, one has
\[
\ffc{n}{j}\left(  \diff{j}{d}p\right) =\ffc{n}{d}\left( \dil{\frac{j}{d}}p^{\boxplus_d\frac{d}{j}} \right)\quad \text{for $1 \leq n \leq j \leq d$,}
\] 
where the free fractional finite free convolution power $p^{\boxplus_d t}$ is a polynomial determined by $\ffc{n}{d}\left(p^{\boxplus_dt}\right) = t\ffc{n}{d}\left(p\right)$ for $t\ge 1$.
\end{corollary}

As a direct consequence of this result we will now prove that derivatives of a sequence of polynomials tend to the free fractional convolution of the limiting distribution.
This result was conjectured by Steinerberger \cite{steinerberger2019nonlocal, steinerberger2020} and then proved formally by Hoskins and Kabluchko \cite{hoskins2020dynamics} using differential equations. A proof using finite free multiplicative convolution was given by Arizmendi, Garza-Vargas and Perales \cite{AGP}. Notice that the main upshot of this new proof is that we no longer need to use the fact that finite free multiplicative convolution tends to free multiplicative convolution. Instead, we will use this result later to give a new proof that finite free multiplicative convolution tends to free multiplicative convolution. 

\begin{proof}[New proof of Theroem \ref{Thm:AGVP_HK_S}]
By Proposition \ref{prop:cumulant_derivatives} we know that 
\[
\ffc{n}{j_i}\left( \diff{j_i}{d_i} p_i \right) =\left( \frac{j_i}{d_i}\right)^{n-1} \ffc{n}{d_i}(p_i) \qquad \text{ for } 1 \leq n \leq j_i.
\]  
Clearly, since $\lim_{i\to\infty}\frac{j_i}{d_i}= t>0$ then $\lim_{i\to\infty}j_i=\infty$. Thus, if we fix $n\in \N$, the above equality holds for all sufficiently large $i$. Then 
\[
\lim_{i\to \infty} \ffc{n}{j_i}\left( \diff{j_i}{d_i} p_i \right)=\lim_{i\to \infty}\left( \frac{j_i}{d_i}\right)^{n-1} \ffc{n}{d_i}(p_i)=t^{n-1} \freec{n}{\mu}=\freec{n}{\dil{t}\left(\mu^{\boxplus\frac{1}{t}}\right)},
\]
and convergence of all finite free cumulants for a measure with compact support is equivalent to weak convergence of measures by Proposition \ref{prop:convergence_cumulants}.
\end{proof}

In Section \ref{ssec:generalcase}, we will show that the assertion of Theorem \ref{Thm:AGVP_HK_S} still holds if we drop the assumption that $\mu$ is compact, see Theorem \ref{thm:general.AGVP}.

\begin{remark} \label{rem:added}
We should add a few words on Theorem \ref{Thm:AGVP_HK_S} and its generalization, Theorem \ref{thm:general.AGVP}.
We can always extend the result for $t=1$ but not for $t=0$.
Let us explain this case in detail.
First we consider $t=1$.
Given $p_d \in \PP_d(\R)$ and any interval $I \subset \R$ by interlacing property, we can see that
\[
  \left|\meas{p_d} (I) - \meas{\diff{(d-1)}{d}p_d}(I)\right| \le \frac{2}{d-1}.
\]

Indeed, if the number of roots of $p_d$ included in $I$ is $k$ then the possibility for the number of roots of $p_d'$ included in $I$ is only $k-1$, $k$, or $k+1$.
This together with the inequality
$\frac{k+1}{d-1} > \frac{k}{d} > \frac{k-1}{d-1}$, gives the bound above. 

Hence, if $j/d \to 1$ then 
\[
  \left|\meas{p_d} (I) - \meas{\diff{j}{d}p_d}(I)\right| \le \sum_{i=1}^{d-j} \frac{2}{d-i} \le \frac{2(d-j)}{j} \to 0
\]
when $d \to \infty$.
This means that $\meas{\diff{j}{d}p_d} \xrightarrow{w} \mu$.

For $t=0$, the situation becomes a bit more complicated. This case is essentially related to the law of large numbers of (finite) free probability.
For instance, let us consider the sequence of polynomials $p_d(x) = x^{d-1}(x-d)$.
Clearly, $\meas{p_d}  \xrightarrow{w} \delta_0$, but  $\diff{1}{d} p_d=(x-1)$ and then $\meas{\diff{1}{d} p_d} = \delta_1 \not \to \delta_0$.
So, there should be some additional assumptions.
If polynomials are uniformly supported by some compact set, we can use the finite free cumulants and their convergence.
Precisely, 
\[
    \ffc{n}{j}(\diff{j}{d} p_d) = \left(\frac{j}{d}\right)^{n-1}  \ffc{n}{d}(p_d).
\]
Hence, if $j/d \to 0$ and $j \to \infty$,the finite cumulants satify that $\ffc{n}{j}(\diff{j}{d} p_d)\to0$ for $n\geq1$ and $\ffc{1}{j}(\diff{j}{d} p_d)= k_1(p_d)=m_1(p_d)$. Thus, the limit measure of $\meas{\diff{j}{d} p_d}$ is $\delta_a$ where $a$ is the mean of $\mu$, which corresponds to the limit of $\mu_t = \dil{t}(\mu^{\boxplus 1/t})$ as $t\to 0$ by the law of large numbers.
In other words, it is necessary to assume that the limit measure $\mu$ has the first moment $a$ and also the convergence of first moments of $p_d$, i.e. $\coef{1}{d}(p_d) \to a$ because the limit of $\meas{\diff{1}{d} p_d}$ should be $\delta_a$.

Besides, let us consider the polynomials $p_d(x) = x^{d-2}(x-d)(x+d)$ for $d\ge 3$.
It is trivial that $\meas{p_d} \weak \delta_0$ and $\coef{1}{d}(p_d) = 0$ but $\coef{2}{d}(p_d) = \frac{2d}{(d-1)} \to 2$ as $d \to \infty$, which means $\diff{2}{d}p_d \to x^2 - 2$.
Thus, we additionally need to assume $\coef{2}{d}(p_d) \to a^2$.

Under this two assumptions, we may conclude by the following key formula:
\[
  \frac{\ffc{2}{d}(p_d)}{d} = \coef{1}{d}(p_d)^2 - \coef{2}{d}(p_d) = \frac{\Var(p_d)}{d-1} = \frac{\Var(\diff{j}{d}p_d)}{j-1}.
\]
Hence, if $\coef{2}{d}(p_d) \to a^2$ then $\Var(\diff{j}{d}p_d) \to 0$ for a fixed integer $j$.
That is, $\meas{\diff{j}{d}p_d} \weak \delta_a$.
For $j\to \infty$ and $j/d \to 0$, we need a bit stronger assumption on the boundedness of $\Var(p_d)$:
if
\[
    \Var(\diff{j}{d}p_d) = \frac{j-1}{d-1} \Var(p_d) \to 0,
\]
that is, $\ffc{d}{2}(p_d) = o(d/j)$, then we have $\meas{\diff{j}{d}p_d} \to \delta_a$.

Finally, we conclude this remark by mentioning that the CLT for $t=0$ was recently dealt with by A. Campbell, S. O'rourke, and D. Renfrew in \cite{campbell2024universality}. The considerations above can be modified to give a proof of this result, using finite free cumulants.
\end{remark}

\section{A partial order on polynomials}
\label{sec:partial.order}

In this section we equip $\pols_d(\rr)$ with a partial order $\ll$ that compares the roots of a pair of polynomials.
This partial order, defined through the bijection between $\pols_d(\rr)$ and $\MM_d(\rr)$, comes from a partial order in measures which was studied in connection to free probability in \cite{BerVoi1993}.

\begin{notation}[A partial order on measures]
Given $\mu, \nu\in \MM(\rr)$ we say that $\mu \ll \nu$ if their cumulative distribution functions satisfy
\[F_{\mu}(t) \geq F_{\nu}(t) \qquad \text{ for all }t\in \rr.\]
\end{notation}

In \cite[Propositions 4.15 and 4.16]{BerVoi1993} it was shown that for measures $\mu, \nu\in \MM(\rr)$ such that $\mu \ll \nu$,
\begin{enumerate}[(i)]
\item if $\rho\in \MM(\rr)$, then $(\mu\boxplus \rho)\ll (\nu\boxplus \rho)$, and
\item if $\rho\in \MM(\rr_{\geq 0})$, then $(\mu\boxtimes \rho)\ll (\nu\boxtimes \rho)$.
\end{enumerate}

In, particular, by considering $\rho=(1-t)\delta_0 + t\delta_1$ in {\rm (ii)} above, since $\mu \boxtimes \rho = (1-t)\delta_0 + t\dil{t}(\mu^{\boxplus 1/t})$, we readily obtain the following.

\begin{corollary}
\label{cor:majorization powers}
If $\mu \ll \nu$ then  $\dil{t}(\mu^{\boxplus 1/t}) \ll \dil{t}(\nu^{\boxplus 1/t})$ for $t \in (0,1)$.
\end{corollary}

The goal of this section is to prove finite analogues of the previous results. First we define a partial order in $\pols_d(\rr)$. Recall that given a polynomial $p\in\PP_d(\rr)$ its roots are denoted by $\lambda_1(p)\geq \lambda_2(p) \geq \dots \geq \lambda_d(p)$. 

\begin{notation}[A partial order on polynomials]
Given $p,q\in\PP_d(\rr)$ we say that $p \ll q$ if the roots of $p$ are smaller than the roots of $q$ in the following sense:
\[\lambda_i(p) \leq \lambda_i(q) \qquad \text{for all }i=1,2,\dots,d.\]
\end{notation}
It is readily seen that
\[p \ll q \qquad \Longleftrightarrow \qquad \meas{p} \ll \meas{q}.\]

\begin{theorem}[$\boxplus_d$ and $\boxtimes_d$ preserve $\ll$]
\label{thm:delicioso}
Let $p,q\in\PP_d(\rr)$ such that $p\ll q$.
\begin{enumerate}
    \item If $r\in\pols_d(\rr)$ then $(p\boxplus_d r) \ll (q\boxplus_d r)$.
    \item If $r\in\polplus$ then $(p\boxtimes_d r) \ll (q\boxtimes_d r)$.
\end{enumerate}
\end{theorem}

\begin{proof}
The case of $p=q$ is clear, so below we assume that $p\neq q$.

First we assume that $p$ and $q$ are both simple. For every $t\in [0,1]$ we construct the polynomial $p_t\in\PP_d(\rr)$ with roots given by \[\lambda_i(p_t)=(1-t) \lambda_i(p)+t\lambda_i(q)\qquad \text{ for }i=1,\dots, d,\]
so that $(p_t)_{t\in[0,1]}$ is a continuous interpolation from $p_0=p$ to $p_1=q$. 
Consider the constant
\[M:=\max \left\{|\lambda_i(q)-\lambda_i(p)| : i=1,2,\dots,d\right\}>0.\]
Notice that the assumption, $\lambda_i(p) \leq \lambda_i(q)$, implies that for every $ 0\leq s\leq t\leq 1$ and $i=1,\dots, d$ it holds that 
\begin{equation}
\label{eq:pt.root.i}
0\leq \lambda_i(p_{t}) - \lambda_i(p_{s})= (t-s)(\lambda_i(q)-\lambda_i(p)) \le (t-s)M.
\end{equation}
We now consider the constant
\[m:=\min\left\{|\lambda_i(p)-\lambda_j(p)|, |\lambda_i(q)-\lambda_j(q)| : 1\le i < j \le d\right\}.\]
Notice that for every $t\in [0,1]$ and for $1\leq i<j\leq d$, one has the lower bound 
\begin{align*}
 \lambda_i(p_t)-\lambda_j(p_t)&= (1-t)\lambda_i(p)+t\lambda_i(q)-((1-t)\lambda_j(p)+t\lambda_j(q)) \\
 &= (1-t)(\lambda_i(p)-\lambda_j(p))+t (\lambda_i(q)-\lambda_j(q))\\
 &\geq (1-t) m+tm\\
 &=m.
\end{align*}
This bound, together with \eqref{eq:pt.root.i}, guarantees that for $0 \leq t \leq t+\varepsilon\leq 1$ and $\varepsilon < \frac{m}{M}$.
Then, we have
\[0\leq \lambda_{i+1}(p_{t+\varepsilon}) - \lambda_{i+1}(p_t)\leq \varepsilon M< m \leq \lambda_{i}(p_t) - \lambda_{i+1}(p_t).\]
By adding $\lambda_{i+1}(p_t)$, we obtain that 
\[ \lambda_{i+1}(p_{t})\leq \lambda_{i+1}(p_{t+\varepsilon})  \leq \lambda_{i}(p_{t}),\]
which means that the following interlacing inequality holds:
\[ p_t \preccurlyeq p_{t+\varepsilon} \qquad \text{for $\varepsilon < \frac{m}{M}$ and $0 \leq t \leq t+\varepsilon\leq 1$.}\]
Thus, the family of polynomials $(p_t)_{t\in[0,1]}$ is monotonically increasing, i.e., their roots increase as $t$ increases.

We now explain how to prove part (1).
Since additive convolution preserves interlacing, then 
\[ (p_t\boxplus_d r) \preccurlyeq (p_{t+\varepsilon} \boxplus_d r) \qquad \text{for }\varepsilon < \frac{m}{M}\text{ and  }0 \leq t \leq t+\varepsilon\leq 1.\]
In particular, we get that $(p_t\boxplus_d r)\ll (p_{t+\varepsilon} \boxplus_d r)$. 
In other words, for every $i=1,\dots, d$, the function $f_i:[0,1]\to \rr$ defined as $f_i(t)=\lambda_i(p_t\boxplus_d r)$ is increasing, and (1) follows from
\[\lambda_i(p\boxplus_d r)=\lambda_i(p_0\boxplus_d r) \leq \lambda_i(p_1\boxplus_d r) =\lambda_i(q\boxplus_d r).\]
Part (2) follows by a similar method.

For the case, when $p$ or $q$ have multiple roots, we can simply approximate them with polynomials that are simple.
For example, for $n\in\nn$ consider the polynomials $p_n$ and $q_n$ with roots \[\lambda_1(p)-\tfrac{1}{n}, \lambda_2(p)-\tfrac{2}{n},\dots, \lambda_d(p)-\tfrac{d}{n}\qquad\text{and}\qquad\lambda_1(q)-\tfrac{1}{n}, \lambda_2(q)-\tfrac{2}{n},\dots, \lambda_d(q)-\tfrac{d}{n},\]
respectively. Then $p_n, q_n$ are simple, and satisfy $p_n\ll q_n$,  $\lim_{n\to\infty} p_n=p$ and $\lim_{n\to\infty} q_n=q$. The general result then follows from the continuity of $\boxplus_d$ and $\boxtimes_d$.
\end{proof}

As a direct corollary we get an inequality for the derivatives of polynomials, which can be seen as the finite free analogue of Corollary \ref{cor:majorization powers}.
\begin{corollary}[Differentiation preserves $\ll$]
\label{cor:derivatives.monotonicity}
Given $p,q\in\pols_d(\rr)$ such that $p\ll q$, one has 
\[(\diff{k}{d} p) \ll (\diff{k}{d}q) \qquad \text{for }k=1,\dots,d.\]
\end{corollary}

\begin{proof}
Fix a $k=1,\dots,d$ and recall that if $r(x)=x^{d-k}(x-1)^k$ then $p\boxtimes_d r(x) = x^{d-k} \diff{k}{d} p(x)$ and $q\boxtimes_d r= x^{d-k} \diff{k}{d} q$. By Theorem \ref{thm:delicioso} we get that $(x^{d-k}\diff{k}{d} p) \ll (x^{d-k}\diff{k}{d}q)$ and this is equivalent to $(\diff{k}{d} p) \ll (\diff{k}{d}q)$.
\end{proof}

Another interesting consequence of the theorem, which will be useful later, is that the map $\Phi_d: \pols_d(\rr_{\ge 0}) \to \pols_d(\rr_{\ge 0})$ from Equation \eqref{eq:phi.polynomial} also preserves the partial order.

\begin{proposition}
\label{prop.Phi.preserves.order}
Given $p,q\in\PP_d(\rr_{\geq 0})$ such that $p\ll q$, one has $\Phi_d(p)\ll \Phi_d(q)$.  
\end{proposition}

\begin{proof}
First notice that using induction we can prove  $p^{\boxtimes_d n}  \ll q^{\boxtimes_d n}$. Indeed, the base case $n=1$ is just our assumption and the inductive step follows from applying Theorem \ref{thm:delicioso} twice:
\[p^{\boxtimes_d (n+1)}=\left(p^{\boxtimes_d n}\boxtimes_d p\right) \ll\left(p^{\boxtimes_d n} \boxtimes_d q\right) \ll\left(q^{\boxtimes_d n} \boxtimes_d q\right)=q^{\boxtimes_d (n+1)}. \]
Letting $n\to\infty$, we conclude that
\[\Phi_d(p)=\lim_{n\to \infty}   \left(p^{\boxtimes_d n}\right)^{\langle {1}/{n}\rangle}\ll  \lim_{n\to \infty} \left(q^{\boxtimes_d n}\right)^{\langle {1}/{n}\rangle}=\Phi_d(q).\]
\end{proof}

\section{Root bounds on Jacobi polynomials}
\label{sec:bounds.Jacobi}

The purpose of this section is to show a uniform bound on the extreme roots of Jacobi polynomials, which will be crucial in our proof of Theorem \ref{thm:main}. The bound readily follows from a classic result of Moak, Saff and Varga \cite{moak1979zeros} after reparametrization of Jacobi polynomials, see Theorem \ref{thm:asymptotic.jacobi}.

The Jacobi polynomials have been studied thoroughly (notably by Szeg\"{o} \cite{szego1975orthogonal}) and the literature on them is extensive.
In this section we restrict ourselves to those results that are  necessary to prove the bound.
Notice that these polynomials are a particular case of the much larger class of hypergeometric polynomials.
These polynomials will be reviewed in detail later in Section \ref{ssec:hypergeometric}, providing plenty of examples using our main theorem once proved.

\subsection{Basic facts from a free probability perspective.}
We will adopt a slightly different point of view on Jacobi polynomials to emphasize the intuition coming from free probability. Specifically, by making a simple change of variables we can make the parameters of the polynomials coincide with those parameters of the measures obtained as a weak limit of the polynomials.
This provides interesting insights into the roles of these polynomials within finite free probability by drawing an analogy of the corresponding measures in free probability. 

Following the notation from \cite[Section 5]{mfmp2024hypergeometric}, if we fix a degree $d\in\nn$ and parameters $a\in \R\setminus\left\{\tfrac{1}{d},\tfrac{2}{d}, \dots, \tfrac{d-1}{d}\right\}$ and $b\in \R$, the \emph{(modified) Jacobi polynomial of parameters $a,b$} is the polynomial $\HGP{d}{b}{a}\in\pols_d$ with coefficients given by
\begin{equation}
\label{eq:coef.Jacobi}
\coef{j}{d}\left(\HGP{d}{b}{a}\right):= \frac{\falling{bd}{j}}{\falling{ad}{j}} \qquad \text{ for } j=1,\dots, d.
\end{equation}

Notice that with a simple reparametrization this new notation can be readily translated into the more common expression in terms of ${}_2F_1$ hypergeometric functions or in terms of the standard notation used for Jacobi polynomials $P^{(\alpha,\beta)}_d$.
In particular, this is the notation in the literature \cite{szego1975orthogonal,moak1979zeros}, from where we will import some results:
\begin{align*}
 \HGP{d}{b}{a}(x)&=\frac{(-1)^d\falling{b d}{d}}{\falling{ ad}{d}} \HGF{2}{1}{-d, ad-d+1}{bd -d+1}{x} &\text{\cite[Eq. (80)]{mfmp2024hypergeometric} }
\\&=\frac{d!}{\falling{ad}{d}}  \, P^{((-1+a-b)d,(b-1)d)}_{d}(2x-1).&\text{\cite[Eq. (27)]{mfmp2024hypergeometric}}
\end{align*}

With the standard notation $P^{(\alpha,\beta)}_d$, the classical Jacobi polynomials correspond to parameters $\alpha, \beta \in[-1,\infty)$ and are orthogonal on $[-1,1]$ with respect to the weight function $(1-x)^\alpha(1+x)^\beta$.
In particular, they have only simple roots, all contained in $[-1,1]$.
In our new notation, this means that
for $b>1-\tfrac{1}{d}$ and $a>b+1-\tfrac{1}{d}$
we obtain that 
\begin{equation*}
\HGP{d}{b}{a}\in\pols_d([0,1]).
\end{equation*}

The derivatives of Jacobi polynomials are again Jacobi polynomials \cite[Eq.~(4.21.7)]{szego1975orthogonal}. Specifically, from Eq.~\eqref{eq:coef.Jacobi} and Lemma  \ref{lem:coefficients_derivatives} we know that for any integer $k\leq d$ and arbitrary parameters $a,b$, one has
\begin{equation}
\label{eq:derivatives.Jacobi}
\diff{k}{d}\HGP{d}{b}{a}= \HGP{k}{\frac{bd}{k}}{\frac{ad}{k}}.
\end{equation}

\subsection{Bernoulli and free-binomial distribution}
For non-standard parameters, the Jacobi polynomials may have multiple roots. We are particularly interested in polynomials with all roots equal to $0$ or $1$.
From \cite[Page 40]{mfmp2024hypergeometric} we know that
\begin{equation}
    \label{eq:Jacobi.zero.one}
   \HGP{d}{\frac{l}{d}}{1}(x)= x^{d-l}(x-1)^l\qquad\text{for }0\leq l \leq d
\end{equation}
are polynomials whose empirical root distribution follows a Bernoulli distribution. Clearly, if we let $d\to\infty$ and $l/d\to u\in[0,1]$ then we can approximate an arbitrary Bernoulli distribution $\beta_u:=(1-u)\delta_0 +u\delta_1$ with atoms in $\{0,1\}$ and probability $u$. 

We want to understand the behaviour of these polynomials after repeated differentiation. From \cite[Lemma 3.5]{AGP} we know that this is the same as studying multiplicative convolution of two of these polynomials. Moreover, using \eqref{eq:derivatives.Jacobi} we have that
\begin{equation} \label{eq:libro}
  \HGP{d}{\frac{k}{d}}{1} \boxtimes_d \HGP{d}{\frac{l}{d}}{1}(x) = x^{d-k} \diff{k}{d}\HGP{d}{\frac{l}{d}}{1}(x) = x^{d-k} \HGP{k}{\frac{l}{k}}{\frac{d}{k}}(x)  \qquad \text{for } k,l=1,\dots, d.
\end{equation}
By Theorem \ref{Thm:AGVP_HK_S}, when $d$ tends to infinity with $\tfrac{k}{d} \to t$ and $\tfrac{l}{d} \to u$, the limit of empirical root distributions is given by
\[
  \beta_t \boxtimes \beta_u = (1-t)\delta_0 + \dil{t}(\beta_u^{\boxplus 1/t}).
\]
The distribution $\dil{t}(\beta_u^{\boxplus 1/t})$ has been studied in connection to free probability under the name of free binomial distribution.
It is also related to the free beta distributions of Yoshida \cite{Yoshida}, see also \cite{SaitoYoshida}.

\begin{remark}
\label{rem:symmetric.roles}  
Since the expression on the left-hand side of \eqref{eq:libro} is symmetric, the roles of $k$ and $l$ can be interchanged. In particular, the largest roots of them coincide: 
\begin{equation}
\label{eq:largest.root.coincide}
\lambda_1\left(\HGP{d}{\frac{k}{d}}{1} \boxtimes_d \HGP{d}{\frac{l}{d}}{1}\right)=\lambda_1\left( \HGP{k}{\frac{l}{k}}{\frac{d}{k}}\right)=\lambda_1\left( \HGP{l}{\frac{k}{l}}{\frac{d}{l}}\right).
\end{equation}

\end{remark}

In the next section we will uniformly bound the roots of these polynomials.

\subsection{Uniform bound on the extreme roots}

Finally, our main ingredient for bounding the roots is the well-understood limiting behaviour of the empirical root distribution of the Jacobi polynomials.
In particular, we will use the following result.

\begin{theorem}
\label{thm:asymptotic.jacobi}
Let us consider a degree sequence $(d_j)_{j\in\nn}$ and sequences $(a_j)_{j\in\nn}$ and $(b_j)_{j\in\nn}$ such that  
\begin{equation}
\label{eq:inequalities.Jacobi.parameters.standard}
b_j>1-\tfrac{1}{d_j}\qquad \text{and} \qquad a_j>b_j+1-\tfrac{1}{d_j} \qquad \text{for all }j\in\nn,
\end{equation}
\[\text{and with limits }\qquad \lim_{j\to\infty}a_j=a \qquad \text{and} \qquad \lim_{j\to\infty}b_j=b.\]   
Then, by \cite[page 42]{mfmp2024hypergeometric} the sequence of polynomials $p_j:= \HGP{d_j}{b_j}{a_j}$ weakly converges to the measure $\mu_{b,a}$ with density  
\begin{equation*}
    d\mu_{b,a}=\frac{a}{4\pi}\frac{\sqrt{(L_+-x)(x-L_-)}}{x(1-x)}dx, 
\end{equation*}
where $L_{+}$ and $L_{-}$ are the extremes of the support and depend on the parameters $a,b$:
\begin{equation}
\label{eq:extremes.of.support.jacobi}
L_{\pm}(a,b)=\left(\frac{\sqrt{a-b}\pm \sqrt{(a-1)b}}{a} \right)^2.
\end{equation}
Furthermore, \cite[Theorem 1]{moak1979zeros} guarantees the smallest and largest root converges to the extremes of the support. Namely, 
\begin{equation}
\label{eq:limit.min.max.root.Jacobi}
\lim_{j\to\infty} \lambda_{1}(p_j)=L_+(a,b) \qquad \text{and} \qquad \lim_{j\to\infty} \lambda_{d_j}(p_j)=L_-(a,b).
\end{equation}
\end{theorem}

\begin{remark}
\label{rem:identities.alpha.jacobi}
For $t,u\in(0,1)$, we define the values 
\[\alpha_{\pm}(t,u):=L_{\pm}\left(\tfrac{1}{t},\tfrac{u}{t}\right)= \left(\sqrt{t(1-u)}\pm\sqrt{u(1-t)}\right)^2\]
in preparation for our next result.
Notice that this value is symmetric with respect to $t$ and $u$;
\[\alpha_{\pm}(t,u)=\alpha_{\pm}(u,t).\]
Furthermore, it is easily seen that the following identities hold:
\[\alpha_{\pm}(t,u)=\alpha_{\pm}(1-t,1-u),\]
\[\alpha_{\pm}(t,u)=1-\alpha_{\mp}(t,1-u).\]
\end{remark}

We are now ready to prove the main result of this section, which concerns the asymptotic behaviour of the extreme roots of $\HGP{d}{\frac{l}{d}}{1}$ after repeated differentiation.

\begin{lemma} \label{lem:bound}
  Fix $t, u\in (0,1)$ and let $(d_j)_{j=1}^{\infty}$ be a divergent sequence of integers and $(k_d)_{d=1}^\infty$, $(l_d)_{d=1}^\infty$ diagonal sequences with limit $t,u \in (0, 1)$, respectively.
  \begin{enumerate}
    \item If $t + u <1$ then the max roots $\lambda_1\left(\diff{k_j}{d_j}\HGP{d_j}{\frac{l_j}{d_j}}{1}\right)$ converge to $\alpha(t,u)_+$.
    \item If $t<u$ then the min roots $\lambda_{k_j}\left(\diff{k_j}{d_j} \HGP{d_j}{\frac{l_j}{d_j}}{1}\right)$ converge to $\alpha(t,u)_-$.
  \end{enumerate}
\end{lemma}
\begin{proof}
We first prove (1) under the assumption $t<u$. By Equation \eqref{eq:derivatives.Jacobi}, we have
\begin{equation*} 
    \diff{k_j}{d_j}\HGP{d_j}{\frac{l_j}{d_j}}{1}=\HGP{k_j}{\frac{l_j}{k_j}}{\frac{d_j}{k_j}}.
\end{equation*}
Thus in Theorem \ref{thm:asymptotic.jacobi} we shall consider the case $a_j=\tfrac{d_j}{k_j}$ and $b_j=\tfrac{l_j}{k_j}$. Since 
\[\lim_{j\to\infty}\tfrac{k_j}{d_j}=t<u=\lim_{j\to\infty}\tfrac{l_j}{d_j},\] we have that $k_j< l_j$ for $j$ large enough.
Similarly, since $u+t<1$, we obtain that $k_j+l_j< d_j$ for $j$ large enough.
So both inequalities in hypothesis \eqref{eq:inequalities.Jacobi.parameters.standard} are satisfied for every $j$ larger than some $N$.
Since $\lim_{j\to \infty} a_j = \frac{1}{t}$ and $\lim_{j\to \infty} b_j= \frac{u}{t}$, then the max root of $\HGP{d_j}{\frac{l_j}{d_j}}{1}$  
converges to $\alpha_{+}\left(t,u\right)$ as desired.

In the case $t>u$, since \(\lambda_1\left(\diff{l_j}{d_j}\HGP{d_j}{\frac{k_j}{d_j}}{1}\right) = \lambda_1\left(\diff{k_j}{d_j}\HGP{d_j}{\frac{l_j}{d_j}}{1}\right)\) by Eq.~\eqref{eq:libro}, we have the same conclusion.
  
In the case $t=u$, let $t_1<t<t_2$ such that $t_2 + u<1$ and diagonal sequences $(k_j^{(1)})_{j=1}^\infty$, $(k_j^{(2)})_{j=1}^\infty$ with limit $t_1, t_2 \in (0, 1)$, respectively.
Then, by Corollary \ref{cor:derivatives.monotonicity},
\[
  \lambda_1\left(\diff{l_j}{d_j}\HGP{d_j}{\frac{k_j^{(1)}}{d_j}}{\overset{\ }{\scalebox{1.01}{1}}}\right) \le \lambda_1\left(\diff{l_j}{d_j}\HGP{d_j}{\frac{k_j}{d_j}}{1}\right) \le  \lambda_1\left(\diff{l_j}{d_j}\HGP{d_j}{\frac{k_j^{(2)}}{d_j}}{\overset{\ }{\scalebox{1.01}{1}}}\right)
  \]
  for large enough $j$ and 
  \[
    \lim_{j\to\infty} \lambda_1\left(\diff{l_j}{d_j}\HGP{d_j}{\frac{k_j^{(i)}}{d_j}}{\overset{\ }{\scalebox{1.01}{1}}}\right) = \alpha(t_i,u)_+
  \]
  for $i=1,2$.
  It implies 
  \[
    \alpha(t_1,u)_+ \le \liminf_{j\to\infty} \lambda_1\left(\diff{l_j}{d_j}\HGP{d_j}{\frac{k_j}{d_j}}{1}\right)  \le \limsup_{j\to\infty} \lambda_1\left(\diff{l_j}{d_j}\HGP{d_j}{\frac{k_j}{d_j}}{1}\right) \le \alpha(t_2,u)_+.
  \]
  Letting $t_1, t_2 \to t$, we obtain the result \( \lim_{j\to\infty} \lambda_1\left(\diff{l_j}{d_j}\HGP{d_j}{\frac{k_j}{d_j}}{1}\right) = \alpha(t,u)_+\).  

  For the proof of (2), note that $\HGP{d}{\frac{l}{d}}{1}(x) = (-1)^d \HGP{d}{\frac{d-l}{d}}{1}(1-x)$, which is just the changing of the variables.
  Then one has $\lambda_{k_j}\left(\diff{k_j}{d_j} \HGP{d_j}{\frac{l_j}{d_j}}{1}\right) = 1 - \lambda_{1}\left(\diff{k_j}{d_j} \HGP{d_j}{\frac{d_j-l_j}{d_j}}{1}\right)$.
  Hence, if $t<u$ (equivalently $t + (1-u) <1$) then
  \[
    \lim_{j\to\infty} \lambda_{k_j}\left(\diff{k_j}{d_j}\HGP{d_j}{\frac{l_j}{d_j}}{1}\right) = 1 - \alpha(t,1-u)_+ =\alpha(t,u)_-
  \]
  from the result of (1) and Remark \ref{rem:identities.alpha.jacobi}.
\end{proof}

\begin{lemma}\label{lem:support.fin.freeconv}
Let $(p_j)_{j\in\nn}$ be a sequence of polynomials converging to a measure $\mu\in\MM(\rr_{\ge 0})$.
For every diagonal sequence $(k_j)_{j \in \N}$ with limit $t \in (0, 1 - \mu(\{0\}))$, there exists $\epsilon >0$ such that \[
\liminf_{j\to\infty} \lambda_{k_j} (\diff{k_j}{d_j} p_j) \ge \epsilon.
\]
\end{lemma}

\begin{proof}
    Let us take $u \in (t, 1- \mu(\{0\})).$
    Since $\mu(\{0\}) < 1-u$, there exists $a>0$ such that $F_{\mu}(a) <1-u$ and $a$ is a continuous point of $F_\mu$.
Then we consider some sequence of integers $(l_j)_{j\in\N}$ such that $k_j\le l_j \le d_j$ and $l_j/d_j \to u$ and use it to construct a sequence of polynomials
\[
    q_j(x) := \dil{a}\HGP{d_j}{\frac{l_j}{d_j}}{1}(x) = x^{d_j-l_j}(x-a)^{l_j} \qquad \text{for } j\in\N.
\]

By construction we know $q_j \ll p_j$, and using Corollary \ref{cor:derivatives.monotonicity} we obtain that $\diff{k_j}{d_j} q_j \ll \diff{k_j}{d_j} p_j$. Finally, Lemma \ref{lem:bound} yields that $\lim \lambda_{k_j}(\diff{k_j}{d_j} q_j) = a \alpha(t,u)_- >0$. Thus, letting $\varepsilon:= a \alpha(t,u)_-$ we conclude that $\liminf \lambda_{k_j}(\diff{k_j}{d_j} p_j) \ge \epsilon$.
\end{proof}

\begin{corollary}\label{cor. supporfreeconv}
Let $\mu$ be a measure on $\mu \in \MM(\R_{\geq 0})\setminus\{ \delta_0\}$.    For any $t\in (0,1-\mu(\{0\}))$, there exists $\epsilon>0$ such that the measure $\mu^{\boxplus 1/t}$ is supported on $[\epsilon,\infty)$.
\end{corollary}

\section{Finite \texorpdfstring{$S$}\ -transform}
\label{sec:finite.Strans}

In this section we introduce our main object, the finite $S$-transform. The development is parallel to that of Section \ref{ssec:prelim.S.transform}, but in the finite free probability framework. After defining the finite $S$-transform we also introduce a finite $T$-transform, which contains the same information but is more suitable to deal with certain cases. Then we study all the basic properties of the finite $S$-transform. These properties can be readily transferred to properties of the finite $T$-transform.

\subsection{Definition of Finite \texorpdfstring{$S$}\ -transform and Finite \texorpdfstring{$T$}\ -transform}
We are now ready to introduce the finite $S$-transform that was advertised in the introduction. 

\begin{definition}[Finite $S$-transform]
\label{def:finite.Strans}
Let $p \in \polplus$ such that $p(x)\neq x^d$ and $r$ the multiplicity of the root $0$ in $p$. We define the {finite $S$-transform} of $p$ as the map 
\[S_p^{(d)} : \left\{\left.-\frac{k}{d}\ \right|\ k =1,2,\dots, d-r \right\} \to \R_{> 0}\] such that
\begin{equation}
\label{eq:def.Strans}
\strans{p}{d}\left(-\frac{k}{d}\right) := \frac{\coef{k-1}{d}(p)}{\coef{k}{d}(p)} \quad \text{ for } \quad k=1,2, \dots, d-r.
\end{equation}
\end{definition}

\begin{remark}
\label{rem:finite.Strans}
Notice that since $p$ has all non-negative roots, then $\coef{k}{d}(p)>0$ if and only if $0\leq k \leq d-r$. Thus, the $S$-transform is well defined, and cannot be extended to $k=d-r+1,\dots,d$ as it would produce a division by 0. Similar to what was pointed out in Remark \ref{rem:S.transform.intuition}, it is useful for the intuition to allow the values $\strans{p}{d}\left(-\tfrac{k}{d}\right) := \infty$ when $k=d-r+1,\dots,d$.
This will be formally explained below when considering the modified $T$-transform.

Another natural question is why the domain is a discrete set.
Notice that in principle the domain of $S_p^{(d)}$ can be extended to the whole interval $\left(-1 + \tfrac{r}{d},0\right)$ so that it resembles more Voiculescu's $S$-transform.
There are several ways to achieve this: defining a continuous piece-wise linear function whose non-differentiable points are at $-\frac{k}{d}$; using Lagrange's interpolation\footnote{Seems like this option produces a function that is closer to Marcus' $m$-finite $S$-transform in \cite[Eq.~(14)]{Mar21}.
We believe this because looking at \cite[page 20]{Mar21}, the function $f_A$, that is related to Marcus' $S$-transform, is obtained as Lagrange interpolation in the same set. However, we were unable to devise a clear connection between our definition and Marcus' definition. } to define a polynomial with values $\frac{\coef{k-1}{d}(p)}{\coef{k}{d}(p)}$ at $-\frac{k}{d}$; or simply defining a step function.
Although the last option do not produce a continuous function, it seems to be in more agreement with the intuition coming from the $T$-transform. 

Since it is not clear which extension is the best, we simply opted to restrict the definition to the discrete set $\left\{\left.-\frac{k}{d}\ \right|\ k =1,2,\dots, d-r \right\}$. Recall that the values of the finite $S$-transform at this point  are enough to recover the coefficients. Indeed, we just need to multiply them:
\begin{equation}
\label{eq:Strans.to.coef}
\coef{k}{d}(p)=\frac{1}{\strans{p}{d}\left(-\frac{k}{d}\right)\strans{p}{d}\left(-\frac{k-1}{d}\right)\cdots \strans{p}{d}\left(-\frac{1}{d}\right)}.
\end{equation}
\end{remark}

Recall from Equation \eqref{eq:T.transform.def} that the $T$-transform in free probability is a shift of the multiplicative inverse of the $S$-transform. Alternatively, from Equation \eqref{eq:THMcharacterization}, the $T$-transform of $\mu\in\MM(\rr_{\geq 0})$ can be understood as the inverse of the cumulative distribution function of $\Phi(\mu)$. Since  the map $\Phi_d$ from \eqref{eq:phi.polynomial} is the finite version of map $\Phi$, we will use the last interpretation to define the finite counterpart of the $T$-transform.

\begin{definition}[Finite $T$-transform]
Given a polynomial $p\in\pols_d(\rr_{\geq 0})$ we define the \emph{finite $T$-transform} of $p$ as the function $\ttrans{d}{p}(t) :(0,1)\to\rr_{\geq 0}$ that is the
right-continuous inverse of $F_{\meas{\Phi_d(p)}}$ in $(0,1)$.
\end{definition}

\begin{remark}
\label{rem:finite.Ttransform}
Using \eqref{eq:phi.roots} it is straightforward that the finite $T$-transform can be explicitly defined in terms of the coefficients of $p$. Indeed, if $r$ is the multiplicity of the root $0$ of $p$, then $\ttrans{d}{p} :(0,1)\to\rr_{\geq 0}$ is the map such that
\begin{equation}
T^{(d)}_{p}(t) :=\begin{cases} 0 & \text{if } t\in \left(0, \tfrac{r}{d}\right) \vspace{2mm}\\
\dfrac{\coef{d-k+1}{d}(p)}{\coef{d-k}{d}(p)} & \text{if } t\in \left[\tfrac{k-1}{d}, \tfrac{k}{d}\right) \text{ for }k=r+1,\dots, d.
\end{cases}
\end{equation}

Then, it is also clear that 
\begin{equation}
\label{eq:relation.T.S.finite}
T^{(d)}_{p}\left( \frac{d-k}{d}\right)= \frac{1}{\strans{p}{d}\left(-\frac{k}{d}\right)}\qquad \text{for }k=1,\dots,d-r.   
\end{equation}
\end{remark}

\subsection{Basic properties of \texorpdfstring{$S$}\ -transform}
\label{sec:properties_of_Stransform}
We now turn to the study the basic properties of $\strans{p}{d}$.
All these basic facts are analogous to those of Voiculescu's $S$-transform.
Notice also that these properties can be readily adapted to fit the finite $T$-transform; after each result we will leave a brief comment pointing out the corresponding result.

First we notice that the $S$-transform is non-increasing and compute its extreme values.
\begin{proposition}[Monotonicity and extreme values]
\label{prop:monotone.extreme.values}
Let $p \in \polplus$ such that $p(x)\neq x^d$. Let $\lambda_1\geq \dots \geq\lambda_d$ be the roots of $p$, and denote by $r$ the multiplicity of the root $0$ in $p$, namely $\lambda_{d-r}>\lambda_{d-r+1}=0$. Then
\begin{enumerate}
    \item If $p(x)=(x-c)^d$ for some constant $c>0$, then 
\begin{equation}
\label{eq:Strans.dirac.constant}
\strans{p}{d}\left(-\frac{k}{d}\right)=\frac{1}{c}\qquad \text{for }k=1,\dots,d.
\end{equation}
\item Otherwise, whenever $p$ has at least two distinct roots, the finite $S$-transform is strictly decreasing:
\begin{equation}
\label{eq:Strans.decreasing}
\strans{p}{d}\left(-\frac{k+1}{d}\right) > \strans{p}{d}\left(-\frac{k}{d}\right)\qquad \text{for }k=1,\dots,d-r-1.
\end{equation}
\end{enumerate}
Moreover the smallest and largest values are, respectively,
\begin{equation}
\label{eq:Strans.extrem.values}
\strans{p}{d}\left(-\frac{1}{d}\right)=\left(\frac{1}{d}\sum_{j=1}^d \lambda_j\right)^{-1}\quad \text{and}\quad \strans{p}{d}\left(-\frac{d-r}{d}\right)=\frac{r+1}{d-r}\sum_{j=1}^{d-r} \frac{1}{\lambda_j}.    
\end{equation}
When $p$ has no roots at 0 (when $r=0$), we can identify the latter as the value at 0 of the Cauchy transform of the empirical root distribution of $p$
\begin{equation}
\label{eq:Strans.Cauchy}
\strans{p}{d}(-1)=\frac{1}{d}\sum_{j=1}^{d} \frac{1}{\lambda_j}=-G_{\meas{p}}(0).
\end{equation}
\end{proposition}

\begin{proof}
The coefficients of $p(x)=(x-c)^d$ are of the form $\coef{k}{d}(p)=c^k$, which directly implies Equation \eqref{eq:Strans.dirac.constant}.

Assertion \eqref{eq:Strans.decreasing} follows from Newton's inequality:
\[
  \strans{p}{d}\left(-\frac{k+1}{d}\right) = \frac{\coef{k}{d}(p)}{\coef{k+1}{d}(p)} > \frac{\coef{k-1}{d}(p)}{\coef{k}{d}(p)} = \strans{p}{d}\left(-\frac{k}{d}\right)\qquad \text{for }k=1,\dots,d-r-1.
\]
This in turn implies that the smallest and largest values of $\strans{p}{d}$ are attained at $-\frac{1}{d}$ and $-\frac{d-r}{d}$, respectively. Using the definition we can compute these values explicitly:
\begin{align*}
\strans{p}{d}\left(-\frac{1}{d}\right)&=\frac{\coef{0}{d}(p)}{\coef{1}{d}(p)}=\left(\frac{1}{d}\sum_{j=1}^d \lambda_j\right)^{-1}\quad \text{and} \\
\strans{p}{d}\left(-\frac{d-r}{d}\right)&=\frac{\coef{d-r-1}{d}(p)}{\coef{d-r}{d}(p)}=\frac{\binom{d}{d-r}}{\binom{d}{d-r-1}} \sum_{j=1}^{d-r}\frac{\lambda_1\cdots\lambda_{j-1}\lambda_{j+1}\cdots \lambda_{d-r}}{\lambda_1\cdots \lambda_{d-r}}=\frac{r+1}{d-r}\sum_{j=1}^{d-r} \frac{1}{\lambda_j}.    
\end{align*}
\end{proof}

\begin{remark}
\label{rem.finite.T.trans.basic.facts}
The previous result implies that the $T$-transform is non-decreasing step function with smallest value 
\begin{equation*}
\lim_{t\to 0}\ttrans{p}{d}(t)=\begin{cases}
\left(\frac{1}{d}\sum_{i=1}^{d} \frac{1}{\lambda_i}\right)^{-1} & \text{if }r=0,  \\ 
0 & \text{if } r>0,
\end{cases}    
\end{equation*}
and largest value
\[\lim_{t\to 1}\ttrans{p}{d}(t)=\frac{1}{d}\sum_{j=1}^d \lambda_j.\] 
\end{remark}

A fundamental property of Voiculescu's $S$-transform is that it is multiplicative with respect to $\boxtimes$. The analogous property for our finite $S$-transform is an easy consequence of the fact that the coefficients of the polynomial are multiplicative with respect to $\boxtimes_d$.

\begin{proposition}\label{lem:S_multiplicative}
Let $p,q \in \pols_d(\R_{\geq 0})$ with $p(x)\neq x^d\neq q(x)$ and let $r,s$ be the multiplicities of root $0$ in $p,q$, respectively. Then we have
\[
\strans{p\boxtimes_d q}{d} \left(-\frac{k}{d}\right)= \strans{p}{d}  \left(-\frac{k}{d}\right)\strans{q}{d} \left(-\frac{k}{d}\right) \qquad \text{for all }k=1,2\dots ,d-\max\{r,s\}.
\]
\end{proposition}
\begin{proof}
From the definitions of finite $S$-transform \eqref{eq:def.Strans} and finite multiplicative convolution it follows that
\[
\strans{p\boxtimes_d q}{d} \left(-\frac{k}{d}\right)= \frac{\coef{k-1}{d}(p\boxtimes_d q)}{\coef{k}{d}(p\boxtimes_d q)} =\frac{\coef{k-1}{d}(p)\coef{k-1}{d}(q)}{\coef{k}{d}(p)\coef{k}{d}(q)} =\strans{p}{d}  \left(-\frac{k}{d}\right)\strans{q}{d} \left(-\frac{k}{d}\right)
\]
for $k=1,2\dots, d-\max\{r,s\}$.
\end{proof}

Notice that in terms of the finite $T$-transform we do not need to worry about excluding $x^d$ or the multiplicity of the root 0. For $p,q \in \pols_d(\R_{\geq 0})$ one has that 
\begin{equation} \label{eq:rico}
    \ttrans{p\boxtimes_d q}{d} (t)= \ttrans{p}{d}(t)\ttrans{q}{d}(t) \qquad \text{for all }t\in(0,1).
\end{equation}

\begin{remark}
\label{rem:Marcus}
Let us mention that Marcus \cite[Lemma 4.10]{Mar21} defined a modified finite $S$-transform  and proves its multiplicativity, similar to Proposition \ref{lem:S_multiplicative}.
However, while he relates his finite $S$-transform with Voiculescu's $S$-transform, to the best of our knowledge, this cannot be used to connect $\boxtimes_d$ with $\boxtimes$.
\end{remark}

We present here a direct application of the finite $S$-transform (or $T$-transform which is easier to handle in the situation below).
\begin{proposition}
    Let $p, q \in \pols_d (\R_{\ge 0})$ and assume $p \boxtimes_d q(x) = (x-c)^d$ for some $c \in \R_{\ge 0}$.
    If $c=0$, then $p(x) = x^d$ or $q(x) = x^d$.
    If $c>0$, there exist $a, b \in \R_{> 0}$ such that $ab=c$, $p(x) = (x-a)^d$, and $q(x) = (x-b)^d$.
\end{proposition}
\begin{proof}
    If $c=0$, we have $\coef{1}{d}(p) \coef{1}{d}(q) = 0$ and hence $\coef{1}{d}(p)= 0$ or $\coef{1}{d}(q)= 0$.
    Since $p$ and $q$ have only non-negative roots, it means $p(x) = x^d$ or $q(x) = x^d$.

    If $c>0$, we have $\ttrans{p}{d}(t)\ttrans{q}{d}(t) \equiv c$.
    Since the finite $T$-transform is non-negative and weakly increasing, $\ttrans{p}{d}(t) \equiv a$ and $\ttrans{q}{d}(t) \equiv b$ for some $a,b >0$ such that $ab=c$.
    It implies $p(x) = (x-a)^d$ and $q(x) = (x-b)^d$.    
\end{proof}
Note that similar arguments lead to the corresponding result for free multiplicative convolution. Now we show a formula for the reversed polynomial.

\begin{proposition}[Reversed polynomial]
\label{prop:reversed.polynomial}
Given a polynomial $p\in \pols(\rr_{>0})$ and its reversed polynomial $p^\reversed$, their $S$-transforms satisfy the relation
\begin{equation}
\label{eq:Strans.reversed}
\strans{p}{d}\left(-\frac{k}{d}\right) \strans{p^\reversed}{d} \left(-\frac{d+1-k}{d}\right)=1,\qquad \text{for }k=1,\dots, d.
\end{equation}
\end{proposition}

\begin{proof}
By formula \eqref{eq:coef_pinverse}, the coefficients of the reversed polynomial are
\[\coef{d-k}{d}(p^\reversed)
=\frac{\coef{k}{d} (p)} {\coef{d}{d} (p)}.\]
In terms of the $S$-transform, this yields 
\[
\strans{p^\reversed}{d}\left(-\frac{d+1-k}{d}\right) =\frac{\coef{d-k}{d}(p^\reversed)}{\coef{d+1-k}{d}(p^\reversed)}=\frac{\coef{k}{d} (p)} {\coef{k-1}{d} (p)} = \frac{1}{S_p^{(d)}\left(-\frac{k}{d}\right)}
\]
as desired.
\end{proof}
 
We then continue to study some other properties of the finite $S$-transform in connection to the behaviour under taking derivatives and shifts of polynomials. We can study these operations using the Cauchy transform of the empirical root distribution of the polynomials.
These properties can also be understood as finite counterparts of known facts in free probability, specifically Lemma \ref{lem:original_Strans}.

\begin{lemma}[Derivatives and shifts]\label{lem:Strans_conditions}
Consider $p\in \polplus$ and let $r$ be the multiplicity of root $0$ in $p$.
\begin{enumerate}[\rm (1)]
\item For any $k \le l \le d-r$, we have
\[
\strans{p}{d}\left(-\frac{k}{d}\right) = \strans{\diff{l}{d}p}{l}\left(-\frac{k}{l}\right).
\]

\item Given $a>0$, we obtain
\[
\strans{\shift{a}p}{d}\left(-\frac{k}{d}\right) = -G_{\meas{\diff{k}{d}p}}(-a) \qquad \text{for } 1\le k \le d.
\]
If $r=0$, then this can be extended to $a=0$:
\[
\strans{p}{d}\left(-\frac{k}{d}\right) = -G_{\meas{\diff{k}{d}p}}(0).
\]
\end{enumerate}
\end{lemma}

\begin{remark}
The equation in part (1) has the following analogue in free probability
\begin{equation*} \label{eq:sora}
S_{\dil{u}(\mu^{\boxplus 1/u})}\left(-\frac{t}{u}\right) = S_{\mu}(-t) \qquad \text{for }0 < t <u < 1- \mu(\{0\}),
\end{equation*}
that follows easily from Lemma \ref{lem:original_Strans} (1).
\end{remark}

\begin{proof}
Part (1) follows from the fact that coefficients of a polynomial are preserved under differentiation, Lemma \ref{lem:coefficients_derivatives}:
\[
\strans{p}{d}\left(-\frac{k}{d}\right) =\frac{\coef{k-1}{d}(p)}{\coef{k}{d}(p)} = \frac{\coef{k-1}{l}(\diff{l}{d}p)}{\coef{k}{l}(\diff{l}{d}p)} = \strans{\diff{l}{d}p}{l}\left(-\frac{k}{l}\right) \qquad\text{for all }k \leq l \leq d-r.
\]
For part (2) we use that differentiation and shift operator commute $\diff{k}{d} \circ \shift{a}=\shift{a} \circ \diff{k}{d}$, and thus \begin{align*}
\strans{\shift{a}p}{d}\left(-\frac{k}{d}\right)
&= \strans{\diff{k}{d}\shift{a}p}{k}(-1) & \text{(by part (1) above)}\\
&= \strans{\shift{a}\diff{k}{d}p}{k}(-1)\\
&= -G_{\meas{\shift{a}\diff{k}{d}p}} (0) & \text{(by Equation \eqref{eq:Strans.Cauchy})}\\
&=-G_{\meas{\diff{k}{d}p}} (-a).
\end{align*}
Notice that the assumption $a>0$ ensures the polynomial has no roots at $0$ and hence $\strans{\shift{a}\diff{k}{d}p}{k}(-1)$ is well defined. In the case $r=0$, the previous computation holds for $a=0$ as well.
\end{proof}

To finish this section we prove the interesting fact that the partial order studied in Section \ref{sec:partial.order} implies an inequality for the finite $S$-transforms of the polynomial. 

\begin{lemma}
\label{lem:ll.strans.inequality}
Let  $p,q\in \polplus$ with $r$ the multiplicity of the root 0 in $p$. If $p\ll q$ then 
\[
\strans{p}{d}\left(-\frac{k}{d}\right) \geq  \strans{q}{d}\left(-\frac{k}{d}\right)\qquad \text{for all } 1\leq k \leq d-r.
\]
\end{lemma}

\begin{proof}
By Proposition \ref{prop.Phi.preserves.order} we know that the map $\Phi_d$ preserves the partial order $\ll$. Thus we get that $\Phi_d(p)\ll \Phi_d(q)$, and  Using \eqref{eq:phi.roots} we conclude that
\[
\strans{p}{d}\left(-\frac{k}{d}\right) =\frac{\coef{k-1}{d}(p)}{\coef{k}{d}(p)} =\frac{1}{\lambda_k(\Phi_d(p))}\geq \frac{1}{\lambda_k(\Phi_d(q))} =\frac{\coef{k-1}{d}(q)}{\coef{k}{d}(q)} = \strans{q}{d}\left(-\frac{k}{d}\right)
\]
for all $1\leq k \leq d-r$.
\end{proof}

\section{Finite \texorpdfstring{$S$}\ -transform tends to \texorpdfstring{$S$}\ -transform }
\label{sec:finite.S.to.S}

The goal of this section is to prove Theorem \ref{thm:main} advertised in the introduction, which is our main approximation theorem. In order to simplify the presentation, we will make use of the concept of converging sequence of polynomials and diagonal sequence introduced in Notation \ref{not:converging}.
Recall that given a sequence of polynomials $(p_j)_{j\in \N}$ we say $(p_j)_{j\in \N}$ \emph{converges to} $\mu\in\MM$ if
\begin{equation}
\begin{cases}
p_{j}\in \pols_{d_j} \quad \text{for all } j\in\N,\\
\displaystyle \lim_{j\to\infty}d_j=\infty,\\
\meas{p_j}\weak\mu\text{ as }j\to\infty.
\end{cases}
\end{equation}
Furthermore, if $(p_j)_{j\in \N}$ is a sequence of polynomials converging to $\mu$ with degree sequence $(d_j)_{j\in\N}$ we say that $(k_j)_{j\in\nn}$ is a \emph{diagonal sequence with ratio limit $t$} if 
\begin{equation}
\begin{cases}
k_j\in \{1,\dots,d_j\} \text{ for every $j$,}\\
\displaystyle \lim_{j\to\infty} \frac{k_j}{d_j}=t. 
\end{cases}
\end{equation}

With this new terminology, Theorem \ref{thm:main} can be rephrased as follows. Given a measure $\mu\in\MM(\rr_{\geq 0})\setminus\{\delta_0\}$, 
the following are equivalent:
\begin{enumerate}
    \item The sequence of polynomials  $(p_j)_{j\in\nn}\subset \pols(\rr_{\geq 0})$ converges to $\mu$.
    \item For every diagonal sequence $(k_j)_{j\in\nn}$ with ratio limit $t\in(0,1-\mu(\{0\}))$, it holds that
\[\lim_{j\to\infty}\strans{p_j}{d_j}\left(-\frac{k_j}{d_j}\right)=S_\mu(-t).\]
\end{enumerate}

The following lemma relating the $S$-transform of a measure evaluated at $t$ with the Cauchy transform of the $t$-fractional convolution evaluated at $0$ will be useful throughout this section.

\begin{lemma} \label{lem:approx_box_times} Let $\mu$ be measure on $\mu \in \MM(\R_{\geq 0})\setminus\{ \delta_0\}$.    Then, 
$$S_\mu(-t) = - G_{\dil{t}(\mu^{\boxplus 1/t})}(0)\qquad \text{for }t\in (0,1-\mu(\{0\})).$$
As a consequence,  $S_{\mu\boxplus \delta_a}(-t)=-G_{\dil{t}(\mu^{\boxplus 1/t})}(-a)$ for all $t\in (0,1)$ and $a>0$.
\end{lemma}

\begin{proof}
Fix $t\in (0,1-\mu(\{0\}))$, and consider the measure $\mu_t:=\dil{t}(\mu^{\boxplus 1/t})$.
By Corollary \ref{cor. supporfreeconv}, $\mu_t$ is supported on $[\epsilon,\infty)$ for some $\epsilon>0$.

Then, the Cauchy transform is well-defined in $\mathbb{C}\setminus [\epsilon,\infty)$, in particular, we know that 
\[
G_{\mu_t}(0)=\int^\infty_0 x^{-1} \mu_t(dx) \le \frac{1}{\varepsilon} <  \infty.
\] 
On the other hand, since $\mu_t(\{0\})=0$, \cite[Lemma 4]{HM13} yields
\begin{equation} \label{eq. StoG1} \lim_{z\to 1} S_{\mu_t}(-z)=\int^\infty_0 x^{-1} \mu_t(dx)= G_{\mu_t}(0).\end{equation}

Finally, using relation $S_\mu(-tz)= S_{\mu_t}(-z)$ from Lemma \ref{lem:original_Strans} (1) and the continuity of the $S$-transform on $(0,1-\mu(\{0\}))$ we get
\begin{equation}\label{eq. StoG2} S_\mu(-t)=\lim_{z \to 1} S_\mu(-zt) = \lim_{z \to 1} S_{\mu_t}(-z).\end{equation}
Putting \eqref{eq. StoG1} and \eqref{eq. StoG2} together, we obtain the desired result.
\end{proof}

In its simplest form, when $ (p_j)_{j\in\nn}\subset \pols(C)$ and $C=[\alpha,\beta]\subset (0,\infty)$ is a compact interval that does not contain 0, the proof of $(1)\Rightarrow(2)$ follows naturally from the basic properties of the $S$-transform. However, the same proof in the most general case, requires of several steps, where we gradually generalize each result, building upon simpler cases. On the other hand, the proof of $(2)\Rightarrow (1)$ in the general case follows from using Helly's selection Theorem to guarantee that there exist a limit, and then using the implication $(1)\Rightarrow (2)$ to assure that the limit must coincide with the given measure. 

To guide the reader on how our claim is generalized in each step, this section is divided into several cases.
Each case builds upon the previous until we reach the most general case. 
In Section \ref{ssec:compac.interval.not.zero} we illustrate the simplicity of the ideas used in the case where all the roots of the polynomials lie on a compact interval that does not contain 0. Then, in Section \ref{ssec:compac.interval.zero} we use a uniform bound of the smallest root after differentiation, to reduce the case when all the roots of the polynomials lie on a compact interval that contains 0, to the previous case (where the interval does not contain 0). In Section \ref{ssec:case.unbounded.support} we use the reversed polynomial of a shift, to generalize Theorem \ref{Thm:AGVP_HK_S} to allow measures with unbounded support, then the same is done for our main result, using the same ideas from the previous section. Finally, in \ref{ssec:generalcase} we explain how to obtain the converse statement and  prove Theorem \ref{thm:main}.

\subsection{Compact interval not containing 0.}
\label{ssec:compac.interval.not.zero}

We first prove it for the case where all the roots of the polynomials lie on a compact interval that does not contain zero.
In this case, the proof of the first implication follows easily from the basic properties of the $S$-transform.

\begin{proposition}[Compact interval not containing 0]
\label{prop:Strans.convergence.compact}
Fix a compact interval $C=[\alpha,\beta] \subset (0,\infty)$. If $(p_j)_{j\in \N}  \subset \pols(C)$ is a sequence of polynomials converging to $\mu\in\MM(C)$, then for every diagonal sequence $(k_j)_{j\in \N}$ with ratio limit $t\in (0,1)$ it holds that
\[
\lim_{j\to \infty} \strans{p_j}{d_j}\left(-\frac{k_j}{d_j}\right) = S_\mu(-t).
\]
\end{proposition}

\begin{proof}
By Lemma \ref{lem:Strans_conditions} (2) we know that
\begin{align*}
\strans{p_j}{d_j}\left(-\frac{k_j}{d_j}\right)  = - G_{\meas{\diff{k_j}{d_j}p_j}}(0).
\end{align*}
Since $\meas{p_j}\weak \mu$, then Theorem \ref{Thm:AGVP_HK_S} implies that $\meas{\diff{k_j}{d_j}p_j}\weak \dil{t}{(\mu^{\boxplus 1/t})}$. 

Using Lemma \ref{lem:weak_convergence_Cauchy} and Lemma \ref{lem:approx_box_times}, we conclude that 
\[
\lim_{j\to\infty}\strans{p_j}{d_j}\left(-\frac{k_j}{d_j}\right)= -\lim_{j\to\infty} G_{\meas{\diff{k_j}{d_j}p_j}}(0) = -G_{\dil{t}{(\mu^{\boxplus 1/t})}}(0) = S_\mu(-t)\]
for every $t\in (0,1)$, as desired.
\end{proof}

\subsection{Compact interval containing 0.}
\label{ssec:compac.interval.zero}

The goal of this section is to extend Proposition \ref{prop:Strans.convergence.compact} to the case where the interval is allowed to contain $0$.
The approach is to reduce to the previous case by observing that after repeatedly differentiating polynomials with non-negative roots, we can find a uniform lower bound (away from zero) of the smallest roots. To achieve this we make use of the bounds obtained in Section \ref{sec:bounds.Jacobi} (that in turn rely on classical bounds of Jacobi polynomials), and then we use the partial order from Section \ref{sec:partial.order} to extrapolate this bound to an arbitrary polynomial. 

\begin{proposition}[Compact support containing 0]
\label{prop:Strans_convergence_compact_zero}
Fix a compact interval $C=[0,\beta]$ and let $(p_j)_{j\in \N} \subset \pols(C)$ be sequence of polynomials converging to $\mu\in\MM(C)\setminus\{\delta_0\}$. Then for every diagonal sequence $(k_j)_{j\in \N}$ with ratio limit $t\in (0,1-\mu(\{0\}))$ it holds that
\[
\lim_{j\to \infty} \strans{p_j}{d_j}\left(-\frac{k_j}{d_j}\right) = S_\mu(-t).
\]
\end{proposition}

\begin{proof}
  Since $\meas{p_j}\weak \mu$, Theorem \ref{Thm:AGVP_HK_S} implies that $\meas{\diff{k_j}{d_j}p_j}\weak \dil{t}{(\mu^{\boxplus 1/t})}$.
  Also, by Lemma \ref{lem:bound}, there exists some \(\epsilon>0\) such that \(\meas{\diff{k_j}{d_j}p_j} \in \MM(\R_{\ge \epsilon})\) for large enough $j$.
  Hence, by Lemmas \ref{lem:Strans_conditions} (2) and \ref{lem:weak_convergence_Cauchy}, 
  we conclude that 
  $$
  \lim_{j\to\infty}\strans{p_j}{d_j}\left(-\frac{k_j}{d_j}\right)= -\lim_{j\to\infty} G_{\meas{\diff{k_j}{d_j}p_j}}(0) = -G_{\dil{t}{(\mu^{\boxplus 1/t})}}(0) = S_\mu(-t).$$
  \end{proof}

\begin{remark}
The authors thank Jorge Garza-Vargas for some insightful discussions regarding relations between the supports of $\mu^{\boxplus t}$ and of $\meas{\diff{k_j}{d_j}p_j}$ that ultimately helped in the proof of Proposition \ref{prop:Strans_convergence_compact_zero}. 
\end{remark}

\subsection{Case with unbounded support}
\label{ssec:case.unbounded.support}

We now turn to the study of measures with unbounded support. The key idea here is that if $\mu\in\MM(\rr_{\geq 0})$, then if we shift $\mu$ by a positive constant $a$ and then take the reversed measure, we obtain a measure contained in a compact interval.
In this way we can reduce the problem to the previous case.
First, we introduce the cut-down and cut-up measures. 

\begin{notation}[Cut-up and cut-down measures]
\label{not:cut.operators}
Given a measure $\mu\in \MM(\R)$ with cumulative distribution function $F_\mu$, we define the cut-down measure at $a\in\rr$ as the measure $\mu|_a\in\MM(\rr_{\geq a})$ with cumulative distribution function
    \[
        F_{\mu|_a}(x) = 
        \begin{cases}
        0 & \text{if }x<a,\\
        F_{\mu}(x) & \text{if }x \ge a.
        \end{cases}
    \]
Similarly, we define the cut-up measure at $a\in\R$ as the measure $\mu|^a\in\MM(\rr_{\leq a})$ with cumulative distribution function
    \[
        F_{\mu|^a}(x) = \begin{cases}
        F_{\mu}(x) & \text{if }x < a,\\
        1 & \text{if }x \geq a,
    \end{cases}
    \]
We can define the corresponding cut-down and cut-up measures on polynomials using the bijection $\meas{\cdot}$ between $\PP_d(\R)$ and $\MM_d(\R)$. Then 
for every $p\in \PP_d(\R)$ and $a\in \rr$ the cut-down polynomial $p|_a\in\PP_d(\R_{\geq a})$ and cut-up polynomial $p|^a\in \PP_d(\R_{\leq a})$ have roots given by 
\begin{equation*}
\lambda_i(p|_a) = 
\begin{cases}
    a & \text{if }\lambda_i(p) \leq a,\\
    \lambda_i(p) & \text{if }\lambda_i(p) > a
\end{cases}    
\qquad \text{and} \qquad  
\lambda_i(p|^a) = 
\begin{cases}
    \lambda_i(p) & \text{if }\lambda_i(p) < a,\\
    a & \text{if }\lambda_i(p) \geq a.
\end{cases} 
\end{equation*}
\end{notation}

\begin{remark}
\label{rem:basic.prop.cut}
We will use three basic properties of the cut-down and cut-up measures that follow directly from the definition. 
\begin{enumerate}
    \item $\mu|_a\weak \mu$ as $a\to -\infty$ and $\mu|^a\weak \mu$ as $a\to \infty$.
    \item $F_{\mu|_a} \leq F_{\mu}$, so $\mu \ll (\mu|_a)$. And similarly, $(\mu|^a) \ll \mu$.
    \item If $(p_j)_{j\in\nn}$ is a sequence of polynomials converging to $\mu$ then $(p_j|_a)_{j\in\nn}$ converges to $\mu|_a$, and similarly $(p_j|^a)_{j\in\nn}$ converges to $\mu|^a$.
\end{enumerate} 
\end{remark}

We are now ready to extend Theorem \ref{Thm:AGVP_HK_S} on the limits of derivatives of polynomials to the case of measures with unbounded support.

\begin{theorem}
\label{thm:general.AGVP}
Let $(p_j)_{j\in \N}\subset\pols(\rr)$ be a sequence of polynomials converging to $\mu \in \MM(\R)$ and let $(k_j)_{d\in \N}\subset\nn$ be diagonal sequence with limit $t \in (0, 1)$. Then, 
\begin{equation}
\label{eq.claim.AGVP}
\meas{\diff{k_j}{d_j} p_j} \weak \dil{t}(\mu^{\boxplus 1/t})\qquad \text{ as }\qquad j\to\infty.    
\end{equation}
\end{theorem}

\begin{proof}
We first prove the claim assuming that $(p_j)_{j\in \N}\subset\pols(\rr_{\geq L})$ for some $L\in\rr$.
We fix $a\in\rr$ such that $a+L>0$, so that $\shift{a}p_j \in \pols_{d_j}([a+L,\infty))$. Thus, we can consider $q_j:=(\shift{a}p_j)^\reversed \subset \pols_{d
_j}\left(\left(0,\tfrac{1}{a+L}\right]\right)$. By Proposition \ref{prop:reversed.polynomial} and Lemma \ref{lem:Strans_conditions} (2) we know that
\[
-G_{\meas{\diff{k_j}{d_j}p_j}}(-a)=\strans{\shift{a}p_j}{d_j}\left(-\frac{k_j}{d_j}\right)  =\frac{1}{\strans{q_j}{d_j}\left(-\frac{d_j+1-k_j}{d_j}\right) }.
\]
On the other hand, since the roots of the polynomials $(q_j)_{j\in\nn}$ are contained in a compact interval and $\meas{q_i}$ converges weakly to $(\mu\boxplus\delta_a)^\reversed$, Proposition \ref{prop:Strans_convergence_compact_zero} yields
\[
\lim_{j\to\infty} \frac{1}{\strans{q_j}{d_j}\left(-\frac{d_j+1-k_j}{d_j}\right) } =\frac{1}{S_{(\mu \boxplus \delta_a)^\reversed}(t-1)}=S_{\mu \boxplus \delta_a}(-t) = -G_{\dil{t}(\mu^{\boxplus{1/t}})}(-a).
\]
Therefore,
\[
\lim_{j\to\infty} G_{\meas{\diff{k_j}{d_j}p_j}}(-a)=G_{\dil{t}(\mu^{\boxplus{1/t}})}(-a), \qquad \text{for all } -a < L
\]
and by Lemma \ref{lem:weak_convergence_Cauchy} we conclude that $\meas{\diff{k_j}{d_j}p_j}\weak \dil{t}(\mu^{\boxplus{1/t}})$ as $j\to \infty$.
Thus, the claim is proved in the case  $(p_j)_{j\in \N}\subset\pols(\rr_{\geq L})$ for some $L\in\rr$. 

Notice that the claim is also true if $(p_j)_{j\in \N}\subset\pols(\rr_{\leq L})$ for some $L\in\rr$. Indeed, we simply apply the reflection map $\dil{-1}$ to the polynomials, and use the previous case on the new sequence.

Now, the proof of the general case, when $(p_j)_{j\in \N}\subset \pols(\rr)$, follows from using the cut-down and cut-up measures from Notation \ref{not:cut.operators} to reduce the problem to the previous case. Indeed, if we fix an $L\in\nn$ then from Remark \ref{rem:basic.prop.cut} (2) and Corollary \ref{cor:derivatives.monotonicity} we know that 
\[\left(\diff{k_j}{d_j}\left.p_j\right|^L\right) \ll \left(\diff{k_j}{d_j} p_j\right) \ll \left(\diff{k_j}{d_j}p_j|_{-L}\right)\qquad \text{for every }j\in\nn.\]
By Remark \ref{rem:basic.prop.cut} (3), as $j\to\infty$ we have the weak convergence 
\[\meas{\diff{k_j}{d_j}p_j|_{-L}} \weak \dil{t}(\mu|_{-L})^{\boxplus 1/t}\qquad\text{and} \qquad \meas{\diff{k_j}{d_j}p_j|^{L}} \weak \dil{t}(\mu|^{L})^{\boxplus 1/t}.\]
Therefore, at every point $x\in \rr$ one has
\[
    F_{\dil{t}(\mu|_{-L})^{\boxplus 1/t}} (x)
    \le \liminf F_{\meas{\diff{k_j}{d_j} p_j}}(x) 
    \le \limsup F_{\meas{\diff{k_j}{d_j} p_j}}(x) 
    \le F_{\dil{t}(\mu|^L)^{\boxplus 1/t}}(x).
\]
  Letting $L \to \infty$, we conclude that
$\displaystyle
\lim_{j\to\infty} F_{\meas{\diff{k_j}{d_j} p_j}}(x)= F_{\dil{t}(\mu^{\boxplus 1/t})}(x)$ for every continuous point  $x\in\rr$ of $\dil{t}(\mu^{\boxplus 1/t})$ and this is equivalent to \eqref{eq.claim.AGVP}.
\end{proof}

With this theorem, we finally upgrade our main result to measures with unbounded support.

\begin{proposition}[Unbounded support]
\label{prop:Strans.unbounded} 
Let $(p_j)_{j\in \N}\subset \pols(\rr_{\geq 0})$ be a sequence of polynomials converging to $\mu\in\MM(\rr_{\geq 0})\setminus\{\delta_0\}$. Then for every diagonal sequence $(k_j)_{j\in \N}$ with limit $t\in (0,1-\mu(\{0\}))$ it holds that
\[
\lim_{j\to \infty} \strans{p_j}{d_j}\left(-\frac{k_j}{d_j}\right) = S_\mu(-t).
\]
\end{proposition}

\begin{proof}
The proof is almost identical to the proof of Proposition \ref{prop:Strans_convergence_compact_zero}, except that Theorem \ref{Thm:AGVP_HK_S} is replaced by Theorem \ref{thm:general.AGVP}.
\end{proof}

\subsection{The converse and proof of Theorem \ref{thm:main}}
\label{ssec:generalcase}

For the proof of our main theorem to be complete, we must prove that if the finite $S$-transform converges to the $S$-transform of a measure, then the sequence of polynomials converge to the measure.

\begin{proposition}[Converse]
\label{prop:converse.main.thm}  
Let $(p_j)_{j\in \N}$ be a sequence of polynomials with $p_j\in \pols_{d_j}(\rr_{\geq 0})$. Assume there exist a $t_0 \in [0,1)$ and a function $S:(-1+t_0,0) \to \rr_{>0}$ such that for every $t\in(0,1-t_0)$ and  every diagonal sequence $(k_j)_{j\in \N}$ with limit $t$ one has 
\[
\lim_{\substack{j\to \infty}}\strans{p_j}{d_j}\left(-\frac{k_j}{d_j}\right)= S(-t).
\]
Then there exists a measure $\mu \in \MM(\R_{\ge 0})$ such that $\meas{p_d} \weak \mu$, $\mu(\{0\})\leq t_0$, and $S_\mu(-t) = S(-t)$ for all $t \in (0, 1 - t_0)$.
\end{proposition}

\begin{proof}
Consider the sequence of cumulative distribution functions $(F_{\meas{p_j}})_{j\in\nn}$.
By Helly's selection, every subsequence of functions has a further subsequence, denoted by $(F_{i})_{i\in\N}$, converging to some function $F$.
It is clear that $F$ is non-decreasing, $F$ is equal to $0$ on the negative real line, and the image of $F$ is contained in $[0,1]$.
In order to justify that $F$ is the cumulative distribution function of some probability measure $\mu$, we just need to check that $\lim_{x\to\infty} F(x)=1$.
For the sake of contradiction, we assume that there exists $u\in(0,1)$ such that $F(x) < 1 - u$ for all $x\in\R$. Let $a>0$ and $l_j:=\lfloor u d_j\rfloor$ so that $\lim_{j\to\infty}\frac{l_j}{d_j}=u$. Then the sequence of polynomials $q_j:= x^{d_j-l_j}(x-a)^{l_j}$ satisfies that $q_j\ll p_j$ for large enough $j$. By Lemma \ref{lem:ll.strans.inequality}, this implies
\[
 \strans{q_j}{d} \left(-\frac{k_j}{d_j}\right)  \geq   \strans{p_j}{d} \left(-\frac{k_j}{d_j}\right).
\]
Since the $q_j$ converge to $\dil{a}{\nu} = \delta_a \boxtimes \nu$ where $\nu=(1-u) \delta_0 + u\delta_1$. Then, in the limit we obtain 
  \[
    \frac{S_{\nu}(-t)}{a} \geq S(-t) >0
  \]
On the other hand, we can choose $a$ arbitrarily large so that $\frac{S_{\nu}(-t)}{a}< S(-t)$, which yields a contradiction.

Therefore, $F=F_\mu$ is the cumulative distribution function of some probability measure $\mu$. By Proposition \ref{prop:Strans.unbounded} we obtain that $S_\mu(-t) = S(-t)$ for all $t \in (0, \min\{1-\mu(\{0\}),t_0\})$.

Recall that this $\mu$ was obtained as the convergent subsequence of an arbitrary subsequence of the original sequence of polynomials. However, since the $S$-transforms of any two such measures coincide in a small neighborhood $(-\epsilon, 0)$, the measures are the same. Thus, there is a unique limiting measure and we conclude that $\meas{p_d} \weak \mu$.
\end{proof}

The proof of the main Theorem \ref{thm:main} now follows from the previous two results.

\begin{proof}[Proof of Theorem \ref{thm:main}]

(1) $\Rightarrow$ (2). This implication follows from Proposition \ref{prop:Strans.unbounded}.

(2) $\Rightarrow$ (1). This implication is a particular case of Proposition \ref{prop:converse.main.thm}.
\end{proof}

\section{Symmetric and unitary case}
\label{sec:symmetric.and.unitary}

\subsection{Symmetric case} \label{sec:symmetric}

We say that a probability measure $\mu\in \MM(\rr)$ is \emph{symmetric} if $\mu(-B) = \mu(B)$ for all Borel sets $B$ of $\R$, we denote by $\MM^S(\rr)$ the set of symmetric probability measures on the real line. There is a natural bijection from $\MM^S(\rr)$ to $\MM(\rr_{\geq 0})$ by taking the square of the measure. Specifically, we denote by $\mathbf{Sq}(\mu) \in \MM(\R_{\ge 0})$ the pushforward of $\mu$ by the map $x \mapsto x^2$ for $x\in\R$.
Arizmendi and Pérez-Abreu \cite{arizmendi2009} used this map to extend the 
definition of $S$-transform to symmetric measures:
\begin{equation} \label{eq:symmetric_S}
    \widetilde{S}_{\mu}(z):= \sqrt{\frac{1+z}{z}S_{\mathbf{Sq}(\mu) }(z)} \qquad \text{for } z\text{ in some neighborhood of 0}.
\end{equation}
A similar approach works to define a finite $S$-transform for symmetric polynomials. We say that $p\in \pp_{2d}(\rr)$ is a \emph{symmetric polynomial} if its roots are of the form:
\[\lambda_1(p)\geq \lambda_2(p)\geq \dots \geq \lambda_d(p)\geq 0\geq -\lambda_d(p)\geq \dots \geq -\lambda_2(p)\geq -\lambda_1(p),\]
and denote by $\pols_{2d}^S(\rr)$ the subset of symmetric polynomials. Given $p\in\pols_{2d}^S(\rr)$ we denote by $\mathbf{Sq}(p)\in\pols_d(\rr_{\geq 0})$ the polynomial with roots
\[(\lambda_1(p))^2\geq (\lambda_2(p))^2\geq \dots \geq (\lambda_d(p))^2\geq 0.\]

It is readily seen that $\mathbf{Sq}\left(\meas{p}\right)=\meas{\mathbf{Sq}(p)}$. Moreover, $p$ and $\mathbf{Sq}(p)$ are easily related by the formula
\[\mathbf{Sq}(p)(x^2)=p(x).\]
In particular,
\begin{equation} \label{eq:sugar}
\binom{2d}{2k}\widetilde{\mathsf{e}}_{2k}(p) =  (-1)^k \binom{d}{k}\widetilde{\mathsf{e}}_{k}(\mathbf{Sq}(p))\qquad \text{for }k=1,\dots, d.
\end{equation}

With this in hand, we can extend our definition.

\begin{definition}[$S$-transform for symmetric polynomials]
Let $p\in \PP_{2d}^S
(\rr)$ with a multiplicity of $2r$ in the root 0. We define its finite $S$-transform map
\[\stranstilde{p}{2d} : \left\{\left.-\frac{k}{d}\ \right|\ k =1,2,\dots, d-r \right\} \to (\sqrt{-1})\R_{\geq 0}\]
such that
\[\stranstilde{p}{2d}\left(-\frac{k}{d}\right):=
\sqrt{\frac{\coef{2(k-1)}{2d}(p)}{\coef{2k}{2d}(p)}} \qquad \text{for }k=1,\dots, d-r.\]
\end{definition}

\begin{remark}
Notice that the $S$-transform is well defined because $\coef{2k}{2d}(p)\neq 0$ for $k=1,\dots,d-r$. Moreover, from \eqref{eq:sugar} it follows that $\frac{\coef{2(k-1)}{2d}(p)}{\coef{2k}{2d}(p)}<0$ for $k=1,\dots,d-r$. Thus, the image is actually contained in the positive imaginary line $(\sqrt{-1})\R_{\geq 0}$. Using \eqref{eq:sugar}, one can also verify that
\begin{equation} \label{eq:symmetric.finite.S}
\stranstilde{p}{2d}\left(-\frac{k}{d}\right)= \sqrt{\frac{1-\frac{k}{d}+\frac{1}{2d}}{-\frac{k}{d}+\frac{1}{2d}} \strans{\mathbf{Sq}(p) }{d}\left(-\frac{k}{d}\right)}
\end{equation}
\end{remark}

We can use this new transform to study the multiplicative convolution of polynomials, one of which is symmetric.

\begin{proposition}
\label{prop:stransform_multiplicative_symmetricpols}
Let $p\in\pols_{2d}^S(\rr)$, $q\in\pols_{2d}(\rr_{\geq 0})$ and let $r$ be the maximum of the multiplicities at the root 0 of $p$ and $q$. Then 
\[\left(\stranstilde{p\boxtimes_{2d} q}{2d}\left(-\frac{k}{d}\right)\right)^2=\left(\stranstilde{p}{2d} \left(-\frac{k}{d}\right) \right)^2\strans{q}{2d}\left(-\frac{2k}{2d}\right) \strans{q}{2d}\left(-\frac{2k-1}{2d}\right)\] 
for $k = 1,2,\dots,d-r$.
\end{proposition}

\begin{proof}
Using the definition, we compute
\begin{align*}
 \left(\stranstilde{p\boxtimes_{2d} q}{2d}\left(-\frac{k}{d}\right)\right)^2 &= \frac{\coef{2(k-1)}{2d}(p\boxtimes_{2d} q)}{\coef{2k}{2d}(p\boxtimes_{2d} q)} \\
 &= \frac{\coef{2k-2}{2d}(p)}{\coef{2k}{2d}(p)}\cdot\frac{\coef{2k-2}{2d}(q)}{\coef{2k}{2d}(q)}\\
 &= \frac{\coef{2k-2}{2d}(p)}{\coef{2k}{2d}(p)}\cdot\frac{\coef{2k-2}{2d}(q)\coef{2k-1}{2d}(q)}{\coef{2k}{2d}(q)\coef{2k-1}{2d}(q)}\\
 &=\left(\stranstilde{p}{2d} \left(-\frac{k}{d}\right) \right)^2\strans{q}{2d}\left(-\frac{2k}{2d}\right) \strans{q}{2d}\left(-\frac{2k-1}{2d}\right).
\end{align*} 
\end{proof}

It is also easy to check that our finite symmetric $S$-tranform tends to the symmetric $S$-transform from \cite{arizmendi2009} in the limit.

\begin{proposition}
\label{prop:Strans.symmetric.convergence}
Let $(p_j)_{j\in\nn}\subset \pols^S(\rr)$ be a sequence of symmetric polynomials with degree sequence $(2d_j)_{j\in\nn}$ and assume $(p_j)_{j\in\nn}$ converges to $\mu\in \MM^S(\rr)$. Then for every diagonal sequence $(2k_j)_{j\in\nn}\subset\nn$ with limit $t\in(0,1-\mu(\{0\}))$, it holds that
\[\lim_{j\to\infty}\stranstilde{p_j}{2d_j}\left(-\frac{k_j}{d_j}\right)= \widetilde{S}_\mu(-t).\]
\end{proposition}
\begin{proof}
Using Theorem \ref{thm:main} with the sequence $\left(\mathbf{Sq}(p_j)\right)_{j\in\nn}$, Equation \eqref{eq:symmetric.finite.S} tends to Equation \eqref{eq:symmetric_S} in the limit: 
\[\lim_{j\to\infty}\stranstilde{p_j}{2d_j}\left(-\frac{k_j}{d_j}\right)=  \lim_{j\to\infty}\sqrt{\frac{1-\frac{k_j}{d_j}+\frac{1}{2d_j}}{-\frac{k_j}{d_j}+\frac{1}{2d_j}} \strans{\mathbf{Sq}(p_j)}{d}\left(-\frac{k}{d}\right)}= \sqrt{\frac{1-t}{-t} S_{\mathbf{Sq}(\mu) }(-t)}=\widetilde{S}_\mu(-t).\]
\end{proof}

\subsection{Unitary case}

It would be interesting to construct an $S$-transform that can handle the set $\pols_d(\mathbb{T})$ of polynomials with roots in the unit circle $\tt:=\{z\in\cc : |z|=1\}$. However, if we naively try to apply the same approach used in the previous cases, we run into some problems.
To illustrate such difficulties, let us consider the following example. When considering polynomials that resemble the Haar unitary measure in $\cc$, namely polynomials with roots uniformly distributed in $\tt$, there are at least two natural candidates:
\[
  h_d(x) = x^d - 1 = \prod_{k=1}^d (x - e^{\frac{2\pi i k}{d}}),
\]
\[
   \widehat{h}_{d-1}(x) = \frac{x^d-1}{x-1} = \sum_{j=0}^{d-1} x^j=\prod_{k=1}^{d-1} (x - e^{\frac{2\pi i k}{d}}).
\]
Notice that $h_d$ has the same roots as $\widehat{h}_{d-1}$ with an extra root in 1.
Thus, when $d\to\infty$, the empirical distributions of $h_d$ and $\widehat{h}_{d-1}$ both tend to $\chi$, the uniform distribution on $\mathbb{T}$.

For $h_d$, the only non-vanishing coefficients are $\coef{0}{d}\left(h_d\right)$ and $\coef{d}{d}\left(h_d\right)$.
Thus, our method of looking at the quotient of coefficients $\{\coef{k}{d}\left(h_d\right)\}_{k=0}^n$ does not work at all simply because all the quotients are undefined.
On the other hand, for $\widehat{h}_d$,
\[
\coef{k}{d}(\widehat{h}_d) = (-1)^k\binom{d}{k}^{-1}.
\]
Thus, their ratio limit is
\[
\frac{\coef{k-1}{d}(\widehat{h}_d)}{\coef{k}{d}(\widehat{h}_d)} =-\frac{\binom{d}{k}}{\binom{d}{k-1}} = -\frac{d-k+1}{k} \to -\frac{1-t}{t},
\]
as $d\to \infty$ with $k/d \to t$. On the other hand, since $m_1(\chi)=0$, the $S$-transform of $\chi$ cannot be defined and thus there is no relation to the last limit.

Even though our approach does not seem to work every sequence of polynomials contained in $\pp(\tt)$, in some cases we do obtain the expected limit. For instance, fix $t\geq 0$ and consider the unitary Hermite polynomials 
$$H_d(z;t) = \sum_{k=0}^d (-1)^k\binom{d}{k}\exp\left(-\frac{tk(d-k)}{2d}\right)$$ that where studied in \cite[Section 6]{arizmendifujieueda} and \cite{kabluchko2022lee}. Then, if we take the ratio of consecutive coefficients and take the corresponding diagonal limit approaching $t\in(0,1)$, we obtain the $S$-transform of $\sigma_t$, the free normal distribution on $\tt$:
$$S_{\sigma_t}(z) = \exp\left(t\left(z+\frac{1}{2}\right) \right).$$

We can also prove that $\boxtimes_d$ approaches to $\boxtimes$ as $d\to\infty$ on the unitary case without using the $S$-transform, to the best our knowledge, there is no literature which explicitly states the assertion:
\begin{proposition}
    Let $p_d, q_d \in \pols_d(\mathbb{T})$ for $d \in \N$ and $\mu, \nu \in \MM(\mathbb{T})$.
    If $\meas{p_d} \weak \mu$ and $\meas{q_d} \weak \nu$ as $d\to \infty$, repectively, then $\meas{p_d \boxtimes_d q_d} \weak \mu \boxtimes \nu$.
\end{proposition}

Since the proof of this result is very similar to the proof in the real case (see Propositions \ref{prop:convergence_cumulants} and \ref{prop:finiteAsymptotics}), we only provide the idea of the proof. For the details, we refer the reader to \cite[proof of Corollary 5.5]{AP}.

\begin{proof}[Idea of the proof]
Since $\mathbb{T}$ is compact, then the moment convergence of polynomial sequence $(p_d)_{d\in\N} \subset \pols(\mathbb{T})$ is equivalent to weak convergence.
The proof then follows from the equivalence of the convergences of moments and of finite free cumulants of $(p_d)$, which is similar to how it is done in the real case. 
\end{proof}

\section{Approximation of Tucci, Haagerup, and M\"{o}ller}
\label{sec:approx.THM}

The purpose of this section is to prove Theorem \ref{thm:main3} stating that Fujie and Ueda's limit theorem \cite{fujie2023law}  is an approximation of Tucci, Haagerup and M\"{o}ller's limit theorem \cite{tucci2010,HM13}.

The main idea is that this approximation is equivalent to the convergence of the finite $T$-transform from the Section \ref{sec:finite.Strans} to the $T$-transform introduced in \eqref{eq:T.transform.def}. This in turn is almost equivalent to the convergence of the finite $S$-transform to the $S$-transform, except that the $T$-transform is more adequate to handle the case where the polynomial has roots in 0. Notice that we also need include the case of $\delta_0$. First, we will adapt Theorem \ref{thm:main} to a version with $T$-transforms.

\begin{theorem} \label{thm:main.T.transform}
Given a measure $\mu\in\MM(\rr_{\geq 0})$ and a sequence of polynomials $(p_j)_{j\in\nn}\subset \pols(\rr_{\geq 0})$,
the following are equivalent:
\begin{enumerate}
\item The weak convergence of $(p_j)_{j\in\nn}$ to $\mu$.
\item For every diagonal sequence $(k_j)_{j\in\nn}$ with limit $t\in(0,1)$, it holds that
\[\lim_{j\to\infty}\ttrans{p_j}{d_j}\left(\frac{k_j}{d_j}\right)=T_\mu(t).\]
\item For every $t\in(0,1)$, it holds that
\[\lim_{j\to\infty}\ttrans{p_j}{d_j}(t)=T_\mu(t).\]
\end{enumerate}
\end{theorem}

\begin{proof}
Using the definition of $T$-transform in terms of $S$-transform from \eqref{eq:T.transform.def}, its finite analogue \eqref{eq:relation.T.S.finite} and our main Theorem \ref{thm:main}, we obtain that $\meas{p_j} \weak \mu$ with $\mu\neq\delta_0$, if and only if for every diagonal sequence $(k_j)_{j\in\nn}$ with limit $t\in(\mu(\{0\}),1)$ one has that
\[\lim_{j\to\infty} \ttrans{p_j}{d_j}\left( \frac{k_j}{d_j}\right)=\lim_{j\to\infty} \frac{1}{\strans{p_j}{d_j}\left( \tfrac{k_j-d_j}{d_j}\right) }= \frac{1}{S_{\mu}(t-1)}=T_\mu(t).\]
Since the functions $\ttrans{p_j}{d_j}$ are positive and non-decreasing (see Remark \ref{rem.finite.T.trans.basic.facts}) and $T_\mu$ is positive, continuous and increasing, then the previous limit can be extended to the whole interval $(0,1)$, and this is equivalent to part (2).

The equivalence between (2) and (3) follows by the increasing property of the $T$-transform. Indeed,  for any $t\in(0,1)$ one can find diagonal sequences $(k_j)_j$ and $(k'_j)_j$ with $k_j/d_j \leq t \leq  k'_j/d_j$ with limit $t$  and then
\[T_\mu(t)=\lim_{j\to\infty}\ttrans{p_j}{d_j}\left(\frac{k_j}{d_j}\right)\leq \lim_{j\to\infty}\ttrans{p_j}{d_j}\left(t \right)\leq \lim_{j\to\infty}\ttrans{p_j}{d_j}\left(\frac{k_j'}{d_j}\right)=T_\mu(t). \]
The converse statement follows by a similar argument.

Therefore, we are only left to check what happens when $\mu=\delta_0$ and $T_\mu(t) = 0$ for $t\in (0,1)$.
If $(p_j)_{j\in\N}$ converges to $\delta_0$, it is clear $\meas{\shift{\epsilon}(p_j)} \xrightarrow{w} \delta_\epsilon$ as $j\to\infty$ for any $\epsilon>0$.
Thus, 
$$ \limsup_{j\to\infty}  \ttrans{p_j}{d_j}(t) \le \epsilon \qquad \text{for }t\in (0,1)$$
because $\displaystyle \lim_{j\to\infty} \ttrans{\shift{\epsilon}(p_j)}{d_j}(t) = T_{\delta_\epsilon}(t) = \epsilon$ and $\ttrans{p_j}{d_j}(t) \le \ttrans{\shift{\epsilon}(p_j)}{d_j}(t)$ by Proposition \ref{prop.Phi.preserves.order}.
Letting $\epsilon$ tend to 
$0$  we obtain $\displaystyle \lim_{{j \to \infty}} \ttrans{p_j}{d_j}(t) = 0$.

For the converse, assume that $\displaystyle \lim_{j\to \infty} \ttrans{p_j}{d_j}(t) = 0$ for $t\in(0,1)$. For the sake of contradiction, we assume that $(p_j)_{j\in\N}$ does not converge to $\delta_0$. Then there exist $\epsilon>0$ and $a\in\R_{>0}$ such that
\[ 
 \limsup_{j\to\infty} \meas{p_j}(\R_{\ge a}) \ge \epsilon.
\]
By taking a subsequence, we may assume $\meas{p_j}(\R_{\ge a}) \ge \epsilon$ for all $j$.
Let us set another sequence of polynomials $q_j(x):= x^{d_j-k_j}(x-a)^{k_j}$ where $k_j = \lfloor \epsilon d_j \rfloor$.
Then it is clear that $\meas{q_j} \xrightarrow{w} (1 - \epsilon) \delta_0 + \epsilon \delta_{a} =: \nu$ and $\ttrans{q_j}{d_j}(t) \le \ttrans{p_j}{d_j}(t)$ for all $t \in (0,1)$ by Proposition \ref{prop.Phi.preserves.order}.
However, $\ttrans{q_j}{d_j}(t) \to T_\nu(t) \not \equiv 0$.
This is a contradiction.
Therefore we conclude that $(p_j)_{j\geq 1}$ converges to $\delta_0$. 
\end{proof}

Recall that given a measure $\mu\in\MM(\rr_{\geq 0})$ the map $\Phi$ from \eqref{eq:THMcharacterization} yields a measure $\Phi(\mu)\in\MM(\rr_{\geq 0})$ such that $T_{\mu}$ and $\cdf{\Phi(\mu)}$ are inverse functions 

We are now ready to give a proof of Theorem \ref{thm:main3}, namely that
\[\meas{p_d} \weak \mu \qquad\Leftrightarrow\qquad \meas{\Phi_d(p_d)}\weak \Phi(\mu).\]

\begin{proof}[Proof of Theorem \ref{thm:main3}]
Notice that $\meas{\Phi_d(p_d)}\weak \Phi(\mu)$ is equivalent to the convergence $F_{\meas{\Phi_d(p_d)}}(x)\to F_{\Phi(\mu)}$ as $d\to\infty$ for every $x\in\rr$ that is a continuous point of $F_{\Phi(\mu)}$. In turn, Lemma \ref{lem:oscuridad} and the definition of the $T$-transform assures us that the later is equivalent to the convergence $T^{(d)}_{p_d}(x)\to T_\mu(x)$ for every continuous point $x\in(0,1)$ of $T_\mu$. Since $T_\mu$ is continuous on $(0,1)$, the later is equivalent to $\meas{p_d}\weak \mu$ due to Theorem \ref{thm:main.T.transform}. 
\end{proof}

\section{Examples and applications}
\label{sec:examples.applications}

In this section, we present various limit theorems relating finite free probability to free probability. Thus, throughout the whole section, we will consider situations where the dimension $d$ or $d_j$ tends to infinity, and assume that the polynomials converge to a measure, as in Notation \ref{not:converging}.

\subsection{\texorpdfstring{$p_d\boxtimes_d q_d$}\ \  approximates \texorpdfstring{$\mu \boxtimes \nu$}\ }

As announced, we present the first application of our main theorems. That is, we present a new independent proof of part (2) in Proposition \ref{prop:finiteAsymptotics} which includes the general case.

\begin{proposition}
\label{prop:approx_boxtimesd}
Let $(p_d)_{d\in \N}$ and $(q_d)_{d\in \N}$ be sequences of polynomials such that $p_d,q_d \in \pols_d(\R_{\geq 0})$ and let $\mu, \nu \in \MM(\R_{\geq 0})$ such that $(p_d)_{d\in\N}$ and $(q_d)_{d\in \N}$ weakly converge to $\mu$ and $\nu$, respectively.
Then $(p_d\boxtimes_d q_d)_{d\in \N}$ weakly converges to $\mu\boxtimes \nu$.
\end{proposition}
\begin{proof}
By Equation \eqref{eq:rico} and Theorem \ref{thm:main.T.transform}, we obtain
\[
T^{(d)}_{p_d\boxtimes_d q_d}(t) =T_{p_d}^{(d)}(t)T_{q_d}^{(d)}(t) \rightarrow T_\mu(t) T_\nu(t) = T_{\mu\boxtimes \nu}(t)
\]
as $d\to\infty$ for every $t \in (0,1)$.
Hence, $\meas{p_d\boxtimes_d q_d}\weak \mu\boxtimes \nu$ by Theorem \ref{thm:main.T.transform} again.
\end{proof}

\subsection{A limit for the coefficients of a sequence of polynomials.}

Our main theorem provides a limit for the ratio of consecutive coefficients, in a converging sequence of polynomials. The convergence of ratios can be easily translated to understand other ratios, or the behaviour of the coefficients alone.
\begin{proposition}
Fix a measure $\mu \in \MM(\R_{\geq 0})$ and fix $0<t<u<1-\mu(\{0\})$. Let $p_j\in \PP_{d_j}(\rr)$ be a sequence of polynomials converging to $\mu$ and let $(k_j)_{j\in \N}$ and $(l_j)_{j\in \N}\subset\nn$ be diagonal sequences with ratio limit $t$ and $u$, respectively.
Then 
\[\lim_{j\to\infty}  \sqrt[d_j]{\frac{\coef{l_j}{d_j}(p_j)}{\coef{k_j}{d_j}(p_j)}} = \exp\left(- \int_t^u \log S_{\mu}(-x) dx \right).\] 
Additionally, unless $\mu = \delta_a$ for some $a>0$, 
\[\lim_{j\to\infty}  \sqrt[d_j]{\frac{\coef{l_j}{d_j}(p_j)}{\coef{k_j}{d_j}(p_j)}} = \exp\left(-\int_t^u \log S_{\mu}(-x) dx \right)=\exp\left(\int_{T_{\mu}(1-u)}^{T_{\mu}(1-t)}\log x \,  \Phi(\mu) (dx) \right).\] 
\end{proposition}

\begin{proof}
By Theorem \ref{thm:main.T.transform}, one has the convergence of $\ttrans{p_j}{d_j}$ to $T_\mu$.
Note that it is locally uniform by Polya's theorem since they are monotone functions.
The same is true for the convergence of $\log \ttrans{p_j}{d_j}$ to $\log T_\mu$.
Hence, 
\begin{equation} \label{eq:casa}
    \lim_{j\to\infty} \int_t^u \log \ttrans{p_j}{d_j}(x) dx = \int_t^u \log T_\mu(x) dx.
\end{equation}
Besides, if $\mu$ is not a Dirac measure then $T_\mu(x)$ is a strictly monotone function.
Thus, by the change of variables $y = T_\mu(x)$ that is equivalent to $F_{\Phi(\mu)}(y) = x$, one has
\[
\int_t^u \log T_\mu(x) dx = \int_{T_{\mu}(t)}^{T_{\mu}(u)} \log y\; \Phi(\mu)(dy).
\]

For diagonal sequences $(k_j)_{j\in \N}, (l_j)_{j\in \N}$ with ratio limits $\lim_{j\to\infty}\frac{k_j}{d_j}=t$ and $\lim_{j\to\infty}\frac{l_j}{d_j}=u$, the left-hand limit of \eqref{eq:casa} coincides with the limit of 
\[
   \begin{split}
    \int_{\frac{k_j}{d_j}}^{\frac{l_j}{d_j}} \log \ttrans{p_j}{d_j}(x) dx &= \frac{1}{d_j}\sum_{i=k_j}^{l_j-1} \log \ttrans{p_j}{d_j}\left(\frac{i}{d_j}\right)\\
    &= \frac{1}{d_j}\sum_{i=k_j}^{l_j-1} \log\left( \frac{\coef{d_j-i}{d_j}(p_j)}{\coef{d_j-i-1}{d_j}(p_j)}\right) \\
    &= \frac{1}{d_j} \log\left( \frac{\coef{d_j-l_j-1}{d_j}(p_j)}{\coef{d_j-k_j-1}{d_j}(p_j)}\right) 
   \end{split}
\]
by the approximation of Riemann sum.
Taking the exponential, we have 
\[
\lim_{j\to\infty}  \sqrt[d_j]{\frac{\coef{d_j-l_j-1}{d_j}(p_j)}{\coef{d_j-k_j-1}{d_j}(p_j)}} = \exp\left(\int_t^u \log T_\mu(x) dx \right).
\] 
Finally, by the simple parameter change, we obtain the desired result. 
\end{proof}

Notice that if we may take $t = 0$ in the previous result then we have the following:

\begin{corollary} \label{cor:lim_coef}
Fix a measure $\mu \in \MM(\R_{\geq 0})$ and a sequence of polynomials $p_j\in \PP_{d_j}(\rr)$ converging to $\mu$.
Assume $S_{\mu}(0) = 1/T_\mu(1) = 1/m_1(\mu) \in (0,\infty)$ and
\[
  \strans{p_j}{d_j}\left(-\frac{1}{d_j}\right) = \frac{1}{\ttrans{p_j}{d_j}\left(\frac{d_j-1}{d_j}\right)} = \frac{1}{\coef{1}{d_j}(p_j)} \to \frac{1}{m_1(\mu)}
\]
as $j\to \infty$, which means the first moment convergence of $p_j$ to $\mu$.
Then for every $t\in(0, 1-\mu(\{0\}))$ and diagonal sequence $(k_j)_{j\in \N}\subset\nn$ with ratio limit $t$, we have
\[\lim_{j\to\infty}  \left(\coef{k_j}{d_j}(p_j)\right)^{\frac{1}{d_j}} = \exp\left(-\int_0^t \log S_{\mu}(-x) dx \right).\]  
Additionally, unless $\mu = \delta_a$ for some $a>0$, 
\[\lim_{j\to\infty}  \left(\coef{k_j}{d_j}(p_j)\right)^{\frac{1}{d_j}} = \exp\left(-\int_0^t \log S_{\mu}(-x) dx \right) = \exp\left(\int_{T_{\mu}(1-t)}^{T_{\mu}(1)}\log x \,  \Phi(\mu) (dx) \right),\]
where we may replace $T_\mu(1) = m_1(\mu)$ by $\infty$ because the support of $\Phi(\mu)$ is included in $[T_{\mu}(0), T_{\mu}(1)]$, see Equation \eqref{eq:cierra}.
\end{corollary}

\begin{remark} In \cite{HM13} it is shown that the following integrals are all equal:
\[
\int_0^1 \log S_{\mu}(-x) dx =\int_0^\infty \log x \, \mu (dx) =
\int_0^\infty \log x \Phi(\mu) (dx),\]
whenever one of them is finite. Now, recall from \cite{fuglede1952determinant} that the Fuglede-Kadison determinant of a positive operator $T$ in a tracial $W^*$-algebra, with distribution $\mu_T$ is given by

$$\Delta(T)=\exp\left(\int_0^\infty \log x\, \mu_{T} (dx)\right).$$

In the case of a positive matrix $A_d$ of dimension $d$ with eigenvalues $\{\lambda_j\}_{j=1}^{d}$, the Fuglede-Kadison determinant can be written as 
$$\Delta(A_d)=(\det A_d)^{\frac1d}= \prod^d_i{\lambda_i}^{\frac1d}= \left(\widetilde{\mathsf{e}}_d^{(d)}\right)^{\frac{1}{d}}.$$
Hence, the statements above can be seen as a generalization of the convergence of the Fuglede-Kadison determinant for finite dimensional operators that converge weakly to an operator $T$.

\end{remark}

\subsection{Hypergeometric polynomials}
\label{ssec:hypergeometric}

The hypergeometric polynomials are a particular family of generalized hypergeometric functions. They were studied in connection with finite free probability in \cite{mfmp2024hypergeometric} and several families of parameters where the polynomials have all positive roots were determined. This large family of polynomials contains as particular cases some important families of orthogonal polynomials, such as Laguerre, Bessel and Jacobi polynomials. In this section, we compute their finite $S$-transform and use it to directly obtain the $S$-transform of their limiting root distribution.

\begin{definition}[Hypergeometric polynomials]
For $a_1, \dots, a_i\in \R\setminus \left\{\tfrac{1}{d},\tfrac{2}{d}, \dots, \tfrac{d-1}{d} \right\}$ and $b_1, \dots, b_j\in \R$, we denote\footnote{Our notation slightly differs from that in \cite{mfmp2024hypergeometric}, by a dilation of $d^{i-j}$ on the roots. This simplifies the study of the asymptotic behaviour of the roots, and does not change the fact that roots are all real (or all positive).} by $\HGP{d}{b_1, \dots, b_j}{a_1, \dots, a_i}$ the unique monic polynomial of degree $d$ with coefficients given by
\[
\coef{k}{d}\left(\HGP{d}{b_1, \dots, b_j}{a_1, \dots, a_i}\right):= d^{k(i-j)} \frac{\prod_{s=1}^j\falling{b_s d}{k}}{\prod_{r=1}^i\falling{a_r d}{k}} \qquad \text{ for } k=1,\dots, d.
\]
\end{definition}
Notice that the reason why we do not allow a parameter below to be of the form $\frac{k}{d}$ is to avoid indeterminacy (a division by 0). 

We also allow the cases where there is no parameter below or no parameter above ($i=0$ or $j=0$), in these cases the coefficients are
\[
\coef{k}{d}\left(\HGP{d}{b_1,\dots,b_j}{-}\right):= d^{-kj}\prod_{s=1}^j\falling{b_s d}{k} \quad \text{and} \quad \coef{k}{d}\left(\HGP{d}{-}{a_1,\dots, a_i}\right):= d^{ki}\prod_{r=1}^i \frac{1}{\falling{a_r d}{k}}.
\]

Since the ratio of two consecutive coefficients of a hypergeometric polynomial is easily expressed in terms of the parameters, a direct computation yields their finite $S$-transform. 

\begin{lemma}
\label{lem:Strans.hypergeometric}
For parameters $a_1, \dots, a_i\in \R\setminus\left\{\tfrac{1}{d},\tfrac{2}{d}, \dots, \tfrac{d-1}{d}\right\}$ and $b_1, \dots, b_j\in \R$, the finite $S$-transform of the polynomial $p:=\HGP{d}{b_1,\dots, b_j}{a_1, \dots, a_i}$ is 
\[
\strans{p}{d}\left(-\frac{k}{d}\right)=  \frac{\prod_{r=1}^i(a_r -\frac{k-1}{d})}{\prod_{s=1}^j(b_s-\frac{k-1}{d})}.
\]
Equivalently, the roots of the polynomial $\Phi_d(p)$ are given by:
\[\frac{\prod_{s=1}^j(b_s-\frac{k-1}{d})}{\prod_{r=1}^i(a_r -\frac{k-1}{d})}\qquad\text{for } k=1,\dots, d.\]
\end{lemma}

\begin{proof}
Directly from the definition we compute
\begin{align*}
\strans{p}{d}\left(-\frac{k}{d}\right) 
&=\prod_{r=1}^i\prod_{s=1}^j\frac{d^{(k-1)(i-j)}\falling{b_s d}{k-1}}{\falling{a_r d}{k-1}}\frac{\falling{a_r d}{k}}{d^{k(i-j)}\falling{b_s d}{k}}\\
&=\prod_{r=1}^i\prod_{s=1}^j d^{j-i}\frac{a_r d-k+1}{b_s d-k +1}\\
&=\prod_{r=1}^i\prod_{s=1}^j \frac{a_r -\frac{k-1}{d}}{b_s-\frac{k-1}{d}}
\end{align*}
as desired. 
\end{proof}

As a direct corollary we can determine the $S$-transform of the limiting measure whenever the hypergeometric polynomials have all positive real roots. Notice that a more general result for hypergeometric polynomials that does not necessarily have real roots was recently proved in \cite[Theorem 3.7]{martinez2024zeros}; there, the approach relies on the three-term recurrence relation that is specific to this family of polynomials. We would like to highlight that with our approach, the computation of the limiting measure using our main theorem is straightforward.

\begin{corollary}
\label{cor:hypergeometric}
For every $d\geq 1$, consider parameters $a_{1}^{(d)}, \dots, a^{(d)}_{i}\in \R\setminus\left\{\tfrac{1}{d},\tfrac{2}{d}, \dots, \tfrac{d-1}{d}\right\}$ and $b_{1}^{(d)}, \dots, b_{j}^{(d)}\in \R$ such that the following limits exist
\[\lim_{d\to\infty}a_{1}^{(d)}= a_1,\quad  \dots, \quad \lim_{d\to\infty}a_{i}^{(d)}= a_i,\quad \lim_{d\to\infty}b_{1}^{(d)}= b_1, \quad \dots,\quad  \lim_{d\to\infty}b_{j}^{(d)}= b_j.\] 
Assume further $p_d:=\HGP{d}{b_{1}^{(d)},\ \dots,\ b_{j}^{(d)}}{a_{1}^{(d)},\ \dots,\ a^{(d)}_{i}} \in \pols_d(\rr_{>0})$ for every $d$. Then
\[
\lim_{d\to\infty}\meas{p_d} =\mu,
\]
where $\mu$ is the measure supported in the positive real line, with $S$-transform given by 
\[S_\mu(t)= \frac{ \prod_{r=1}^i(a_r+t)}{\prod_{s=1}^j(b_s+t)} \qquad \text{for } t\in (-1,0).
\]
\end{corollary}

\begin{remark}
\label{rem:hypergeometric.negative}
Notice that in Corollary \ref{cor:hypergeometric} one can also consider polynomials with all negative roots, or equivalently, one can let the sequence of polynomials be 
\[p_d:=\dil{-1}\HGP{d}{b_{1}^{(d)},\ \dots,\ b_{j}^{(d)}}{a_{1}^{(d)},\ \dots,\ a^{(d)}_{i}} \in \pols_d(\rr_{<0}).\]
In this case, the odd coefficients of the polynomials change sign, and this means that the finite $S$-transforms changes sign, ultimately we obtain that the $S$-transform of the limiting measure is 
\[S_\mu(t)=- \frac{ \prod_{r=1}^i(a_r+t)}{\prod_{s=1}^j(b_s+t)} \qquad \text{for } t\in (-1,0).
\]
\end{remark}

\begin{example}[Identity for $\boxtimes_d$]
When there are no parameters above nor below, we obtain the polynomial  
\[p_d=\HGP{d}{-}{-}(x)=(x-1)^d,\]
which is the identity for the multiplicative convolution. Its finite $S$-transform is given by
\[\strans{p_d}{d}\left(-\frac{k}{d}\right)= 1\qquad \text{for } k=1,\dots ,d.\]
Thus the limiting distribution is the $\delta_1$ measure.
\end{example}

\begin{example}[Laguerre polynomials]
\label{exm:Laguerre}
In the case where we just consider one parameter $b$ above and no parameter below, the polynomial $p=\HGP{d}{b}{-}$ is the well-known Laguerre polynomial and has all positive roots whenever $b>1-\frac{1}{d}$. By Lemma \ref{lem:Strans.hypergeometric} the finite $S$-transform is
\[\strans{p}{d}\left(-\frac{k}{d}\right)= \frac{1}{b-\frac{k-1}{d}}\qquad \text{for } k=1,\dots ,d.\]
With respect to the asymptotic behaviour, notice that using Corollary \ref{cor:hypergeometric} we can retrieve the known result that the limiting zero counting measure of a sequence of Laguerre polynomials is the Marchenko-Pastur distribution. Indeed, for a sequence $(b_d)_{d\geq 1}\subset[1,\infty)$ with $\lim_{d\to\infty}b_d=b\geq 1$, then in the limit we obtain the Marchenko-Pastur distribution of parameter $b$:
\[
\lim_{d\to\infty}\meas{\HGP{d}{b_d}{-}} ={\bf MP}_b,
\]
which is determined by the $S$-transform
\[S_{{\bf MP}_b}(t)= \frac{ 1}{b+t} \qquad \text{for } t\in (-1,0).
\]
On the other hand, the spectral measure of $\Phi_d(p)$ is the uniform distribution on $\{b-1,b-1+\frac{1}{d},\dots, b-\frac{1}{d}\}$ and, in the limit, we retrieve the fact that
\[\Phi({\bf MP}_b)={\bf U}(b-1,b).\]
\end{example}

As an application of the previous example and Proposition \ref{prop:approx_boxtimesd}, we know that the multiplicative convolution of a polynomial with a Laguerre polynomial provides us with a finite approximation of compound free Poisson distribution.

\begin{corollary}
\label{cor:free_compound}
Let $(p_j)_{j\in \N}$ be a sequence of polynomials with $p_j\in \pols_{d_j}(\R_{\geq 0
})$ such that $\meas{p_j}\weak \nu\in \MM(\R_{\geq 0
})$. Then $\left(\HGP{d_j}{b}{-} \boxtimes_{d_j} p_j\right)_{j\in \N}$ weakly converges to the compound free Poisson distribution ${\rm \bf MP_b}\boxtimes \nu$ (with rate $b$ and jump distribution $\nu$).
\end{corollary}

We now turn our attention to the reversed Laguerre polynomials.

\begin{example}[Bessel polynomials]
\label{exm:Bessel}

In the case where we just consider one parameter $a$ below and no parameter above, the polynomial $p(x)=\HGP{d}{-}{a}(-x)$ goes by the name of Bessel polynomial and has all positive roots whenever $a<0$. By Remark \ref{rem:hypergeometric.negative} the finite $S$-transform is
\[\strans{p}{d}\left(-\frac{k}{d}\right)= -a+\frac{k-1}{d} \qquad \text{for } k=1,\dots ,d.\]
If we consider a sequence $(a_d)_{d\geq 1}\subset (-\infty,0)$ with $\lim_{d\to\infty}a_d=a<0$, then the limiting $S$-transform is equal to 
\[S_{{\bf RMP}_{1-a}}(t)= -t-a \qquad \text{for } t\in (-1,0)
\]
that corresponds to the measure ${\bf RMP}_{1-a}=({\bf MP}_{1-a})^\reversed$ that is the reciprocal distribution of a Marchenko-Pastur distribution of parameter $1-a$, see \cite[Proposition 3.1]{Yoshida}.
\end{example}

\begin{example}[Jacobi polynomials]
\label{exm:Jacobi}
In the case where we consider one parameter $b$ above and one parameter $a$ below, we obtain the Jacobi polynomials $p=\HGP{d}{b}{a}$. Recall that we already encountered these polynomials in Section \ref{sec:bounds.Jacobi} and proved some bounds for some specific parameters. We now complete the picture. There are several regions of parameters where the polynomial has all positive roots.
Below we highlight the three main regions of parameters where we get a polynomial with positive roots.
The reader is referred to \cite[Section 5.2 and Table 2]{mfmp2024hypergeometric} for the complete description of all the regions.
Same as for the previous examples, we can consider sequences $(a_d)_{d\geq 1}$, $(b_d)_{d\geq 1}$ in those region of parameters and with limits $\lim_{d\to\infty}a_d=a$, $\lim_{d\to\infty}b_d=b$. By Corollary \ref{cor:hypergeometric} we can compute the $S$-transform of the limiting measure that we denote by $\mu_{a,b}$.
\begin{itemize}
\item For $b>1$ and $a>b+1$ then $\HGP{d}{b}{a}\in \pols_d([0,1])$. The $S$-transform of the limiting measure $\mu_{a,b}$ is given by
\[S_{\mu_{a,b}}(t)= \frac{a+t}{b+t}=1+\frac{a-b}{b+t} \qquad \text{for } t\in (-1,0).\]

\item For $b>1$ and $a<0$ then $\HGP{d}{b}{a}(-x)\in\polplus$. The $S$-transform of the limiting measure $\mu_{a,b}$ is given by
\[S_{\mu_{a,b}}(t)= \frac{-a-t}{b+t} \qquad \text{for } t\in (-1,0).\]

Notice that in this case, $\mu_{a,b}={\bf MP}_{b}\boxtimes ({\bf MP}_{1-a})^\reversed$ is the free beta prime distribution $f\beta'(b,1-a)$ studied by Yoshida \cite[Eq (2)]{Yoshida}. In other words, a simple change of variable yields that for $a,b>1$ the limiting measure of the polynomials $\HGP{d}{a_d}{1-b_d}(-x)$ tends to the measure $f\beta'(a,b)$.

\item For $a<0$ and $b<a-1$, then $\HGP{d}{b}{a}\in \pols_d(\R_{>0})$. The $S$-transform of the limiting measure $\mu_{a,b}$ is given by
\[S_{\mu_{a,b}}(t)= \frac{a+t}{b+t} \qquad \text{for } t\in (-1,0).\]
\end{itemize}
\end{example}

To finish this section we want to highlight a case where our main result applies for hypergeometric polynomials with several parameters. 

\begin{proposition}
\label{prop:hypergeometric}
For every $d\geq 1$, let $a_{1}^{(d)}, \dots, a^{(d)}_{i}<0$ and $b_{1}^{(d)}, \dots, b_{j}^{(d)}>1$ such that the following limits exist
\[\lim_{d\to\infty}a_{1}^{(d)}= a_1,\quad  \dots, \quad \lim_{d\to\infty}a_{i}^{(d)}= a_i,\quad \lim_{d\to\infty}b_{1}^{(d)}= b_1, \quad \dots,\quad  \lim_{d\to\infty}b_{j}^{(d)}= b_j.\]  
Let $p_d(x):=\HGP{d}{b_{1}^{(d)},\ \dots,\ b_{j}^{(d)}}{a_{1}^{(d)},\ \dots,\ a^{(d)}_{i}} \left((-1)^ix\right)$ for every $d$.
Then
\[
\lim_{d\to\infty}\meas{p_d} =\mu,
\]
where $\mu$ is the measure supported in the positive real line with $S$-transform given by 
\[S_\mu(t)= \frac{ \prod_{r=1}^i(-a_r-t)}{\prod_{s=1}^j(b_s+t)} \qquad \text{for } t\in (-1,0).
\]
\end{proposition}

\begin{proof}
From \cite[Theorem 4.6]{martinez2024zeros} we know that all polynomials $p_d$ have positive real roots, so we can apply Corollary \ref{cor:hypergeometric} when $i$ is even (or Remark \ref{rem:hypergeometric.negative} when $i$ is odd) and the conclusion follows directly.
\end{proof}

\begin{remark}
Notice that in the previous result, we can identify $\mu$ as the free multiplicative convolution of Marchenko-Pastur and reversed Marchenko-Pastur distributions. Indeed, it follows from the multiplicativity of the $S$-transform and Examples \ref{exm:Laguerre} and \ref{exm:Bessel} that
\[\mu= {\bf MP}_{b_1}\boxtimes \cdots \boxtimes {\bf MP}_{b_j}\boxtimes ({\bf MP}_{1-a_1})^\reversed\boxtimes \cdots \boxtimes ({\bf MP}_{1-a_i})^\reversed.\]
\end{remark}

\subsection{Finite analogue of some free stable laws}

The purpose of this section is to give some finite analogue of free stable laws. Free stable laws are defined as distributions $\mu$ satisfying that for every $a,b>0$ there exist $c>0$ and $t\in \R$ such that
\[\dil{a}\mu \boxplus \dil{b}\mu = \dil{c}\mu \boxplus \delta_t.\]
Any free stable law can be uniquely (up to a scaling parameter) characterized by a pair $(\alpha,\rho)$, which is in the set of admissible parameters:
\[\mathcal{A}:=\{(\alpha,\rho): 0<\alpha\le 1,\ 0\le \rho \le 1 \} \cup \{(\alpha,\rho):1< \alpha \le 2,\ 1-\alpha^{-1} \le \rho \le \alpha^{-1}\}.\]
More precisely, the Voiculescu transform $\varphi_\mu(z):={G_\mu^{\langle -1\rangle}(z^{-1})-z}$ of the free stable law $\mu$ is given by
\[\varphi_\mu(z)=-e^{i\pi \alpha \rho}z^{-\alpha+1}, \qquad (\alpha,\rho) \in \mathcal{A}, \quad z\in \C^+,\]
see \cite{berpata1999, arizmendi2016classical, hasebe2020some} for details.
We then denote by ${\bf fs}_{\alpha,\rho}$ the free stable law with an admissible parameter $(\alpha,\rho) \in \mathcal{A}$. In particular, the following cases are well-known.
\begin{itemize}
\item (Wigner's semicircle law) ${\bf fs}_{2, \frac{1}{2}}(dx) = \mu_{\mathrm{sc}}(dx):=\frac{1}{2\pi} \sqrt{4-x^2}\mathbf{1}_{[-2,2]}(x)dx.$
\item (Cauchy distribution) ${\bf fs}_{1, \frac{1}{2}}(dx) = \frac{1}{\pi (1+x^2)} \mathbf{1}_{\R}(x)dx$.
\item (Positive free $\frac{1}{2}$-stable law) ${\bf fs}_{\frac{1}{2},1}(dx)= \frac{\sqrt{4x-1}}{2\pi x^2} \mathbf{1}_{[\frac{1}{4},\infty)}(x)dx$.
\end{itemize}

In the following, we construct some finite analogues of free stable laws.

\begin{example}[Hermite polynomials] \label{Example:Hermite}
It is well known that the Hermite polynomials (with an appropiate normalization) converge in distribution to the semicircle law, see for instance \cite{Mar21}. Specifically, if let
    \[H_{2d}(x):=\sum_{k=0}^{d} \binom{2d}{2k} (-1)^k \frac{(2k)!}{k!(4d)^k} x^{2d-2k} \in \pols_{2d}^S(\rr)\]
denote the Hermite polynomial of degree $2d$, then $\meas{H_{2d}} \weak \mu_{\mathrm{sc}}$ as $d\to \infty$.
Thus, we can interpret $H_{2d}$ as the finite analogue of the symmetric free $2$-stable law ${\bf fs}_{2,\frac{1}{2}}$.
The finite symmetric $S$-transform of $H_{2d}$ can be easily computed:
\[\stranstilde{H_{2d}}{2d}\left(-\frac{k}{d}\right)= \frac{1}{\sqrt{-\tfrac{k}{d} +\tfrac{1}{2d}}}.\]
\end{example}

\begin{example}[Positive finite $\frac{1}{2}$-stable]
\label{exm:positive.half.stable}
 From \cite{berpata1999} we know that
${\bf MP}_1^\reversed={\bf fs}_{\frac{1}{2},1}$ is the positive free $\frac{1}{2}$-stable law. We also have that the  compound free Poisson distribution ${\bf MP}_1\boxtimes {\bf fs}_{\frac{1}{2},1}$ coincides with the positive boolean $\frac{1}{2}$-stable law, see \cite{SW1997}. Now, we provide the finite counterparts as follows.

Recall from Example \ref{exm:Laguerre} that the Laguerre polynomial $\HGP{d}{1}{-}$ is the finite free analogue of the Marchenko-Pastur distribution ${\bf MP}_1$. From \cite[Eq. (81)]{mfmp2024hypergeometric} we know that its reversed polynomial is 
\begin{equation}
\label{eq:finitefreestable_1/2}
    f_d(x):=\HGP{d}{-}{-\frac{1}{d}}(-x),
\end{equation}
and letting $d\to\infty$ the empirical root distribution of these polynomials tends to the positive free $\frac{1}{2}$-stable law ${\bf MP}_1^\reversed={\bf fs}_{\frac{1}{2},1}$. Clearly, we have
\[\strans{f_d}{d}\left(-\frac{k}{d}\right)=\frac{k}{d} \qquad \text{for } k=1,\dots, d.\]

By \cite[Eq (82)]{mfmp2024hypergeometric} and Corollary \ref{cor:free_compound}, this means that the Jacobi polynomials
\[\HGP{d}{1}{-\frac{1}{d}}(-x)=\left(\HGP{d}{1}{-}(x)\right)\boxtimes_d \left(\HGP{d}{-}{-\frac{1}{d}}(-x)\right)\]
are the finite analogue of the positive boolean $\frac{1}{2}$-stable law. 
\end{example}

\begin{example}[Symmetric finite $\frac{2}{3}$-stable]
According to \cite[Theorem 12]{arizmendi2009}, there is an interesting relation between positive free stable laws and symmetric free stable laws via the free multiplicative convolution. That is, 
\begin{equation}
\label{eq:free_stable_relation}
{\bf fs}_{\alpha,\frac{1}{2}} = \mu_{\mathrm{sc}} \boxtimes {\bf fs}_{\frac{2\alpha}{\alpha+2}, 1}
\end{equation}
for any $\alpha \in (0,2)$. Here we give the finite analogue of the symmetric $\frac{2}{3}$-stable law.

We define 
\[g_{2d} := H_{2d} \boxtimes_{2d} f_{2d}.\]
By Proposition \ref{prop:stransform_multiplicative_symmetricpols} and Lemma \ref{lem:Strans.hypergeometric}, we have
\begin{equation}
\begin{split}
    \left(\stranstilde{g_{2d}}{2d}\left(-\frac{k}{d}\right)\right)^2&= \left(\stranstilde{H_{2d}}{2d}\left(-\frac{k}{d}\right) \right)^2 \strans{f_{2d}}{2d}\left(-\frac{2k}{2d}\right) \strans{f_{2d}}{2d}\left(-\frac{2k-1}{2d}\right)\\
    &= -\frac{2d}{2k-1} \cdot \frac{2k}{2d} \cdot \frac{2k-1}{2d}\\
    &=-\frac{k}{d},
    \end{split}
\end{equation}
and therefore
\[\stranstilde{g_{2d}}{2d}\left(-\frac{k}{d}\right) = \sqrt{-\frac{k}{d}}.\]
By Proposition \ref{prop:Strans.symmetric.convergence}, if we let $d\to \infty$ with $\frac{k}{d}\to t \in (0,1)$, then
\[\stranstilde{g_{2d}}{2d}\left(-\frac{k}{d}\right) \to \sqrt{-t} = \widetilde{S}_{{\bf fs}_{\frac{2}{3},\frac{1}{2}}} (-t),\]
where the last equality follows from \cite{arizmendi2016classical}. Clearly, we notice this example is the finite analogue of \eqref{eq:free_stable_relation} in the case when $\alpha=\frac{2}{3}$.
\end{example}

\subsection{Finite free multiplicative Poisson's law of small numbers}

In \cite[Lemma 7.2]{BerVoi1992} it was shown that for $\lambda\geq 0$ and $\beta\in\rr\setminus[0,1]$ there exists a measure $\Pi_{\lambda,\beta}\in\MM(\R_{\geq 0})$ with $S$-transform given by 
\[
S_{\Pi_{\lambda,\beta}}(t)=\exp\left( \frac{\lambda}{t+\beta}\right).
\]
This measure can be understood as a free multiplicative Poisson's law. The purpose of this section is to give a finite counterpart.

In this case, we can think of it as a limit of convolution powers of polynomials of the form 
\[(x-\tfrac{\beta-1}{\beta})(x-1)^{d-1}=\HGP{d}{\beta-\frac{1}{d}}{\beta},\]
where the equality follows from a direct computation, see also \cite[Eq (60)]{mfmp2024hypergeometric}.

\begin{proposition}
\label{prop:Poisson_law_limit_theorem_specialcase}
Let $\lambda\geq 0$ and $\beta\in\rr\setminus[0,1]$, and for each $d$ 
consider the polynomial 
\[p_d(x):= (x-\tfrac{\beta-1}{\beta})(x-1)^{d-1}.\]
Then 
\[
\meas{p_d^{\boxtimes_d n}} \weak \Pi_{\lambda,\beta} \qquad \text{as }d\to \infty \quad  \text{ with } \frac{n}{d}\to \lambda.
\]
\end{proposition}

\begin{proof}
Consider $\alpha \in (0,1)$.
Recall that the coefficients of $p_d$ are
\[\coef{k}{d}(p_d)=\frac{\falling{d\beta-1}{k}}{\falling{d\beta}{k}}=\frac{d\beta-k}{d\beta},\]
so its finite $S$-transform is given by 
\[
\strans{p_d}{d}\left(-\frac{k}{d}\right)= \frac{d\beta-k+1}{d\beta-k}= 1+\frac{1}{d\beta-k} \qquad \text{for }k=1,\dots, d,
\]
and the finite $S$-transform of $q_d:=p_d^{\boxtimes_{d} n}$ is given by
\[
\strans{q_d}{d}\left(-\frac{k}{d}\right)= \left(1+\frac{1}{d\beta-k}\right)^{n}=\left(1+\frac{1}{d}\frac{1}{\beta-\frac{k}{d}}\right)^{d\frac{n}{d}} \qquad \text{for }k=1,\dots, d.
\]
Then, if we let $d\to\infty$ with $\frac{k}{d}\to t$ and $\frac{n}{d}\to \lambda$ then we obtain that
\[
\lim_{d\to\infty} \strans{q_d}{d}\left(-\frac{k}{d}\right) =  \exp\left(\frac{\lambda}{\beta-t} \right) = S_{\Pi_{\lambda,\beta}}(-t).
\]
The conclusion follows from Theorem \ref{thm:main}.
\end{proof}

\subsection{Finite max-convolution powers}

In 2006, Ben Arous and Voiculescu \cite{BV2006} introduced the free analogue of max-convolution. Given two measures $\nu_1, \nu_2 \in \MM(\rr)$, their \emph{free max-convolution}, denoted by $\nu_1\boxlor \nu_2$, is the measures with cumulative distribution function given by
\[
F_{\nu_1\boxlor \nu_2}(x) :=\max\{F_{\nu_1}(x)+F_{\nu_2}(x)-1,0\} \qquad\text{for all } x\in \mathbb{R}.
\]
Similarly, given $\nu \in \MM(\rr)$ and $t\geq 1$, one can define the \emph{free max-convolution of $\nu$ to the power $t$} as the unique measure $\nu^{\boxlor t}$ with cumulative distribution function given by
\begin{align}\label{eq:freemaxpower}
F_{\nu^{\boxlor t}}(x):=\max\{tF_\nu(x)-(t-1),0\}, \qquad t\ge 1.
\end{align}
This notion was introduced by Ueda, who used  Tucci-Haagerup-M\"{o}ller's limit to relate it to free additive convolution powers in \cite[Theorem 1.1]{ueda2021max}: given $t\geq 1$ and $\mu\in \MM(\rr_{\geq 0})$, one has that
\begin{equation}\label{eq:Ueda}
\Phi(\dil{1/t}(\mu^{\boxplus t}) )=\Phi(\mu)^{\boxlor t}.
\end{equation}

The purpose of this section is to prove a finite free analogue of this relation and to show that it approximates \eqref{eq:Ueda} as the degree $d$ tends to $\infty$.  

\begin{definition}
\label{def:finite.free.max}
Given $p\in\pols_d(\rr)$ and $1\leq k\leq d$, we define the {\it finite max-convolution power} $\frac{d}{k}$ of $p$ as the polynomial $p^{\boxlor \frac{d}{k}}\in \pols_k(\rr)$ with roots given by
\[\lambda_j\left(p^{\boxlor \frac{d}{k}}\right)=\lambda_j(p) \qquad \text{for }j=1,\dots, k.\]
\end{definition}

It is straightforward from the definition of free-max convolution in \eqref{eq:freemaxpower} that
\[\meas{p^{\boxlor \frac{d}{k}}}= \left(\meas{p}\right)^{\boxlor \frac{d}{k}}.\]

Then, we have the following.

\begin{proposition}Let $p\in\pols_d(\rr_{\geq 0})$. Then 
\begin{equation}
\label{eq:Phi_max_formula}
\Phi_k(\diff{k}{d}p) = \left(\Phi_d(p)\right)^{\boxlor \frac{d}{k}} \qquad \text{for }k=1,\dots,d.    
\end{equation}
\end{proposition}

\begin{proof}
The proof uses two main ingredients: Equation \eqref{eq:phi.roots} that expresses the roots of $\Phi_d(p)$ in terms of the coefficients of $p$; and Lemma \ref{lem:coefficients_derivatives} that relates the coefficients of $\diff{k}{d}p$ and $p$.

Let us denote by $r$ the multiplicity of the root 0 in $p$. Then, the multiplicity of $\diff{k}{d} p$ in 0 is $\max\{r-k,0\}$. Using Lemma \ref{lem:coefficients_derivatives} and Equation \eqref{eq:phi.roots} twice, we compute 
\[
\lambda_j\left(\Phi_k(\diff{k}{d}p)\right) 
=
\frac{\coef{j}{k}(\diff{k}{d}p)}{\coef{j-1}{k}(\diff{k}{d}p)}
=
\frac{\coef{j}{d}(p)}{\coef{j-1}{d}(p)}=\lambda_j\left( \Phi_d(p) \right) \qquad \text{for } 1\leq j\leq \min\{d-r,k\},
\]
and in the case where $k\geq d-r$ we can also check that
\begin{align*}
\lambda_j\left(\Phi_k(\diff{k}{d}p)\right) 
= 0 =\lambda_j\left( \Phi_d(p) \right) \qquad \text{for } d-r+1\leq j\leq k,
\end{align*}
as desired.
\end{proof}

From Equation \eqref{eq:Phi_max_formula} and Theorem \ref{thm:main3} it readily follows that in the limit we obtain an approximation of Equation \eqref{eq:Ueda}.

\begin{corollary}
\label{cor:max.convolution}
Fix $t\geq 1$ and a measure $\mu \in \MM(\rr_{\geq 0})$. Let $(d_j)_{j\in \N} \subset \N$ and $(k_j)_{j\in \N}\subset \N$ a diagonal sequence with ratio limit $1/t$. If $p_j\in \pols_{d_j}(\rr_{\geq 0})$ is a sequence of polynomials converging to $\mu$, then
\[
\meas{\Phi_{k_j}(\diff{k_j}{d_j}p_{d_j})}\weak \Phi(\mu)^{\boxlor t} \qquad\text{ as }j\to\infty.
\]
\end{corollary}

We now look at an example related to free stable laws and their finite counterparts.

\begin{example}
\label{exm:stable}
From \cite[Example 4.2]{ueda2021max} we know that $\Phi\left({\bf fs}_{\frac{1}{2},1}\right)^{\boxlor t}$ coincides with the free Fr\'{e}chet distribution (or Pareto distribution) with index $t\geq 1$.

In the finite free framework, recall from Example \ref{exm:positive.half.stable} that these polynomials correspond to $\Phi_k\left( \diff{k}{d} f_d \right)$, where $f_d$ was defined in \eqref{eq:finitefreestable_1/2} as the reversed polynomial of the standard Laguerre polynomial. Therefore, if we let $d\to \infty$ and $\frac{d}{k}\to t \ge 1$, Corollary \ref{cor:max.convolution} yields
\begin{equation}
\label{eq:HGP_Dil_Diff}
\meas{\Phi_k\left( \diff{k}{d} f_d \right) }\weak \Phi\left({\bf fs}_{\frac{1}{2},1} \right)^{\boxlor t}.
\end{equation}
\end{example}

To finish this section, we study the finite analogue of the map $\Theta$ on positive measures that was defined in \cite{hasebe2021homomorphisms}. If we denote 
\[\Theta(\mu):= \Phi(\mu \boxtimes {\bf MP}_1^\reversed) =\Phi\left(\mu\boxtimes {\bf fs}_{\frac{1}{2},1} \right) \qquad \text{for } \mu\in \MM(\rr_{\geq 0}), \]
then it holds that 
\begin{equation}
\label{eq:HUformula}
\Theta(\mu^{\boxplus t})= \Theta(\mu)^{\boxlor t} \qquad \text{for all }t\ge 1. 
\end{equation}

We can define the finite free analogue of the map $\Theta$ as 
\[\Theta_d (p):=\Phi_d\left(p \boxtimes_d  f_d\right) \qquad \text{for } p\in \pols_d(\rr_{\geq 0})\]

To obtain the finite analogue of \eqref{eq:HUformula}, we first compute the derivatives of $f_d$.
\begin{lemma}
\label{lem:f_k_formula}
For $1\le k \le d$, we get
\begin{equation*}
    \dil{\frac{k}{d}}\left(\diff{k}{d} f_d \right)= f_k.
\end{equation*}
\end{lemma}
\begin{proof}
    For $1\le j \le k$, we compare the $j$-th coefficients of both polynomials. One can see that
    \[\coef{j}{k}(f_k)= \frac{k^j}{j!}.\]
    On the other hand, we have
    \begin{align*}
        \coef{j}{k}\left(\dil{\frac{k}{d}}\left(\diff{k}{d} f_d \right) \right)  &= \left(\frac{k}{d} \right)^j \coef{j}{k} (\diff{k}{d} f_d)\\
    &=\left(\frac{k}{d} \right)^j \coef{j}{d}(f_d) & \text{(by Lemma \ref{lem:coefficients_derivatives})}\\
    &=\left(\frac{k}{d} \right)^j \cdot \frac{d^j}{j!} = \frac{k^j}{j!}
    \end{align*}
    as desired.
\end{proof}

Using this, we infer the following formula.
\begin{proposition}
\label{prop:Thetaformula}
Let $p\in \pols_d(\rr_{\geq 0})$. Then
\[\Theta_k \left( \dil{\frac{d}{k}}\diff{k}{d}p \right) = \left(\Theta_d(p) \right)^{\boxlor \frac{d}{k}} \qquad \text{for } k=1,\dots, d.\] 
\end{proposition}
\begin{proof}
For $1\le k \le d$, we have
\begin{align*}
\Theta_d (p)^{\boxlor \frac{d}{k}} &= (\Phi_d(p\boxtimes_d f_d))^{\boxlor \frac{d}{k}} \\
&=\Phi_k (\diff{k}{d}(p\boxtimes_d f_d)) & \text{(by \eqref{eq:Phi_max_formula})}\\
&=\Phi_k ((\diff{k}{d}p )\boxtimes_k (\diff{k}{d} f_d)) & \text{(by Remark \ref{rem:diff_addtive_multiplicative})}\\
&=\Phi_k \left(\left(\dil{\frac{d}{k}} \diff{k}{d}p\right) \boxtimes_k \left( \dil{\frac{k}{d}} \diff{k}{d} f_d \right)\right)\\
&= \Phi_k \left(\left(\dil{\frac{d}{k}} \diff{k}{d}p\right) \boxtimes f_k \right) & \text{(by Lemma \ref{lem:f_k_formula})}\\
&= \Theta_k \left(\dil{\frac{d}{k}} \diff{k}{d}p \right).
\end{align*}
\end{proof}

Thanks to Proposition \ref{prop:Thetaformula}, the last claim can be seen as an approximation of its free counterpart introduced in Equation \eqref{eq:HUformula}.
\begin{corollary}
Fix $t\geq 1$ and a measure $\mu \in \MM(\rr_{\geq 0})$. Let $(d_j)_{j\in \N} \subset \N$ and $(k_j)_{j\in \N}\subset \N$ a diagonal sequence with ratio limit $1/t$. If $p_j\in \pols_{d_j}(\rr_{\geq 0})$ us a sequence of polynomials converging to $\mu$, then
\[
\meas{\Theta_k(\dil{\frac{d}{k}}\diff{k}{d}p_d)}\weak \Theta(\mu)^{\boxlor t} \qquad\text{ as }j\to\infty.
\]
\end{corollary}

\subsection*{Acknowledgements}
The authors thank Takahiro Hasebe and Jorge Garza-Vargas for fruitful discussions in relation to this project.
We thank Andrew Campbell for useful comments that help improving the presentation of the paper.

The authors gratefully acknowledge the  financial support by the grant CONAHCYT A1-S-9764 and JSPS Open Partnership Joint Research Projects grant no. JPJSBP120209921. We greatly appreciate the hospitality of Hokkaido University of Education during June 2023, where this project originated.
K.F. was supported by the Hokkaido University Ambitious Doctoral Fellowship (Information Science and AI) and JSPS Research Fellowship for Young Scientists PD (KAKENHI Grant Number 24KJ1318).
D.P. was partially supported by the AMS-Simons Travel Grant, and by Simons Foundation via Michael Anshelevich's grant.
Y.U. is supported by JSPS Grant-in-Aid for Young Scientists 22K13925 and for Scientific Research (C) 23K03133. 

\printbibliography 
\end{document}